\documentclass[10pt]{article}
\usepackage{geometry}
\usepackage{amsmath}
\usepackage{amsfonts}
\usepackage{amsthm}
\usepackage{amssymb}
\usepackage{mathtools}
\usepackage{float}
\usepackage{verbatim}
\usepackage{listings}
\usepackage{extpfeil}
\usepackage{url}
\usepackage{quiver}
\usepackage{tikz-cd}
\usepackage{tkz-euclide}
\usepackage{MnSymbol}
\usepackage{tikz}
\usepackage{enumitem}
\usepackage{comment}

\usepackage[dvipsnames]{xcolor}
\usepackage[colorlinks]{hyperref}
\hypersetup{
    pdfborder={0 0 0},
    colorlinks=true,
    linkcolor=Blue,
    filecolor=Blue,      
    urlcolor=Blue,
    citecolor=ForestGreen,
}
\newcommand*{\defeq}{\mathrel{\vcenter{\baselineskip0.5ex \lineskiplimit0pt\hbox{\scriptsize.}\hbox{\scriptsize.}}}=}
\newcommand*\circled[1]{\tikz[baseline=(char.base)]{
            \node[shape=circle,draw,inner sep=2pt] (char) {#1};}}

\newcommand*{\norm}{\operatorname{N}\mathfrak{a}}
\newcommand*{\aL}{\mathfrak{a}^{-1}L}

\newcommand*{\sigC}{\sigma_{\mathbb{C}}}
\newcommand{\Hom}{\operatorname{Hom}}
\newcommand{\SL}{\operatorname{SL}}
\newcommand{\End}{\operatorname{End}}
\newcommand{\Aut}{\operatorname{Aut}}
\newcommand{\disc}{\operatorname{disc}}

\newcommand\coker{\operatorname{coker}}
\newcommand{\triple}[1]{(#1,\langle z \rangle,g)}
\newcommand{\base}{\mathbb{P}(\Lambda\otimes\mathbb{C})-\mathbb{P}(\Lambda\otimes\mathbb{R})}

\theoremstyle{plain}
\newtheorem{theorem}{Theorem}
\newtheorem{conjecture}{Conjecture}
\newtheorem{proposition}{Proposition}
\newtheorem{corollary}{Corollary}
\newtheorem{lemma}{Lemma}

\newtheorem{introthm}{Theorem}

\theoremstyle{definition}
\newtheorem{remark}{Remark}
\newtheorem{example}{Example}
\newtheorem{claim}{Claim}
\newtheorem{definition}{Definition}

\usepackage{titling}
\title{Cocycle Relations and Elliptic Gamma Values}
\author{Teymour Gray}
\date{}

\begin{document}

\maketitle
\begin{abstract}
    Recent work of Bergeron-Charollois-Garc\'ia proposed a conjectural analytic expression for elliptic units for complex cubic fields as the special values of smoothed quotients of elliptic Gamma functions. These values depend on the conductor ideal, ray class group element, smoothing ideal, and an admissible element $h$. If this analytic expression is to fit into the framework of explicit class field theory, then, as in the case of torsion points on CM elliptic curves, the resulting complex number should depend only on the congruence class of $h$ modulo a fractional ideal. This conjecture has been supported by much numerical evidence. In this paper, we use the cocycle relations of the smoothed elliptic Gamma function to prove that the construction is indeed independent of $h$.
\end{abstract}
\tableofcontents

\section{Introduction}
Explicit class field theory, originating in Kronecker's \textit{Jugendtraum} (`dream of youth') and formulated as Hilbert's twelfth problem, seeks transcendental functions whose special values generate abelian extensions of number fields. Recent progress on this problem and its refinements has been broadly based on the construction of arithmetic cocycles, valued in meromorphic or analytic sheaves \cite{BCG23,DV21,DPV24,Kopp25,Radchenko26}. 

In the first of these papers, Bergeron-Charollois-Garc\'{\i}a conjectured that the evaluations of a smoothed quotient of elliptic Gamma functions at cubic irrationalities, denoted ${\bf\Gamma}_{\mathfrak{a},h}(L) $ for the notation explained below, are equal to the image, under a complex embedding, of units in abelian extensions of complex cubic fields.

The elliptic Gamma function, introduced by Ruijsenaars \cite{Ruijsenaars97}, is a meromorphic function on $\mathbb{C}\times\mathbb{H}\times\mathbb{H}$ given by the doubly infinite product:
\begin{equation*}
    \Gamma(z,\tau,\sigma) = \prod_{j,k\geq 0}\frac{1-q_{\tau}^{j+1}q_{\sigma}^{k+1}q_z^{-1}}{1-q_{\tau}^jq_{\sigma}^kq_z}
\end{equation*}
where $q_w = e^{2\pi i w}$. Subsequent work interpreted the elliptic Gamma function cohomologically as part of an arithmetic cocycle \cite{FV00,FHRZ08}. 

The arithmetic setting is that of a complex cubic field $F$ with complex embedding $\sigC$, fractional ideal $L$, and smoothing ideal $\mathfrak{a}$ of prime norm $N$. For this data, we consider \textit{admissible} elements $h\in L\otimes \mathbb{Q}=F$, which are specified by congruence conditions modulo $L$ and $\aL$. Then, in the simplest case, the conjectural complex cubic elliptic unit is expressed using an $N$-smoothed quotient of elliptic Gamma functions:
\begin{equation}\label{eq: smoothed elliptic gamma}
    {\bf\Gamma}_{\mathfrak{a},h}(L)^{-1} = \frac{\Gamma\left(z,\tau,\sigma\right)^N}{\Gamma\left(Nz,N\tau,N\sigma\right)}
\end{equation}
for $z$, $\tau$, $\sigma\in \sigC(F)$ cubic irrationalities determined by the input data. We state the precise arithmetic assumptions in Theorem \ref{theorem: arithmetic theorem in introduction} below and give the general definition of ${\bf\Gamma}_{\mathfrak{a},h}(L)$ after stating the Theorem.

Aside from the far deeper question of the algebraicity of these complex numbers, two more immediately tractable questions arise. Firstly, can the evaluation procedure be interpreted cohomologically? Secondly, are the resulting complex numbers independent of the choice of admissible element $h$?

We show that these questions are closely related: by formalising the cohomological consequences of the smoothing process, we prove that the construction of complex cubic elliptic units is independent of $h$. This resolves the  analytic component of a conjecture by Bergeron-Charollois-Garc\'{\i}a.
\begin{introthm}\label{theorem: arithmetic theorem in introduction}
    Let $F$ be a complex cubic field with complex embedding $\sigC$. Choose a conductor ideal $\mathfrak{f}$ such that the narrow ray class field $F^+(\mathfrak{f})$ is totally complex. Let $\varepsilon$ generate the positive units congruent to $1$ modulo $\mathfrak{f}$ and write $\chi_{\varepsilon}$ for its characteristic polynomial. Let $\mathfrak{b}$ be any integral ideal coprime to $\mathfrak{f}$ and write $L=\mathfrak{f}\mathfrak{b}^{-1}$. Finally, consider an integral smoothing ideal $\mathfrak{a}$ such that $\norm$ is prime.
    
    Suppose that $\operatorname{N}\mathfrak{a}$ is coprime to $6\operatorname{N}(\mathfrak{f})\cdot \disc(\chi_{\varepsilon})$. Then, for any admissible elements $h$ and $h'$,
    \begin{equation*}
        {\bf\Gamma}_{\mathfrak{a},h}(L) = {\bf\Gamma}_{\mathfrak{a},h'}(L).
    \end{equation*}
    In other words, for $\mathfrak{a}$ outside a finite set of primes, the complex number ${\bf\Gamma}_{\mathfrak{a},h}(L)$ is independent of $h$.
\end{introthm}
\begin{remark}
    We call $h$ admissible rather than $\lambda = hq$ for $q$ the positive generator of $\mathfrak{f}\cap\mathbb{Z}$ as in \cite{BCG23}. We recall the definition of admissibility in \eqref{eq: arithmetic admissible} below.
\end{remark}
\begin{remark}
    The requirement that $\operatorname{N}\mathfrak{a}$ is coprime to $ \disc(\chi_{\varepsilon})$ is new to this paper and we highlight its role in the proof sketch of Theorem \ref{theorem: main theorem in main body}. For now, let us note that the consequence of excluding our smoothing ideal $\mathfrak{a}$ from a finite set of primes is relatively mild. After all, the complex numbers ${\bf\Gamma}_{\mathfrak{a},h}(L)$ depend only mildly on this choice. Indeed, for two smoothing ideals $\mathfrak{a}$ and $\mathfrak{c}$, we have the easily verifiable relation:
    \begin{equation*}
        \frac{{\bf\Gamma}_{\mathfrak{a},h}(L)^{\operatorname{N}\mathfrak{c}}}{{\bf\Gamma}_{\mathfrak{a},h}(\mathfrak{c}^{-1}L)}=\frac{{\bf\Gamma}_{\mathfrak{c},h}(L)^{\operatorname{N}\mathfrak{a}}}{{\bf\Gamma}_{\mathfrak{c},h}(\mathfrak{a}^{-1}L)}.
    \end{equation*}
\end{remark}
The independence result admits a more general formulation that does not require an arithmetic setting. We work instead with a pair of rank $3$ lattices $L\subset V$ satisfying $V/L\cong\mathbb{Z}/N$, where $N$ is no longer assumed to be prime, and introduce a generalised notion of admissibility. This formulation both generalises the arithmetic theorem and provides a geometric interpretation of the independence result: the admissible data can be associated with a homology class in a homotopy quotient.

So that the generalisation from the arithmetic case is clear, we first recall the general definition of ${\bf\Gamma}_{\mathfrak{a},h}(L)$. For the same notation as used in the Theorem statement above, we define:
\begin{equation*}
    {\bf\Gamma}_{\mathfrak{a},h}(L)^{-1} = \frac{\Gamma_{a,\varepsilon a}(\sigma_{\mathbb{C}}(h),\sigma_{\mathbb{C}};L)^{\norm}}{\Gamma_{a,\varepsilon a}(\sigma_{\mathbb{C}}(h),\sigma_{\mathbb{C}};\mathfrak{a}^{-1}L)}
\end{equation*}
where $a$ is a primitive functional on $L$ uniquely determined by the triple of data $h$, $\sigC$, and $\varepsilon$, and $\Gamma_{a,\varepsilon a}$ is a product of elliptic Gamma functions; see Definition \ref{def: elliptic Gamma cochain}. The quotient above can be expressed as a product of $N$-smoothed Gamma functions; in the simplest case, the product contains one term and we recover the identity in \eqref{eq: smoothed elliptic gamma}. In this arithmetic case, $h$ is required to satisfy two conditions:
\begin{equation}\label{eq: arithmetic admissible}
    h\equiv 1 \text{ modulo } L \quad \quad \langle h \rangle_{\mathbb{Q}} \text{ generates }\aL/L.
\end{equation}
The notation for the second condition is shorthand for asking that some rational rescaling of $h$ should generate the quotient $\aL/L$. We observe that admissibility implies $\varepsilon \cdot h \equiv h$ modulo $L$ and that any two admissible elements are congruent modulo $L$; see Remark \ref{remark: admissible vector gives admissible triple} for more details.
\begin{introthm}\label{theorem: main result in introduction}
    Fix a pair of rank $3$ lattices $(V,L)$ with cyclic quotient $V/L \cong \mathbb{Z}/N$. Consider triples $h\in L\otimes\mathbb{Q}$, $\langle z\rangle \in \mathbb{P}(L^{\vee}\otimes\mathbb{C})-\mathbb{P}(L^{\vee}\otimes\mathbb{R}) $, and an automorphism $g$ of the pair $(V,L)$, which satisfy the following conditions:
    \begin{itemize}
        \item $g\cdot h\equiv h$ modulo $L$.
        \item $\langle h\rangle_{\mathbb{Q}}$ generates the quotient $V/L$.
        \item $g\cdot \langle z\rangle =\langle z\rangle$.
    \end{itemize}
    Let $\chi_g$ denote the characteristic polynomial of $g$. Suppose that $(N,6\cdot \disc (\chi_g))=1$. Then, for all triples $\triple{h}$ and $\triple{h'}$ satisfying the conditions above such that $h\equiv h'$ modulo $L$, 
    \begin{equation*}
        \widetilde{\Phi}\triple{h} =  \widetilde{\Phi}\triple{h'}
    \end{equation*}
    for the notation:
    \begin{equation*}
    \widetilde{\Phi}\triple{h}=\widetilde{\Gamma}_{a,ga}(z(h),z)=\frac{\Gamma_{a,ga}(z(h),z;L)^N}{\Gamma_{a,ga}(z(h),z;V)}
    \end{equation*}
    and the primitive functional $a\in V^{\vee} $ is determined by the triple.
\end{introthm}

The paper is organsised as follows. In Section $2$, we review the BCG conjecture \cite{BCG23} in depth, discussing the context of Stark's Conjecture, the Limit formula and elliptic units for imaginary quadratic fields. Moreover, we will define both the elliptic Gamma cocycle and a smoothed version.

The next three sections establish our main result. The key idea is to deduce the independence of $h$ from the partial vanishing of the smoothed elliptic Gamma cocycle. In Section $3$, we explain the connection between these problems by expressing the difference between the evaluations at $h$ and $h'$ in terms of a $2$-cocycle. The main part of the proof demonstrates that, under certain simplifying assumptions on the functionals $a$, $ga$, $a'$, and $ga'$, as well as $(N,6)=1$, the smoothed cocycle vanishes. We prove this in Section $4$. Under the additional assumption $(N,6\cdot \text{disc}(\chi_g))=1$, we reduce the general case to that of the vanishing cocycle in Section $5$, completing the proof of Theorem \ref{theorem: main result in introduction}.

We finish the main body of the paper with two (co)-homological contributions. We both develop a cohomological evaluation procedure and provide a homological interpretation of the input data; this offers a geometric rationale for the independence of $h$ by organising the admissible data into a single homological object.

In the first appendix, we discuss the case $N$ no longer coprime to $6$. The final appendix directly verifies, up to signs, Theorem \ref{theorem: main result in introduction} in the extremal case that our evaluation point $h$ lies inside the lattice $L$: if $\lambda$ is the strictly complex eigenvalue $g\cdot z=\lambda z$, then for all $h\in L$ such that $\langle h \rangle_{\mathbb{Q}}$ generates $V/L$, we show:
\begin{equation*}
    \widetilde{\Phi}\triple{h}^2 = \lambda^{N-1}.
\end{equation*}

\textbf{Acknowledgements.} I would like to thank my advisor Luis Garc\'ia for the many helpful discussions, insights, and, most importantly, his constant encouragement. This paper is part of on-going PhD work.
\vspace{0.1cm}

\textbf{AI Disclaimer.} All results in this paper are the author's own work, obtained without the aid of AI. 
\newpage
\section{Complex cubic elliptic units}
Since the context of complex cubic elliptic units is Stark's Conjecture, we begin with an overview of the conjecture. For a more detailed, yet very readable exposition, see the writings of Dasgupta \cite{Dasgupta99,Dasgupta_AWS}.
\subsection{Stark's Conjecture for ATRs and the BCG Conjecture}
The setting of the rank $1$ abelian Stark conjecture is a number field $F$, choice of conductor ideal $\mathfrak{f}$ and set $S$, which is the union of the archimedean places and all the finite primes dividing $\mathfrak{f}$. The set $S$ governs the order of vanishing at $s=0$ of the $\mathfrak{f}$-truncated $L$-functions associated to characters $\chi$ of the narrow ray class group $G = \operatorname{Cl}^+(\mathfrak{f})$:
\begin{equation*}
    L_{\mathfrak{f}}(\chi,s) = \prod_{\mathfrak{p}\nmid\mathfrak{f}}(1-\chi(\mathfrak{p})\operatorname{N}\mathfrak{p}^{-s})^{-1} = \sum_{(\mathfrak{a},\mathfrak{f})=1}\chi(\mathfrak{a})\operatorname{N}\mathfrak{a}^{-s}
\end{equation*}
by the following formula:
\begin{equation*}
    \text{ord}_{s=0}L_{\mathfrak{f}}(\chi,s)=\begin{cases}
        |\{v\in S: \chi(G_v)=1\}|&: \chi\neq \chi_{\text{triv}} \\ 
        |S|-1&: \chi = \chi_{\text{triv}}
    \end{cases}
\end{equation*}
where $G_v\subset \operatorname{Gal}(F^+(\mathfrak{f})/F)\cong\operatorname{Cl}^+(\mathfrak{f})$ is the decomposition group associated to the place $v$. In particular, to ensure that all $L$-functions vanish, Stark assumed that $|S|\geq 2$ and $S$ contains at least one place $v$ that splits completely in $F^+(\mathfrak{f})$. Of interest to us is the case that $F$ is an Almost Totally Real (ATR) field: $F$ has exactly one complex place, represented by a fixed complex embedding $\sigC$. In this case, our place that splits completely will be the complex place. 

We fix a complex place $\iota$ of $F^+(\mathfrak{f})$ lying above $\sigC$ and let $e_{\mathfrak{f}}$ denote the number of roots of unity in $F^+(\mathfrak{f})$. For simplicity, we assume $|S|\geq 3$. In this case, Stark conjectured the following.
\begin{conjecture}\label{conj: Stark L-functions}
    There exists a unit $u_{\mathfrak{f}}\in \mathcal{O}_{F^+(\mathfrak{f})}^*$ such that:
    \begin{equation*}
        L_{\mathfrak{f}}'(\chi,0) = -\frac{1}{e_{\mathfrak{f}}}\sum_{\sigma \in G} \chi(\sigma)\log|\iota (u_{\mathfrak{f}}^{\sigma})|^2
    \end{equation*}
    and for every place $w$ lying above a real place, we have $|u_{\mathfrak{f}}|_w=1$.
\end{conjecture}
\begin{remark}
    The simplifying assumption $|S|\geq 3$ covers all non-trivial cases for complex cubic fields. Indeed, when $S$ contains two places that split completely, the Stark Conjecture is trivially true as all the $L$-functions vanish to order at least $2$ (except for the trivial character when $|S|=2$ - this case is dealt with separately) and so we may take $u_{\mathfrak{f}}=1$. In particular, for a non-trivial example, all the real places must ramify in $F^+(\mathfrak{f})$. For a complex cubic field with unique real embedding denoted by $\sigma_{\mathbb{R}}$, this is equivalent to:
    \begin{equation*}
        \operatorname{Cl}^+(\mathfrak{f}) \neq \operatorname{Cl}(\mathfrak{f})\iff \mathcal{O}_{\mathfrak{f}}^{+,*} =\mathcal{O}_{\mathfrak{f}}^{*}
    \end{equation*}
    for the notation $ \mathcal{O}_{\mathfrak{f}}^{+,*}=\{x\in 1+\mathfrak{f}\cap \mathcal{O}_{F}^*: \sigma_{\mathbb{R}}(x)>0\} $ and $\mathcal{O}_{\mathfrak{f}}^{*} =1+\mathfrak{f}\cap \mathcal{O}_{F}^* $. In particular, $\mathfrak{f}$ cannot be $(1)$ and so to obtain a non-trivial Stark unit, we must have $|S|\geq 3$.
\end{remark}
The conjecture may equivalently be formulated in partial-zeta functions $\zeta_{\mathfrak{f}}(\sigma,s)$ since they are Fourier dual to $L_{\mathfrak{f}}(\chi,s)$. For $\sigma\in G$, they are defined:
\begin{equation*}
    \zeta_{\mathfrak{f}}(\sigma,s) = \sum_{(\mathfrak{a},\mathfrak{f})=1,\sigma_{\mathfrak{a}}= \sigma}\operatorname{N}\mathfrak{a}^{-s}.
\end{equation*}
From the Fourier inversion relation:
\begin{equation*}
     \zeta_{\mathfrak{f}}(\sigma,s) = \frac{1}{|G|}\sum_{\chi\in \widehat{G}}\chi(\sigma)L_{\mathfrak{f}}(\overline{\chi},s)
\end{equation*}
we may rewrite Conjecture \ref{conj: Stark L-functions}.
\begin{conjecture}
    There exists a unit $u_{\mathfrak{f}}\in \mathcal{O}_{F^+(\mathfrak{f})}^*$ such that:
    \begin{equation*}
          \zeta_{\mathfrak{f}}'(\sigma,0) = -\frac{1}{e_{\mathfrak{f}}}\log|\iota (u_{\mathfrak{f}}^{\sigma})|^2
    \end{equation*}
    and for every place $w$ lying above a real place, we have $|u_{\mathfrak{f}}|_w=1$.
\end{conjecture}
The conjecture only specifies the absolute value of $u_\mathfrak{f}$ at all places: $u_\mathfrak{f}$ is only well-defined up to a root of unity in $F^+(\mathfrak{f})$. Tate addressed this ambiguity by \textit{smoothing} the zeta functions \cite{Tate81,Tate84}. Precisely, we choose a smoothing prime ideal $\mathfrak{a}$ of $F$ that is coprime to $\mathfrak{f}$ and satisfies $e(\mathfrak{a}|p) < p-1$; so for imaginary quadratic and complex cubic fields, it suffices to take $\mathfrak{a}$ coprime to $6\mathfrak{f}$. This second condition ensures that $\zeta_p\notin F^+(\mathfrak{f})$: thus, all the roots of unity of $F^+(\mathfrak{f})$ are coprime to $p$ and embed into $k(\mathfrak{q})^*$ for $\mathfrak{q}|\mathfrak{a}$. In particular, $\operatorname{N}\mathfrak{a}-\sigma_{\mathfrak{a}}\in \mathbb{Z}G$ annihilates $\mu_{F^+(\mathfrak{f})}$. This motivates the definition of the $\mathfrak{a}$-smoothed partial zeta function:
\begin{equation*}
    \zeta_{\mathfrak{f},\mathfrak{a}}(\sigma,s ) = \zeta_{\mathfrak{f}}(\sigma_{\mathfrak{a}}\sigma,s)-\operatorname{N}\mathfrak{a}^{1-s}\zeta_{\mathfrak{f}}(\sigma,s).
\end{equation*}
\begin{conjecture}[Smoothed Stark's Conjecture]
    There exists a \emph{unique} unit $u_{\mathfrak{f},\mathfrak{a}}\in \mathcal{O}_{F^+(\mathfrak{f})}^*$ satisfying $u_{\mathfrak{f},\mathfrak{a}}-1\in \mathfrak{a}\mathcal{O}_{F^+(\mathfrak{f})}$ such that:
    \begin{equation*}
          \zeta_{\mathfrak{f},\mathfrak{a}}'(\sigma,0) = \log|\iota (u_{\mathfrak{f},\mathfrak{a}}^{\sigma})|^2
    \end{equation*}
    and for every place $w$ lying above a real place, we have $|u_{\mathfrak{f},\mathfrak{a}}|_w=1$. We call $u_{\mathfrak{f},\mathfrak{a}}$ \emph{the} $\mathfrak{a}$-smoothed Stark unit.
\end{conjecture}
\begin{remark}
    The change in rational constant is explained by the relation $u_{\mathfrak{f}}^{e_{\mathfrak{f}}}=u_{\mathfrak{f},\mathfrak{a}}^{\operatorname{N}(\mathfrak{a})-\sigma_\mathfrak{a}}$. We note that Stark's conjecture predicts that $u_{\mathfrak{f},\mathfrak{a}}^{\operatorname{N}(\mathfrak{a})-\sigma_\mathfrak{a}}$ is an $e_{\mathfrak{f}}$-th power in $F^+(\mathfrak{f})$ \cite[$\S4$ Lemma $6$]{Stark80}. The congruence condition ensures taking $e_{\mathfrak{f}}$-roots preserves uniqueness.
\end{remark}
A goal of explicit class field theory is to find an analytic function whose evaluation at specific `torsion' points is precisely the smoothed Stark unit. A necessary first step is to prove a limit formula relating the derivative of the smoothed partial zeta function at $s=0$ with the logarithm of the absolute value squared of this analytic function. In the imaginary quadratic case, this follows from the classical Kronecker limit formula, which expresses the derivative at $s=0$ in terms of a smoothed quotient of theta functions evaluated at a primitive $\mathfrak{f}$-torsion point on an elliptic curve; see Theorem \ref{theorem: KLF} for an explicit statement.

The generalisation of this limit formula to the complex cubic case was recently proven \cite{BCG23}[Theorem $3.2$]. In this case, the analogue of the primitive $\mathfrak{f}$-torsion point is the admissible element $h\in \mathfrak{f}^{-1}L$.
\begin{theorem}[Limit formula]\label{theorem: limit formula}
    Let $F$ be a complex cubic field with fixed complex embedding $\sigC$, $\mathfrak{f}$ a proper integral ideal and $\mathfrak{a}$ a prime ideal coprime to $6\operatorname{N}\mathfrak{f}$ and of prime norm. For $\mathfrak{b}$ integral ideal coprime to $\mathfrak{f}$, we let $L=\mathfrak{f}\mathfrak{b}^{-1}$. If we assume $\mathcal{O}_{\mathfrak{f}}^{+,*}=\mathcal{O}_{\mathfrak{f}}^*$, then the leading term of $\zeta_{\mathfrak{f},\mathfrak{a}}'(\sigma_{\mathfrak{b}},s)$ at $s=0$ is given by:
    \begin{equation*}
        \zeta_{\mathfrak{f},\mathfrak{a}}'(\sigma_{\mathfrak{b}},0)=\log|{\bf\Gamma}_{\mathfrak{a},h}(L)|^2
    \end{equation*}
    where $h\in L\otimes\mathbb{Q}$ is an admissible element: that is $h\equiv 1$ modulo $L$ and there exists a rational rescaling of $h$ which generates $\aL/L$.
    Explicitly, the analytic expression is a smoothed quotient of elliptic Gamma functions:
    \begin{equation*}
        {\bf\Gamma}_{\mathfrak{a},h}(L)^{-1} = \frac{\Gamma_{a,\varepsilon a}(\sigC(h),\sigC;L)^{\operatorname{N}\mathfrak{a}}}{\Gamma_{a,\varepsilon a}(\sigC(h),\sigC;V)}
    \end{equation*}
    where $\varepsilon$ is the unique generator of $\mathcal{O}_{\mathfrak{f}}^{+,*}$ that satisfies $\sigma_{\mathbb{R}}(\varepsilon)\in (0,1)$ for the unique real embedding and $a$ is a primitive functional on $L$ uniquely determined by the data $h$, $\sigC$ and $\varepsilon$.
\end{theorem}
From Stark's conjecture, the following is expected of the complex number ${\bf\Gamma}_{\mathfrak{a},h}(L)$.
\begin{conjecture}[BCG conjecture] 
The complex number ${\bf\Gamma}_{\mathfrak{a},h}(L) = \iota (u_{L,\mathfrak{a}})$ for $\iota : F^+(\mathfrak{f})\hookrightarrow\mathbb{C}$ a complex embedding of the narrow ray class field extending $\sigC$. Then:
\begin{itemize}
    \item The preimage is a unit, $\textbf{u}_{L,\mathfrak{a}}\in \mathcal{O}_{F^+(\mathfrak{f})}^*$.
    \item For every complex place $w$ extending the real embedding of $F$, we have $|\textbf{u}_{L,\mathfrak{a}}|_w=1$.
    \item Reciprocity law: for $\mathfrak{c}$ a fractional ideal coprime to $\mathfrak{f}$,  $\textbf{u}_{L,\mathfrak{a}}^{\sigma_{\mathfrak{c}}}= \textbf{u}_{\mathfrak{c}^{-1}L,\mathfrak{a}}$.
\end{itemize}
Moreover, the complex number ${\bf\Gamma}_{\mathfrak{a},h}(L)$ is independent of $h$.
\end{conjecture}
Our theorem settles the final part of the conjecture, the \textit{analytic} component relating to the behaviour of smoothed elliptic Gamma as a meromorphic function.
\begin{example}
    We now give an example of this phenomenon. Consider the complex cubic field $F=\mathbb{Q}(\beta)$ for $\beta^3-\beta^2+3\beta-4=0$, viewed as a subfield of $\mathbb{C}$ by specifying $\beta\in \mathbb{H}$. Above the rational primes $p\in\{2,3,7\}$, we have a unique prime ideal of norm $p$, which we denote $\mathfrak{p}_p$. We choose our conductor ideal to be $\mathfrak{f}=\mathfrak{p}_3$; this has narrow ray class group $\operatorname{Cl}^+(\mathfrak{f})\cong\mathbb{Z}/4\mathbb{Z}$, generated by $\mathfrak{b}= \mathfrak{p}_2$. Finally, take as our smoothing ideal $\mathfrak{a}=\mathfrak{p}_7$.

    For $L=\mathfrak{f}$, the following are two admissible elements:
    \begin{equation*}
        h = \beta^2 + 3\beta + 15 \quad \quad h' = 11\beta^2 + 5\beta + 46.
    \end{equation*}
    Write $\Gamma^N(z,\tau,\sigma) = \frac{\Gamma(z,\tau,\sigma)^N}{\Gamma(Nz,N\tau,N\sigma)}$ for the $N$-smoothed Gamma function. Here $N=\operatorname{N}\mathfrak{a}=7$. Then:
    \begin{align*}
        {\bf\Gamma}_{\mathfrak{a},h}(L)^{-1}&= \Gamma^N\left(\frac{1}{3},\frac{-\beta-52}{168},\frac{-\beta^2+16}{168}\right) \approx -4.336094 -1.464725i
        \\ {\bf\Gamma}_{\mathfrak{a},h'}(L)^{-1}&= \Gamma^N\left(\frac{1}{3},\frac{-\beta-10}{189},\frac{-\beta^2-89}{189}\right)\cdot \Gamma^N\left(\frac{1}{3}+\frac{-\beta^2-\beta-99}{378},\frac{-\beta-10}{189},\frac{-\beta^2-89}{189}\right)\approx -4.336094 -1.464725i
    \end{align*}
   which numerically verifies the equality ${\bf\Gamma}_{\mathfrak{a},h}(L)={\bf\Gamma}_{\mathfrak{a},h'}(L)$. 
   
   As an aside, recall that this complex number (conjecturally) equals $\iota(u_{\mathfrak{f},\mathfrak{a}})$ for a complex embedding $\iota:F^+(\mathfrak{f})\hookrightarrow\mathbb{C}$. To compute this Stark unit precisely, we use the (conjectured) reciprocity law: that is to say, the algebraic conjugates of $u_{\mathfrak{f},\mathfrak{a}}$ are given by ${\bf\Gamma}_{\mathfrak{a},h}(L)$ for $L=\mathfrak{f}\mathfrak{b}^{-j}$, $j\in \{0,1,2,3\}$. Then, using the \texttt{algdep} functionality of Pari-GP, we deduce that $u_{\mathfrak{f},\mathfrak{a}}$ is a root of the following polynomial:
    \begin{equation*}
        f(x)=x^4 - (1-\beta)\cdot x^3+(6\beta^2 - 2\beta - 5)\cdot x^2-(1-\beta)\cdot x+1\in \mathcal{O}_{F}[x].
    \end{equation*}
\end{example}
\begin{remark}
    Since the Stark conjecture is trivial in the case that $S$ contains at least two places that split completely in $F^+(\mathfrak{f})$, the only other class of cases is when $F$ is totally real. In this case, there is now a choice for which place splits completely, namely whether it is archimedean or finite. Progress in the totally real case when the archimedean place splits dates back to the work of Shintani \cite{Shintani76}; for recent work in the real quadratic field case, see the work of Kopp \cite{Kopp25} and breakthrough paper of Radchenko-Wheeler \cite{Radchenko26}. The finite prime case has recently been proven in the groundbreaking work of Dasgupta-Kakde \cite{DK_Brumer,DK_Z}; see also \cite{DK24} for the implications for the Gross-Stark conjecture and explicit class field theory for totally real fields. To appreciate the distinct flavours of the archimedean and finite place formulations, consider the case of $\mathbb{Q}$: in the archimedean setting, this is the theory of cyclotomic units; in the finite setting, Stark's conjecture follows from Stickelberger's Theorem.
\end{remark}
\subsection{The tower of multiple elliptic Gamma functions}
As seen in the example above, complex cubic elliptic units ${\bf\Gamma}_{\mathfrak{a},h}(L)$ can be expressed using (a product of) $N$-smoothed elliptic Gamma functions. Writing $q_w = e^{2\pi i w}$, we recall the elliptic Gamma function is defined on 
$\mathbb{C}\times\mathbb{H}\times\mathbb{H}$ by the doubly infinite product:
\begin{equation*}
    \Gamma(z,\tau,\sigma) = \prod_{j,k\geq 0}\frac{1-q_{\tau}^{j+1}q_{\sigma}^{k+1}q_z^{-1}}{1-q_{\tau}^jq_{\sigma}^kq_z}
\end{equation*}
and extended to $\mathbb{C}\times(\mathbb{C}-\mathbb{R})\times (\mathbb{C}-\mathbb{R})$ by:
\begin{equation*}
    \Gamma(z,\tau,\sigma) = \Gamma(z-\tau,-\tau,\sigma)^{-1}=\Gamma(z-\sigma,\tau,-\sigma)^{-1}.
\end{equation*}
Together with the normalising condition $\Gamma(\frac{1+\tau+\sigma}{2},\tau,\sigma)=1$, the elliptic Gamma function is characterised by the following difference equations \cite{FV00}[Theorem 1.1]:
\begin{equation}\label{eq: difference relation}
    \Gamma(z+\sigma,\tau,\sigma) = \theta_0(z,\tau)\cdot \Gamma(z,\tau,\sigma) \quad \quad \Gamma(z+\tau,\tau,\sigma) = \theta_0(z,\sigma)\cdot \Gamma(z,\tau,\sigma)
\end{equation}
for the holomorphic theta function on $\mathbb{C}\times\mathbb{H}$ given by:
\begin{equation*}
    \theta_0(z,\tau) = \prod_{j\geq 0}(1-q_{\tau}^{j+1}q_z^{-1})(1-q_{\tau}^jq_z).
\end{equation*}
This theta function satisfies modularity properties under $\mathbb{Z}^2\rtimes\SL_2(\mathbb{Z})$ which ensure it descends to a section of a line bundle on the universal elliptic curve over the moduli stack of elliptic curves. The process of $N$-smoothing yields a function on the universal elliptic curve over $Y_0(N)$, whose evaluation at torsion above CM points is classically known to yield elliptic units, integral elements that lie in abelian extensions of imaginary quadratic fields of specified conductor. 

To illustrate this phenomenon, consider the imaginary quadratic field $K=\mathbb{Q}(i)$ and conductor ideal $\mathfrak{f}=(3)$. Then, for a smoothing ideal $\mathfrak{a}=(2+i)$ of norm $N=5$, we define the auxiliary data of a CM point $\tau = \frac{1+i}{4+3i}\in K $ and $z = \frac{1}{3}$, a primitive $\mathfrak{f} = (3)$-torsion point on $\mathbb{C}/\mathbb{Z}\oplus \mathbb{Z}\tau$; these choices are explained in Example \ref{example: elliptic units}. Then,
\begin{equation}\label{eq: classical elliptic units}
    \theta_{\mathfrak{a}}(\mathfrak{f}) \defeq \frac{\theta_0(z,\tau)^N}{\theta_0(Nz,N\tau)}\approx 0.232051+ 0.866025i
\end{equation}
lies in the ray class field $K(\mathfrak{f})$, an abelian extension of $K$ ramified only above $\mathfrak{f}$. As in the complex cubic case, the torsion point is the congruence class $1$ modulo $\mathfrak{f}$. A representative of this class is suppressed from the notation $\theta_{\mathfrak{a}}(\mathfrak{f})$ as the complex number is \textit{independent of the choice of representative}; see Corollary \ref{corollary: independence of h for IQF}.

Thus, the analytic functions resolving Stark's Conjecture in the lowest-degree examples of ATRs are the theta function $\theta_0(z,\tau)$ and elliptic Gamma function $\Gamma(z,\tau,\sigma)$. These functions appear as the first two in a tower of multiple elliptic Gamma functions $\Gamma_n$, introduced by \cite{Nishizawa01}, in the $n+2$ variables $z\in \mathbb{C}$, $\tau_0,\dots,\tau_n\in \mathbb{C}-\mathbb{R}$. They fit together in a tower via the difference relations:
\begin{equation}\label{eq: general difference}
    \Gamma_n(z+\tau_i,\tau_0,\dots,\tau_n) = \Gamma_{n-1}(z,\tau_0,\dots,\hat{\tau_i},\dots,\tau_n)\cdot \Gamma_n(z,\tau_0,\dots,\tau_n)
\end{equation}
which precisely generalise the ones given in \eqref{eq: difference relation}. Computational evidence that smoothed quotients of these functions evaluated at certain `torsion' points compute smoothed Stark units for higher ATRs has recently been provided \cite{Morain24}. 

Narukawa \cite{Narukawa04} proved the following functional identities for these multiple elliptic Gamma functions, expressed in homogeneous coordinates:
\begin{equation}\label{eq: cocycle identity}
    \prod_{j=1}^{n}\Gamma_{n-2}\left(\frac{w}{x_j},\left(\frac{x_i}{x_j}\right)_{i\neq j}\right) = \exp\left(-\frac{2\pi i }{n!}B_{n,n}(w,\underline{x})\right)
\end{equation}
for $w\in \mathbb{C}$ and $x_1,\dots,x_n\in \mathbb{C}^*$ such that $\frac{x_i}{x_j}\in \mathbb{C}-\mathbb{R}$ for all $i\neq j$. Here, $B_{n,m}(w,\underline{x})$ is the multiple Bernoulli polynomial: a degree $m$ polynomial, homogeneous in the $n+1$ variables $w\in \mathbb{C}$, $x_1,\dots,x_n\in \mathbb{C}^*$. It is defined by the following exponential generating function:
\begin{equation}\label{eq: exp generating function}
    e^{wt}\prod_{j=1}^{n}\frac{t}{e^{x_jt}-1} = \sum_{m=0}^{\infty}B_{n,m}(w,\underline{x})\frac{t^m}{m!}.
\end{equation}
The analogue of the difference relation \eqref{eq: general difference} is the following identity:
\begin{equation}\label{eq: phase relation}
    \frac{1}{n}(B_{n,n}(w+x_i,\underline{x})-B_{n,n}(w,\underline{x})) = B_{n-1,n-1}(w,x_1,\dots,\hat{x_i},\dots,x_n)
\end{equation}
which is readily deduced from the generating function. Indeed, this is how one proves the identity in \eqref{eq: cocycle identity}: combining \eqref{eq: general difference} and \eqref{eq: phase relation}, one shows the quotient of left by right hand side in \eqref{eq: cocycle identity} is, restricted to $\mathbb{C}$, an $n$-periodic meromorphic function and so constant. For $n\geq 3$, this is immediate, whilst for $n=2$ and $\theta_0$, we must additionally verify the quotient is holomorphic; together with double periodicity, this implies the quotient is constant. Finally, to determine the constant we evaluate at the central value $\frac{1}{2}(x_1+\cdots+x_n)$. 

The final identities satisfied by the tower of functions that we wish to highlight are the so-called distribution relations.
\begin{lemma}[Distribution Relations]
    For a positive integer $N$, we have the following functional identities:
    \begin{align}
        \prod_{j=0}^{N-1}\Gamma_n\left(z+\frac{j}{N},\tau_0,\dots,\tau_n\right) &= \Gamma_n(Nz,N\tau_0,\dots,N\tau_n) \label{eq: first DR}
        \\ \prod_{j=0}^{N-1}\Gamma_n(z+j\tau_i,\tau_0,\dots,N\tau_i,\dots,\tau_n) &= \Gamma_n(z,\tau_0,\dots,\tau_i,\dots,\tau_n) \label{eq: second DR}
    \end{align}
    Combining these two, yields the pushforward relation:
    \begin{equation}\label{eq: general pushforward}
           \prod_{j,j_0,\dots, j_n=0}^{N-1}\Gamma\left(z+\frac{j+j_0\tau_0+\cdots+j_n\tau_n}{N},\tau_0,\dots,\tau_n\right)=\Gamma(Nz,\tau_0,\dots,\tau_n).
    \end{equation}
\end{lemma}
\subsection{Elliptic Gamma Cocycle}
For $\Gamma_0=\theta_0$ and $\Gamma_1=\Gamma$, the use of homogeneous coordinates, as used in identity \eqref{eq: cocycle identity}, has been given a geometric interpretation in \cite{FHRZ08}, for which the identities above become cocycle relations in a group-equivariant \v{C}ech double complex. We explain these in reverse, starting with the geometric cocycle for elliptic Gamma.

We begin by fixing $L$ to be a free abelian group of rank $3$, with an orientation fixed by an isomorphism $\bigwedge^3L\xrightarrow{\sim}\mathbb{Z}$, and write $\Lambda = \Hom(L,\mathbb{Z})$ for its $\mathbb{Z}$-dual. The complex manifold we consider is the total space of the dual tautological bundle $\mathcal{O}(1)\rightarrow \mathbb{P}(\Lambda\otimes\mathbb{C}) -\mathbb{P}(\Lambda\otimes\mathbb{R})$, which we denote $X_3$. This carries a natural action of $G_3\defeq \Lambda^{\vee}\rtimes \SL(\Lambda)$. Explicitly, $\mathcal{O}(1)$ is a quotient of the trivial bundle with fiber $\Lambda^{\vee}\otimes\mathbb{C}$ on which $\SL(\Lambda)$ acts diagonally and $\Lambda^{\vee}$ by (inverse) fiberwise translation. Using the canonical identification $\Lambda^{\vee}\cong L$, this is simply:
\begin{equation}\label{eq: non-homo action}
    (l,g)\cdot (u,\langle z\rangle) = (gu - l,\langle gz\rangle ).
\end{equation}
This natural action descends to the quotient $\mathcal{O}(1)$. As in \cite{FHRZ08}, we will work with \textit{homogeneous} coordinates on trivialisations on the bundle, which we now describe. First note that a trivialisation of $\mathcal{O}(1)|_{V}$ is an element $v\in L\otimes\mathbb{C}$ such that $z(v)\neq 0 $ for all $z\in V$; indeed, the fiber above $\langle z\rangle $ is precisely $\langle z\rangle^{\vee} = \Lambda^{\vee}\otimes\mathbb{C}/\operatorname{Ann}(\langle z\rangle)\cong L\otimes\mathbb{C}/\ker( z)$. To obtain homogeneous coordinates, we take as our trivialising section the normalised vector $\frac{v}{z(v)}$ so that:
\begin{equation*}
    \mathcal{O}(1)|_V\cong \mathbb{C}\times V\xrightarrow{\sim} (\mathbb{C}\times U)/\mathbb{C}^*\quad \quad (u,\langle z \rangle )\longmapsto (z(u),z)
\end{equation*}
for $U = \pi^{-1}(V)$, $\pi:\Lambda\otimes\mathbb{C}-0\rightarrow \mathbb{P}(\Lambda\otimes\mathbb{C})$, and a diagonal $\mathbb{C}^*$-action on the right hand side. 
\begin{lemma}
    In homogeneous coordinates, the action of $G_3$ is given by:
    \begin{equation}\label{eq: homogeneous action}
        (l,g)\cdot (w,z) = (w-gz(l),gz).
    \end{equation}
\end{lemma}
\begin{proof}
    If $\frac{v}{z(v)}$ trivialises $\mathcal{O}(1)$ over $U$, then $\frac{gv}{gz(gv)}=\frac{gv}{z(v)}$ trivialises $\mathcal{O}(1)$ over $V$. If we let $w=z(u)\in \mathbb{C}$, then with respect to these trivialisations, the action becomes:
   \[\begin{tikzcd}
	{(u,\langle z\rangle )} && {(gu - l,\langle gz\rangle )} \\
	{(w,z)} && \begin{array}{c} (gz(gu)-gz(l),gz)=(w-gz(l),gz) \end{array}
	\arrow[maps to, from=1-1, to=1-3]
	\arrow[tail reversed, from=1-1, to=2-1]
	\arrow[tail reversed, from=1-3, to=2-3]
	\arrow[maps to, from=2-1, to=2-3]
\end{tikzcd}\]
\end{proof}
To compute $G_3$-equivariant cohomology of $X_3$, we require a $G_3$-invariant open cover. This will be indexed by primitive elements in $\Lambda$.
\begin{definition}\label{def: open cover}
    Our open cover of $X_3$ is $\mathcal{U}=\{V_a\}_{a\in \Lambda_{\text{prim}}}$ for $V_a = (\mathbb{C}\times U_a)/\mathbb{C}^*$ and $U_a$ defined by:
    \begin{equation*}
        U_a = \{z \in \Lambda\otimes\mathbb{C}-\mathbb{C}\cdot \Lambda\otimes\mathbb{R}: z(\lambda)/z(\mu)\in \mathbb{H} \text{ for }\lambda\wedge\mu \text{ oriented basis of }\ker(a)\otimes\mathbb{Q}\}
    \end{equation*}
    The orientation on $\ker(a)$ is that induced by $v\wedge -:\bigwedge^2\ker(a)\xrightarrow{\sim}\bigwedge^3L\cong \mathbb{Z}$ for any $v$ satisfying $a(v)>0$. 
\end{definition}
We note that $\pi(U_a)$ defines a trivialising open for $\mathcal{O}(1)$ as $\langle z\rangle  \in \pi(U_a) $ if and only if $ z:\ker(a)\otimes\mathbb{R}\xrightarrow{\sim}\mathbb{C}$; in particular, any element of $\ker(a)$ trivialises $\mathcal{O}(1)|_{\pi(U_a)}$ and so $V_a$ is an open subset of $X_3$. Moreover, it is easy to check that $U_{ga}=g\cdot U_{a}$ for $g\in\SL(\Lambda)$ and so $\mathcal{U}$ is a $G_3$-invariant open cover.
\begin{remark}
     Notice that $V_a$ is biholomorphic to $ \mathbb{C}\times \pi(U_a)$ and $\pi(U_a)$ is biholomorphic to $\mathbb{C}\times\mathbb{H}$. In particular, the opens $V_a$ are contractible domains of holomorphy,
\end{remark}
\begin{remark}
    We give a geometric interpretation of the open $V_a$. Recall that primitive functionals are in bijection with oriented planes in $L\otimes\mathbb{Q}$, via the correspondence $a\mapsto \ker(a)$ and the orientation described in Definition \ref{def: open cover}.
    
    We now claim that any oriented plane $P$ in $L\otimes\mathbb{R}$ determines a contractible domain of holomorphy in $\base$. To see this, recognise that $\base$ is fibered over the Grassmannian of oriented planes $\operatorname{Gr}_2^+(\Lambda\otimes\mathbb{R})\cong S^2$ with fibers biholomorphic to $\mathbb{H}$. An oriented plane $P$ in $ L\otimes\mathbb{R}$ determines a loop $S^1$ inside $\operatorname{Gr}_2^+(\Lambda\otimes\mathbb{R})$ by considering lines in a complementary plane to $(L/P)^{\vee}\otimes\mathbb{R}\cong \mathbb{R}$. The orientation allows us to single out one of the two discs in $S^2-S^1$. The corresponding open of $\base$ obtained by taking the preimage is clearly contractible and a domain of holomorphy. This open consists of those functionals that restrict to an isomorphism $P\xrightarrow{\sim}\mathbb{C}$ and further satisfy an orientation condition.
\end{remark}
We are now ready to introduce the double complex for a $G_3$-equivariant sheaf on $X_3$, such as $\mathcal{O}^*$ or $\mathcal{M}^*$.
\begin{definition}
    Let $\mathcal{F}$ be a $G_3$-equivariant sheaf of abelian groups on $X_3$. For a $G_3$-invariant open cover $\mathcal{U}$, we form the double complex:
    \begin{equation*}
        C_{G_3}^{p,q}(\mathcal{U},\mathcal{F}) = C^p(G_3,\check{C}^q(\mathcal{U},\mathcal{F})) \quad \quad C^n_{G_3}(\mathcal{U},\mathcal{F}) = \oplus_{p+q=n}C_{G_3}^{p,q}(\mathcal{U},\mathcal{F})
    \end{equation*}
    with vertical and horizontal differentials given by the group and \v{C}ech differentials, denoted $d$ and $\partial$, respectively. The cohomology of the total complex $(C^\bullet_{G_3}(\mathcal{U},\mathcal{F}),D)$ maps to the $G_3$-equivariant cohomology $H^n_{G_3}(X_3,F)$.
\end{definition}
\begin{remark}
    The cover $\mathcal{U}$ is a Leray cover for equivariant constant and coherent sheaves. Thus, it is also a Leray cover for $\mathcal{O}^*$, but not necessarily for $\mathcal{M}^*$.
\end{remark}
To define the elliptic Gamma cochain, we first introduce the notion of a framing, as well as that of modulus and direction vector, as in \cite{FHRZ08}.
\begin{definition}[Framing]
    A framing $f$ is an assignment to each primitive functional $a\in \Lambda$ an oriented basis $f(a)=(x_1,x_2,x_3)$ such that $a(x_1)=1$ and $x_2$, $x_3\in \ker(a)$. 
\end{definition}
The condition that $a(x_1)>0$ implies that $(x_2, x_3)$ must be an oriented basis of $\ker(a)$. 
\begin{definition}[Direction vector \& modulus]\label{def: direction vector and modulus}
    Given the ordered pair $(a,b)$ of primitive functionals, we have $\det(a\wedge b\wedge -) = s\gamma \in \Lambda^{\vee}\cong L$, for $s$ a positive integer and $\gamma$ a primitive vector. We define $s$ to be the modulus and $\gamma$ to be the direction vector of the pair $(a,b)$. We will often use the notation $\gamma = (a,b)_L$.
\end{definition}
Equivalently, $s$ is the size of torsion in $\Lambda/\langle a,b\rangle$. Moreover, since $\ker(a,b)$ carries a canonical orientation, we can describe $\gamma$ as the unique primitive positive vector in $\ker(a,b)$.
\begin{definition}[Elliptic Gamma cochain]\label{def: elliptic Gamma cochain}
    The elliptic Gamma cochain is the following $G_3$-equivariant $1$-cochain valued in $\mathcal{M}^*$:
    \begin{equation*}
        (\Delta^f,\Gamma)\in C^1_{G_3}(\mathcal{U},\mathcal{M}^*) = C^1(G_3,\check{C}^0(\mathcal{U},\mathcal{M}^*)) \oplus C^0(G_3,\check{C}^1(\mathcal{U},\mathcal{M}^*)) 
    \end{equation*}
    The $(1,0)$-cochain $(\Delta^f_a:G_3\rightarrow\mathcal{M}^*(V_a))$ is defined by restriction to the translation subgroup $L$ as:
    \begin{equation*}
        \Delta^f_a((u,g))(w,z) =\Delta^f_a(u)(w,z) = \prod_{k=0}^{a(u)-1}\theta_0\left(\frac{w+kz(x_1)}{z(x_3)},\frac{z(x_2)}{z(x_3)}\right)
    \end{equation*}
    in the case that $a(u)>0$; see \cite{FHRZ08}[Definition 3.12] for the full generality.
    
    The $(0,1)$-cochain $(\Gamma_{a,b})\in \prod \mathcal{M}^*(V_a\cap V_b)$ is defined as a product of standard elliptic Gamma:
    \begin{equation*}
        \Gamma_{a,b}(w,z) = \prod_{\delta\in F_{a,b}^{\alpha,\beta}}\Gamma\left(\frac{w+z(\delta)}{z(\gamma)},\frac{z(\alpha)}{z(\gamma)},\frac{z(\beta)}{z(\gamma)}\right)
    \end{equation*}
    for the additional choices of $\alpha$ and $\beta$ satisfying $a(\alpha)$, $b(\beta)>0$ and $a(\beta)=b(\alpha)=0$, and index set the following  natural representatives for $L/\langle \alpha,\beta,\gamma\rangle $:
    \begin{equation*}
        F_{a,b}^{\alpha,\beta} = \{\delta\in L/\mathbb{Z}\gamma: 0\leq a(\delta)< a(\alpha), 0\leq b(\delta)< b(\beta)\}.
    \end{equation*}
\end{definition}
Before verifying this construction is independent of the choices of $\alpha$ and $\beta$, we make some preliminary remarks.
\begin{remark}
    The functions are homogeneous in $(w,z)$ and so are well-defined on $V_a$ and $V_a\cap V_b$ respectively. Indeed, $\left(\frac{w}{z(\gamma)},\frac{z(\alpha)}{z(\gamma)},\frac{z(\beta)}{z(\gamma)}\right)$ is a representative for the $\mathbb{C}^*$-equivalence class $[(w,(z(\alpha),z(\beta),z(\gamma)))]$.
\end{remark}
\begin{remark}
    We address the dependence of the cochain $(\Delta^f,\Gamma)$ on the choice of framing $f$ in Section \ref{sec: framing}.
\end{remark}
\begin{remark}
    In \cite{FHRZ08}[\S 4.3], the $(1,0)$-cochain is defined via $\Delta^{f}_{a}((u,g))(w,z) = \Delta^{f}_{a}(u\circ g^{-1})(w,z)=\Delta^{f}_{a}(gu)(w,z)$. The issue with this definition is that the $(1,1)$ component of $D(\Delta^f,\Gamma)$, namely $\frac{d\Gamma}{\partial \Delta}$, will no longer be $\mathcal{O}^*$-valued, contradicting \cite{FHRZ08}[Theorem 4.1]. However, this is not too alarming, as the authors appear to use the `correct' definition via restriction in \cite{FHRZ08}[\S 5.3], suggesting this misstep was a typo.
\end{remark}
\begin{remark}
    Since $(\alpha,\beta,\gamma)$ give a basis of $L\otimes\mathbb{C}$, we can explicitly describe the open set:
    \begin{equation*}
        \pi(U_a\cap U_b) = \left\{(z(\alpha):z(\beta):z(\gamma))\in \mathbb{P}^2(\mathbb{C}): \frac{z(\alpha)}{z(\gamma)}\in \mathbb{L}, \frac{z(\beta)}{z(\gamma)}\in \mathbb{H}\right\}\cong \mathbb{L}\times\mathbb{H}\subset \mathbb{C}^2=\{z(\gamma)\neq 0\}
    \end{equation*}
    In particular, $V_{a}\cap V_b\cong \mathbb{C}\times\mathbb{L}\times\mathbb{H}$ inside the local trivialisation of $\mathcal{O}(1)$ given by $\gamma\in \ker(a)\cap \ker(b)$. 
\end{remark}
\begin{remark}\label{remark: alpha and beta}
    The choice of a primitive pair $(\alpha,\beta)$ is equivalent to extending $(a,b)$ to a positive $\mathbb{Q}$-basis of $\Lambda\otimes\mathbb{Q}$ by a primitive functional. Precisely:
    \begin{equation*}
       \{(\alpha,\beta)\in L_{\text{prim}}^2: a(\alpha),b(\beta)>0, a(\beta)=b(\alpha)=0\}\longleftrightarrow \{c\in \Lambda_{\text{prim}}: \det(a\wedge b\wedge c)>0\}
    \end{equation*}
    There are natural notions of minimality on both sides. We could pick $\alpha$ and $\beta$ such that $\langle \alpha,\gamma\rangle =\ker(b)$ and $\langle\beta,\gamma\rangle =\ker(a) $. These conditions, together with the positivity constraint, determine $\alpha$ and $\beta$ up to $\mathbb{Z}\gamma$. For the functional $c$, the positivity constraint is equivalent to asking that $c(\gamma)>0$. We could further ask that $c(\gamma)=1$. The notions of minimality are not equivalent: the condition $c(\gamma)=1$ is strictly stronger. In other words, $c(\gamma)=1$ implies that $(\alpha,\beta)$ are chosen minimally, but the converse is false.

    Finally, we note that when $(\alpha,\beta)$ are chosen minimally, the size of the index set is precisely the modulus of the pair, that is $|F_{a,b}^{\alpha,\beta} |=s$.
\end{remark}
\begin{lemma}
    The meromorphic function $\Gamma_{a,b}\in \mathcal{M}^*(V_a\cap V_b)$ is independent of the choice of $\alpha$ and $\beta$. 
\end{lemma}
\begin{proof}
    To highlight the role played by $\alpha$ and $\beta$, let us provisionally write $\Gamma_{a,b}^{\alpha,\beta}$. By symmetry, it suffices to focus on the dependence on $\alpha$. The conditions specify $\alpha$ up to positive rescaling and translation by $\mathbb{Z}\gamma$. Since each of the entries of $\Gamma$ are invariant under translation by $\mathbb{Z}$, one can reduce to showing $\Gamma_{a,b}^{\alpha,\beta}$ and $\Gamma_{a,b}^{N\alpha,\beta}$ agree. This then follows directly from the second distribution relation \eqref{eq: second DR}. Alternatively, one can define $\Gamma_{a,b}$ without reference to either $\alpha$ or $\beta$ by describing it as a quotient of infinite products over cones \cite{FHRZ08}[\S 3.2].
\end{proof}
The main result of \cite{FHRZ08}[Theorem $4.1$] is that, although $(\Delta^f,\Gamma)$ is not a $1$-cocycle valued in $\mathcal{M}^*$, its differential is $\mathcal{O}^*$-valued. 
\begin{theorem}\label{theorem: section of gerbe}
    Fix a framing $f$ and consider the $1$-cochain $(\Delta^f,\Gamma)\in C^1_{G_3}(\mathcal{U},\mathcal{M}^*)$. Then $D(\Delta^f,\Gamma)$ lies in $C^2_{G_3}(\mathcal{U},\mathcal{O}^*)$. In particular, it defines a cocycle $[(\Delta^f,\Gamma)]\in H^1_{G_3}(X_3,\mathcal{M}^*/\mathcal{O}^*)$ whose image under the boundary map associated to the long exact sequence associated to:
    \[\begin{tikzcd}
    	0 & {\mathcal{O}^*} & {\mathcal{M}^*} & {\mathcal{M}^*/\mathcal{O}^*} & 0
    	\arrow[from=1-1, to=1-2]
    	\arrow[from=1-2, to=1-3]
    	\arrow[from=1-3, to=1-4]
    	\arrow[from=1-4, to=1-5]
    \end{tikzcd}\]
    is $[D(\Delta^f,\Gamma)]\in H^2_{G_3}(X_3,\mathcal{O}^*)$.
\end{theorem}
\begin{remark}
    In Section \ref{sec: framing}, we prove the cohomology class in $H^1_{G_3}(X_3,\mathcal{M}^*/\mathcal{O}^*)$ is independent of $f$.
\end{remark}
The geometric interpretation of this result is discussed in depth in \cite{FHRZ08}: briefly, the $2$-cocycle defines a $G_3$-equivariant holomorphic gerbe on $X_3$ for which $(\Delta^f,\Gamma)$ is naturally a meromorphic section.

To end this section, let us briefly outline the smoothing process. For a superlattice $V\supset L$, with cyclic quotient $V/L\cong \mathbb{Z}/N$, we consider the smoothing operation:
\begin{equation*}
    \widetilde{\Delta} = \frac{\Delta(-;L)^N}{\Delta(-;V)}\quad \widetilde{\Gamma} = \frac{\Gamma(-;L)^N}{\Gamma(-;V)}.
\end{equation*}
These will define smoothed cochains for a congruence subgroup. Moreover, with some further assumptions, we will be able to control where the differential of this cochain vanishes, almost allowing us to promote it to an $\mathcal{M}^*$-valued cocycle. It is this cohomological simplification that enables us to prove the independence of $h$ conjecture. The relation between these problems is explained in Sections \ref{sec: expression} and \ref{sec: vanishing theorems}.

\subsection{Classical elliptic units and $\theta_0$}\label{sec: definition of theta_0}
Since the $\theta_0$ function will play such a prominent part in this paper, let us describe it in more detail. The analogous setup is a rank $2$ lattice $L$ and complex manifold $X_2$, the total space of $\mathcal{O}(1)$ over $\base$, where as previously $\Lambda = L^{\vee}$. Things are far simpler in this case as the base space is biholomorphic to $\mathbb{H}\sqcup \mathbb{L}$, a union of two contractible domains of holomorphy. Therefore the line bundle is globally trivial and a trivialisation is given by \textit{any} element $v\in L\otimes\mathbb{R}$: for any $z\notin \mathbb{P}(\Lambda\otimes\mathbb{R})$, it follows $z$ does not vanish on any real hyperplanes and so in the rank $2$ case, this implies $z$ has trivial kernel on $L\otimes\mathbb{R}$. 

Moreover, taking as our $G_2$-invariant cover $\mathcal{U} = \{\mathcal{O}(1)|_{\mathbb{H}},\mathcal{O}(1)|_{\mathbb{L}}\}$, the \v{C}ech part of the story is trivial. It therefore suffices to consider the case of $G_2\cong \mathbb{Z}^2\rtimes\operatorname{SL}_2(\mathbb{Z})$ acting on $\mathbb{C}\times \mathbb{H}$ by:
\begin{equation*}
    \left(\begin{matrix}
        a&b\\c&d\end{matrix}\right)\cdot(x,\tau)=\left(\frac{x}{c\tau+d},\frac{a\tau+b}{c\tau+d}\right) \quad \quad (u\hspace{0.2 cm} v) \cdot(x,\tau)=(x-u\tau -v,\tau).
\end{equation*}
\begin{remark}
    If we took the basis $L=\langle e_1,e_2\rangle$ and expressed $z = (z_1 \hspace{0.2cm} z_2)$ with respect to the dual basis, then for the trivialising section $\frac{e_2}{z_2}$, the action above is the same as described in \eqref{eq: homogeneous action} for $(x,\tau)=\left(\frac{w}{z_2},\frac{z_1}{z_2}\right)$ a representative of $\mathbb{C}^*$-equivalence class $(w,z)$.
\end{remark}
As we describe explicitly below, $\theta_0\in \mathcal{O}(\mathbb{C}\times\mathbb{H})$ satisfies $d\theta_0\in C^1(G_2,\mathcal{O}^*(\mathbb{C}\times\mathbb{H}))$. Formulated in the geometric language of \cite{FHRZ08}, this realises $\theta_0$ as a holomorphic section of a line bundle on the universal elliptic curve over the moduli stack $[\mathbb{C}\times\mathbb{H}/G_2]$. The line bundle is specified by the cohomology class $[d\theta_0]$ in $ H^1(G_2,\mathcal{O}^*(\mathbb{C}\times\mathbb{H}))$, which can be identified with the Picard group of the moduli stack.
\begin{lemma}\label{lemma: explicit cocycle j}
    We write $j\defeq d\theta_0: G_2\rightarrow\mathcal{O}^*(\mathbb{C}\times\mathbb{H})$. Then, for  
    $l = (u \hspace{0.2cm} v)\in \mathbb{Z}^2$ and $g\in \SL_2(\mathbb{Z})$, such that $g^{-1} = \big(\begin{smallmatrix}
        a&b\\c&d
    \end{smallmatrix}\big)$:
    \begin{equation*}
         j(l)(x,\tau) = \frac{\theta_0(x+u\tau+v,\tau)}{\theta_0(x,\tau)}=e^{2\pi iQ_l(x,\tau)}\quad j(g)(x,\tau) = \frac{\theta_0(\frac{x}{c\tau+d},\frac{a\tau+b}{c\tau+d})}{\theta_0(x,\tau)} = e^{\pi iP_g(x,\tau)}
    \end{equation*}
    for $Q_l(x,\tau)$, $P_g(x,\tau) \in \mathbb{Q}(\tau)[x]$ given by:
    \begin{align*}
        Q_l(x,\tau) &= -ux+\frac{1}{2}u(1+\tau-u\tau ) \\
        P_g(x,\tau) &= x^2\frac{c}{c\tau+d}+x\left(\frac{1}{c\tau+d}-1\right)-\frac{a\tau+b}{6(c\tau+d)}+\frac{\tau}{6}-\frac{1}{6}\Phi(g)
    \end{align*}
    where $\Phi:\SL_2(\mathbb{Z})\rightarrow\mathbb{Z}/12$ is a group morphism defined on generators as $\Phi\left(\big(\begin{smallmatrix}
            1&1\\0&1
        \end{smallmatrix}\big)\right) = 1$, $\Phi\left(\big(\begin{smallmatrix}
            0&1\\-1&0
        \end{smallmatrix}\big)\right) = 3$.

\end{lemma}
\begin{proof}
    The $\mathbb{Z}^2$-cocycle relation follows from $\theta_0(x+1,\tau)=\theta_0(x,\tau)$ and $\theta_0(x+\tau,\tau)=-q_x^{-1}\cdot\theta_0(x,\tau)$. As $\SL_2(\mathbb{Z})$ is generated by $\big(\begin{smallmatrix}
            1&1\\0&1
        \end{smallmatrix}\big)$ and $\big(\begin{smallmatrix}
            0&1\\-1&0
        \end{smallmatrix}\big)$, the general cocycle relation can be deduced from $\theta_0(x,\tau+1)=\theta_0(x,\tau)$ and the identity described in \eqref{eq: cocycle identity}, namely:
    \begin{equation*}
         \theta_0\left(\frac{w}{z_2},\frac{z_1}{z_2}\right)\cdot \theta_0\left(\frac{w}{z_1},\frac{z_2}{z_1}\right) = \exp(-\pi i B_{2,2}(w,z_1,z_2))
    \end{equation*}
    Notice that the left hand side is precisely $j\left(\big(\begin{smallmatrix}
            0&1\\-1&0
        \end{smallmatrix}\big)\right)(x,\tau)=\theta_0(x,\tau)/\theta_0(\frac{x}{-\tau},\frac{1}{-\tau})$, as we have $\theta_0(x,\tau) = \theta_0(-x,-\tau)^{-1}$.
        
    An alternative perspective is provided in \cite{Kato04}, in which the $\SL_2(\mathbb{Z})$-automorphy factor follows only from the explicit $\mathbb{Z}^2$-cocycle, computation of the divisor of $\theta_0$, and the pushforward relation described in \eqref{eq: general pushforward}. That $\theta_0$ should satisfy an $\SL_2(\mathbb{Z})$-modularity relation is not surprising: the $\mathbb{Z}^2$-cocycle exhibits $\theta_0$ as a section of a line bundle on the universal elliptic curve over $\mathbb{H}$ whose divisor is the zero section $[\mathcal{O}]$ and pullback by $\SL_2(\mathbb{Z})$ preserves this divisor. This perspective is discussed further in Appendix \ref{appendix: coprime assumption}; for instance, see  Remark \ref{remark: kato perspective}.
\end{proof}
Of critical importance to this paper is the effect of $N$-smoothing on the cocycle $j$ restricted to the congruence subgroup $\Gamma_0(N)$.
\begin{definition}[$N$-smoothed theta function]
    For a positive integer $N$, define the $N$-smoothed theta function:
    \begin{equation*}
        \theta_0^N(x,\tau) \defeq \frac{\theta_0(x,\tau)^N}{\theta_0(Nx,N\tau)}.
    \end{equation*}
    We write $j^N = d\theta_0^N:\mathbb{Z}^2\rtimes\Gamma_0(N)\rightarrow\mathcal{O}^*(\mathbb{C}\times\mathbb{H})$ for the smoothed cocycle.
\end{definition}
It suffices to consider $N$ positive by the relation $\theta_0^{-N}=(\theta_0^N)^{-1}$. Moreover, we restrict to $\Gamma_0(N)$ as this is precisely the congruence subgroup which preserves the divisor of $\theta_0^N$, given by $ N\cdot [\mathcal{O}]-\sum_{j=0}^{N-1}\left[\frac{j}{N}\right]$, where these divisors are viewed as sections of the universal elliptic curve over $\mathbb{H}$.
\begin{proposition}\label{prop: j^N trivial}
    For $N$ coprime to $6$, the $N$-smoothed cocycle $j^N$ is trivial. Explicitly, for $l=(u \hspace{0.2cm} v)\in \mathbb{Z}^2$ and $g^{-1}=\big(\begin{smallmatrix}a& b \\c&d\end{smallmatrix}\big)\in\Gamma_0(N)$:
    \begin{equation*}
        \theta_0^N(x+u\tau+v)=\theta_0^N(x,\tau)\quad \theta_0^N\left(\frac{x}{c\tau+d},\frac{a\tau+b}{c\tau+d}\right) = \theta_0^N(x,\tau) .
    \end{equation*}
    In other words, $\theta_0^N$ 
    is a function on the universal elliptic curve over $Y_0(N)$ for $(N,6)=1$.
\end{proposition}
\begin{proof}
    We can explicitly compute:
    \begin{equation*}
        j^N(l)(x,\tau) = \exp\left(2\pi i\cdot \frac{N-1}{2}u\right)=1\quad\quad j^N(g)(x,\tau) = \frac{j(g)(x,\tau)^N}{j(\alpha g\alpha^{-1})(Nx,N\tau)} =1
    \end{equation*}
    where $\alpha=\big(\begin{smallmatrix}
        N&0\\0&1
    \end{smallmatrix}\big)\in \operatorname{GL}_2(\mathbb{Q})$. The assumption $g\in \Gamma_0(N)$ is equivalent to $\alpha g\alpha^{-1}\in \SL_2(\mathbb{Z})$.

    In the appendix, we give a more conceptual proof using the pushforward relation, due to \cite{Kato04}, in the general case $N$ no longer coprime to $6$.
\end{proof}
We finish this section by rephrasing Kronecker's Limit Formula to mirror the BCG Limit Formula in the complex cubic case. To that end, let us introduce theta functions associated to primitive functionals on an oriented rank $2$ lattice $L$, as done in \cite{Morain24}. 
\begin{definition}
    For $a\in \Lambda$ primitive, we write $\gamma= \det(a\wedge - )$ and choose $\alpha\in L$ so that $L=\langle \alpha,\gamma\rangle$ and $a(\alpha)>0$. Then, for $U_a = \{z\in \Lambda\otimes\mathbb{C}-\Lambda\otimes\mathbb{R}:z(\alpha)/z(\gamma)\in \mathbb{H}\}$ and $(w,z)\in V_a =  (\mathbb{C}\times U_a)/\mathbb{C}^*$, we define:
    \begin{equation*}
\theta_a(w,z)=\theta_0\left(\frac{w}{z(\gamma)},\frac{z(\alpha)}{z(\gamma)}\right).
    \end{equation*}
\end{definition}
For a superlattice $V\supset L$, and a primitive functional $a\in L^{\vee}$ that extends to a $\mathbb{Z}$-valued function on $V$, we may define a smoothed theta function with respect to the pair:
\begin{equation*}
    \widetilde{\theta}_a(w,z;L,V) = \frac{\theta_a(w,z;L)^N}{\theta_a(w,z;V)}.
\end{equation*}
We will suppress the lattice pair from the notation on the left hand side when it is clear which pair is involved.
\begin{remark}
    When $V/L\cong\mathbb{Z}/N$, the geometric notion of smoothing coincides with our earlier notion of $N$-smoothing:
    \begin{equation*}
        \widetilde{\theta}_a(w,z) = \theta_0^N\left(\frac{w}{z(\gamma)},\frac{z(\alpha)}{z(\gamma)}\right).
    \end{equation*}
    This is because the assumption that $a$ extends to a $\mathbb{Z}$-valued functional on $V$ is equivalent to its kernel generating the quotient: in other words, $V/L = \langle \gamma/N\rangle$.

    If we further assume that $N$ is coprime to $6$, then by $\Gamma_0(N)$-invariance this construction is independent of $a$: if $b\in L^{\vee}$ is another primitive functional that extends to $V^{\vee}$, we have $\widetilde{\theta}_a(w,z)=\widetilde{\theta}_b(w,z)$.
\end{remark}
To continue the parallel story with the complex cubic case, we consider an imaginary quadratic field $K$ with fixed complex embedding $\sigC$, $\mathfrak{f}$ a proper integral ideal, and $\mathfrak{a}$ a prime ideal coprime to $6\mathfrak{f}$ and of prime norm $N$. For $\mathfrak{b}$ an integral ideal coprime to $\mathfrak{f}$, we set $L=\mathfrak{f}\mathfrak{b}^{-1}$ and define $h\in L\otimes\mathbb{Q}$ to be admissible if $h\equiv 1$ modulo $L$ and $\langle h\rangle_{\mathbb{Q}}$ generates $\aL/L$: $h$ defines a primitive $\mathfrak{f}$-torsion point on the elliptic curve $\mathbb{C}/L$. Then, we define:
\begin{equation*}
    \theta_{\mathfrak{a},h}(L) = \widetilde{\theta}_a(\sigC(h),\sigC;L,\aL)
\end{equation*}
for primitive $a\in \Lambda$ uniquely specified by the conditions that $a(h)=0$ and $\sigC\in U_a$. It is then immediate from our previous work that this construction is independent of $h$.
\begin{corollary}[Independence of $h$ for $\theta_{\mathfrak{a},h}(L)$]\label{corollary: independence of h for IQF}
    Let $(V,L)$ be a pair of rank $2$ lattices with cyclic quotient $V/L\cong \mathbb{Z}/N$ for $(N,6)=1$. Fix $z$ a strictly complex functional and $h$, $h'\in L\otimes\mathbb{Q}$ be two elements both of whose rational span generates $V/L$. Then, for $h\equiv h'$, we have:
    \begin{equation*}
        \widetilde{\theta}_a(z(h),z) =\widetilde{\theta}_{a'}(z(h'),z).
    \end{equation*}
\end{corollary}
\begin{proof}
    By the previous remark and the $\mathbb{Z}^2$-invariance of the $N$-smoothed $\theta_0^N$:
    \begin{equation*}
        \widetilde{\theta}_a(z(h),z) = \widetilde{\theta}_{a'}(z(h),z) =\widetilde{\theta}_{a'}(z(h'),z).
    \end{equation*}
    This proves the independence of $h$ for $\theta_{\mathfrak{a},h}(L)$ by taking $z=\sigC$ and $V=\aL$. We observe that for a prime ideal $\mathfrak{a}$, the assumption that $\aL/L$ is a cyclic $\mathbb{Z}$-module is equivalent to $\mathfrak{a}$ having prime norm.
\end{proof}
By the independence result above, the notation $\theta_{\mathfrak{a}}(L)$ is unambiguous. Using this notation, we reformulate Kronecker's Limit Formula in this context \cite{Colmez04,Kato04}.
\begin{theorem}[Kronecker's Limit Formula]\label{theorem: KLF}
    Let $K$, $\sigC$, $\mathfrak{f}$, $\mathfrak{a}$, $\mathfrak{b}$, and $h$ be as above. Write $L=\mathfrak{f}\mathfrak{b}^{-1}$ and $w_\mathfrak{f}$ for the number of roots of unity in $1+\mathfrak{f}$. Then, the leading term of the smoothed-partial zeta function is given by:
    \begin{equation*}
    \zeta_{\mathfrak{f},\mathfrak{a}}'(\sigma_{\mathfrak{b}},0)=\frac{1}{w_{\mathfrak{f}}} \log|\theta_{\mathfrak{a}}(L)|^2.
\end{equation*}
\end{theorem}
The theory of complex multiplication allows one to complete the proof of Stark's conjecture in this case \cite{Shalit87,Robert73,Rubin99}. We end with two short examples, the first explaining the choice of data in \eqref{eq: classical elliptic units}.
\begin{example}\label{example: elliptic units}
    In the example given in \eqref{eq: classical elliptic units}, we considered $K=\mathbb{Q}(i)$, $\mathfrak{f}=(3)$, and $\mathfrak{a}=(2+i)$ of norm $N=5$. Note that $w_{\mathfrak{f}}=1$. Then, for $L=\mathfrak{f}$, we may take $h=4+3i$, $\gamma = 12+9i$ and $\alpha = 3+3i$. 
    \begin{equation*}
        \theta_{\mathfrak{a}}(\mathfrak{f}) = \theta_0^N\left(\frac{1}{3},\frac{3+3i}{12+9i}\right)\approx 0.232051+ 0.866025i.
    \end{equation*}
    Since the ray class group of conductor $\mathfrak{f}$ has order $2$ and is generated by $\mathfrak{b}=(1+i)$, the conjugate is given by $\theta_{\mathfrak{a}}(\mathfrak{f}\mathfrak{b}^{-1})$. For this, we take $h=4+3i$, $\gamma = 12+9i$ and $\alpha = \frac{1}{2}(3+3i)$.
    \begin{equation*}
        \theta_{\mathfrak{a}}(\mathfrak{f})^{\sigma_{\mathfrak{b}}}=\theta_{\mathfrak{a}}(\mathfrak{f}\mathfrak{b}^{-1}) = \theta_0^N\left(\frac{1}{3},\frac{\frac{1}{2}(3+3i)}{12+9i}\right)\approx -3.232051 - 0.866025i.
    \end{equation*}
    This explains the choice of auxiliary data. These elements satisfy $x^2+3x-3i\in \mathbb{Z}[i][x]$ and so are $3$-units in the ray class field $K(\mathfrak{f})$. This concurs with the refinement of Stark's conjecture in the case $S=\{\sigC,\mathfrak{f}\}$ has size $2$.
\end{example}
\begin{example}
    As we stated Stark's Conjecture for $|S|\geq 3$, for which we expect genuine units, let us give an example in this case. So we take $K=\mathbb{Q}(\sqrt{-2})$, $\mathfrak{f} = (3)= \mathfrak{p}\cdot \overline{\mathfrak{p}} $, and $\mathfrak{a} = (3+\sqrt{-2})$ of norm $N=11$. Again, we have $w_{\mathfrak{f}}=1$. Moreover, since $(3)$ splits, $\mathfrak{f}$ is composite, $|S|=3$ and we expect our $\mathfrak{a}$-smoothed Stark unit to be a genuine unit. As $\mathfrak{b}=(\sqrt{-2})$ generates $\operatorname{Cl}(\mathfrak{f})\cong \frac{(\mathcal{O}/\mathfrak{p})^*\oplus(\mathcal{O}/\overline{\mathfrak{p}})^*}{\mathcal{O}^* \text{ mod }\mathfrak{f}}\cong\mathbb{Z}/2$, we compute:
    \begin{equation*}
        \theta_{\mathfrak{a}}(\mathfrak{f}) = \theta_0^N\left(\frac{1}{3},\frac{12+3\sqrt{-2}}{39+9\sqrt{-2}}\right)\approx -0.050510 + 0.087486i \quad          \theta_{\mathfrak{a}}(\mathfrak{f})^{\sigma_{\mathfrak{b}}} = \theta_0^N\left(\frac{1}{3},\frac{\frac{1}{2}(12+3\sqrt{-2})}{39+9\sqrt{-2}}\right)\approx-4.949490 - 8.572768i.
    \end{equation*}
    Here, we have taken $h = 13+3\sqrt{-2}$, $\gamma = 39+9\sqrt{-2}$ in both cases, and $\alpha = 12+3\sqrt{-2}$ for the former and $\frac{1}{2}(12+3\sqrt{-2})$ for the latter. We verify that these satisfy $x^2+(5+6\sqrt{-2})x+1\in \mathbb{Z}[\sqrt{-2}][x]$.
\end{example}
\subsection{Dependence on framings}\label{sec: framing}
We finish this chapter by dealing with the cohomological dependence of the elliptic Gamma cochain $(\Delta^f,\Gamma)$ on the choice of framing $f$. We thus define the following cochain $\psi^{f,f'}\in C^{1,0}_{G_3}(\mathcal{U},\mathcal{M}^*)$:
\begin{equation}\label{eq: change of framings}
    \quad \psi^{f,f'} \defeq  \frac{\Delta^{f'}}{\Delta^{f}}.
\end{equation}
\begin{lemma}\label{lemma: change of framing}
    The cochain $\psi^{f,f'}$ is $\mathcal{O}^*$-valued and so the $1$-cocycle $(\Delta^f,\Gamma)\in Z^1_{G_3}(\mathcal{U},\mathcal{M}^*/\mathcal{O}^*)$ is independent of $f$. In particular, the cohomology class $[D(\Delta^f,\Gamma)]\in H^2_{G_3}(X_3,\mathcal{O}^*)$ is independent of $f$.
\end{lemma}
\begin{proof}
    Given two framings $f$ and $f'$, we  define the following auxiliary data for each functional $a$. Writing $f(a) = (x_1,x_2,x_3)$ and $f'(a) = (y_1,y_2,y_3)$, we set $l\defeq y_1-x_1 \in \ker(a)$ and $g\in \SL(\ker(a))$ such that $g\cdot (\begin{smallmatrix}
        y_2 \\y_3
    \end{smallmatrix}) = (\begin{smallmatrix}
        x_2 \\x_3
    \end{smallmatrix})$. Then, we can express the cochain $\psi^{f,f'}$ explicitly in terms of the cocycle $j$ for $\theta_0$:
    \begin{equation*}
        \psi^{f,f'}_{a}(u)(w,z) = \prod_{k=0}^{a(u)-1} j((kl, g))\left(\frac{w+kz(x_1)}{z(x_3)},\frac{z(x_2)}{z(x_3)}\right).
    \end{equation*}
    In particular, $\psi^{f,f'}\in C^1_{G_3}(\mathcal{U},\mathcal{O}^*)$.
\end{proof}
\newpage
\section{Cocycle expression}\label{sec: expression}
In this section, we begin the proof of the main result, Theorem \ref{theorem: arithmetic theorem in introduction}, that for $h$ and $h'$ two admissible elements of $L$, and for $\mathfrak{a}$ outside some finite set of primes, we have:
\begin{equation*}
    {\bf\Gamma}_{\mathfrak{a},h}(L)={\bf\Gamma}_{\mathfrak{a},h'}(L).
\end{equation*}
Before we outline the structure of the proof, let us properly introduce the alternative notation for the (smoothed) elliptic Gamma function used in Theorem \ref{theorem: main result in introduction}. 

Throughout this section, we fix a rank $3$ lattice $L$ and framing $f$. We return to the choice of framing in the next section, where it will be appropriate to pick one satisfying certain conditions.

Mirroring the unit $\varepsilon\in \mathcal{O}_{\mathfrak{f}}^{+,\ast}$ in arithmetic setting, we below consider an element $g\in \SL(\Lambda)$. Throughout we also use $g$ to denote the induced element in $\SL(L)$ given by pullback by $g^{-1}$ on $\Lambda^{\vee}\cong L$. With this notation in mind, we record the following action of $\SL(\Lambda)$ on the set of framings. For $g\in\SL(\Lambda)$ and a framing $f$, define:
\begin{equation*}
    g\cdot f: a\longmapsto gf(g^{-1}a) = (gy_1,gy_2, gy_3)
\end{equation*}
when $f(g^{-1}a)=(y_1,y_2, y_3)$.

We restrict our attention to those $g$ which have no rational eigenvalues and a pair of strictly complex ones; equivalently, the characteristic polynomial $\chi_g\in\mathbb{Q}[x]$ is irreducible and has negative discriminant.
\begin{lemma}\label{lemma: a assocaited to (h,z,g)}
    Consider $g\in \SL(\Lambda)$ for which $\chi_g\in\mathbb{Q}[x]$ is irreducible and has negative discriminant and let $z\in\Lambda\otimes\mathbb{C}-\Lambda\otimes\mathbb{R}$ be an eigenfunctional of $g$. Then, for any nonzero element $h\in L\otimes\mathbb{Q}$, there is a unique primitive functional $a\in \Lambda$ such that:
    \begin{enumerate}[label=(\roman*)]
        \item $h\in \ker(a,ga)\otimes\mathbb{Q}$,
        \item $z\in U_a$.
    \end{enumerate}
\end{lemma}
\begin{proof}
    Since $g$ has no rational eigenvalues, the elements $h,g^{-1}h\in L\otimes \mathbb{Q}$ are linearly independent and therefore, the annihilator $\operatorname{Ann}(\langle h,g^{-1}h\rangle)$ is one-dimensional. This condition specifies the primitive functional $a$ up to a choice of sign; the sign ambiguity corresponds to the two orientations on $\ker(a)$. Since $z(h)/z(g^{-1}h)=z(h)/gz(h)\in\mathbb{C}-\mathbb{R}$, we have exactly one of $z\in U_a$ or $z\in U_{-a}$.
\end{proof}
\begin{definition}
    For $g$ and $z$ as in the Lemma above and a nonzero element $h\in L\otimes\mathbb{Q}$, we set:
    \begin{equation*}
        \Phi\triple{h} \defeq \Gamma_{a,ga}(z(h),z)
    \end{equation*}
    where the primitive functional $a$ is uniquely determined by the triple of data.
\end{definition}
\begin{lemma}
    This construction yields a well-defined non-zero complex number. That is $\Phi(h,\langle z\rangle,g)\in \mathbb{C}^*$ is independent of the choice of generator of $\langle z\rangle$. 
\end{lemma}
\begin{proof}
    Recall that our opens are defined as $V_a = (\mathbb{C}\times U_a)/\mathbb{C}^*$ where $\mathbb{C}^*$-action is diagonal. In particular, $\Gamma_{a,b}(z(u),z)= \Gamma_{a,b}(\lambda z(u),\lambda z)$ for any $\lambda\in \mathbb{C}^*$ and $u\in L\otimes \mathbb{C}$. 
    
    For the non-vanishing of $\Phi(h,\langle z\rangle,g) = \Gamma_{a,ga}(z(h),z)$, we recall that in general $\Gamma_{a,b}(z(-),z)$ is holomorphic and non-vanishing on $\ker(a,b)\otimes\mathbb{Q}$ when $z:L\otimes\mathbb{Q}\hookrightarrow\mathbb{Q}$ \cite{FHRZ08}[Proposition 3.3]. Now observe the following. Firstly, $a$ is chosen so that $h\in \ker(a,ga)\otimes\mathbb{Q}$. Secondly, for $z$ a non-rational eigenfunctional of $g$, the kernel $\ker(z:L\otimes\mathbb{Q}\rightarrow\mathbb{C})$ is a rational vector space of dimension $1$ or $0$ that carries an action of $g$; in particular, since $g$ has no rational eigenvectors, we deduce $z:L\otimes \mathbb{Q}\hookrightarrow\mathbb{C}$.
\end{proof}
Now we outline the structure of the proof. For $\triple{h}$ and $\triple{h'}$ satisfying our analogue of admissibility, to be specified shortly, and $h\equiv h'$ modulo $L$, we will show:
\begin{enumerate}
    \item The relevant quotient of \textit{unsmoothed} Gamma functions can be expressed:
    \begin{equation*}
        \frac{\Phi\triple{h}}{\Phi\triple{h'}}=\text{ expression in }D(\Delta,\Gamma)
    \end{equation*}
    in terms of the $2$-cocycle $ D(\Delta,\Gamma)\in C^2_{G_3}(\mathcal{U},\mathcal{O}^*)$. This is the content of Proposition \ref{prop: cocycle expression}.
    \item Then we consider a superlattice $V$. The computation above is immediately promoted to an expression relating the quotient of \textit{smoothed} Gamma functions in terms of the smoothed $2$-cocycle $D(\widetilde{\Delta},\widetilde{\Gamma})$.
    \begin{equation*}
        \frac{\widetilde{\Phi}\triple{h}}{\widetilde{\Phi}\triple{h'}}=\text{ expression in }D(\widetilde{\Delta},\widetilde{\Gamma})
    \end{equation*}
    Assuming that $V/L$ is cyclic of order $N$ coprime to $6$ and $d(h,h')=3$, for an invariant $d$ to be introduced, the expression on the right-hand side vanishes. This is the content of Propositions \ref{prop: smoothed compatible pairs}, \ref{prop: smoothed phi2,0}, \ref{prop: smoothed phi1,1} and \ref{prop: smoothed phi0,2}.
    \item Finally, using the hypothesis that $(N, 6\cdot \disc(\chi_g))=1$, we can reduce to the $d=3$ case. This is the content of Corollary \ref{corollary: d=3} to Proposition \ref{prop: image of Psi}.
\end{enumerate}
We begin by defining the notion of an admissible triple for $L$. The motivation for considering the triple of data is justified later when we give a homological interpretation of the triple as a loop in a homotopy quotient.
\begin{definition}[Admissible triple for $L$]
    Consider a nonzero rational vector $h\in L\otimes\mathbb{Q}$, strictly complex functional $z\in \Lambda\otimes\mathbb{C}-\Lambda\otimes\mathbb{R}$, and automorphism $g\in \SL(\Lambda)$ that has no rational eigenvalues. We call $\triple{h}$ an admissible triple for $L$ if:
    \begin{equation*}
        g\cdot \langle z \rangle =\langle z\rangle \quad\quad  g\cdot  h\equiv h \text{ modulo }L.
    \end{equation*}
\end{definition}
Observe that for $h'\in h+L$, the triple $\triple{h'}$ is also admissible for $L$.
\begin{remark}\label{remark: admissible vector gives admissible triple}
    Observe that an admissible element $h$ for the fractional ideal $L$ in the arithmetic context of \cite{BCG23} gives an admissible triple $(h,\langle \sigC\rangle,\varepsilon)$ for $L$. Indeed, $\sigC$ is an eigenvector of $\varepsilon$ and the assumption that $h\equiv 1$ modulo $L$ implies $\varepsilon h\equiv h$ modulo $L$; to see this note that $h\in \mathfrak{f}^{-1}L$ and $\varepsilon$ is assumed to lie in $1+\mathfrak{f}$. 
\end{remark}
To state the explicit expression alluded to above we need explicit notation for the components of the cocycle:
\begin{equation*}
    D(\Delta,\Gamma) = (\phi^{2,0},\phi^{1,1},\phi^{0,2})\in C^2_{G_3}(\mathcal{U},\mathcal{O}^*) \subset C^2_{G_3}(\mathcal{U},\mathcal{M}^*)
\end{equation*}
where $\phi^{p,q}\in C^{p,q}_{G_3}(\mathcal{U},\mathcal{O}^*)=C^p(G_3,\check{C}^q(\mathcal{U},\mathcal{O}^*))$. 
\begin{proposition}\label{prop: cocycle expression}
    Let $\triple{h}$ an admissible triple for $L$. Then, for all $h'\equiv h$ modulo $L$:  
    \begin{equation*}
        \frac{\Phi\triple{h}}{\Phi\triple{h'}}=\psi^{f,g \cdot f}_{b'}(v)(z(h),z)^{-1}\cdot \frac{\phi^{2,0}_a(u,v-u)(z(h),z)}{\phi^{2,0}_{a'}(u,g^{-1}v-u)(gz(h),z) }\cdot \phi^{1,1}_{a,a'}(u)(gz(h),z)\cdot \phi^{1,1}_{b',a}(v)(z(h),z)\cdot \frac{\phi^{0,2}_{a,b,b'}(z(h),z)}{\phi^{0,2}_{a,a',b'}(z(h'),z)}
    \end{equation*}
    where $a$ and $a'$ are the unique primitive functionals associated to $\triple{h}$ and $\triple{h'}$ respectively and we write $b=ga$, $b'=ga'$. Finally, $u$ and $v$ are the elements of $L$ determined by:
    \begin{equation*}
        h'=h+v \hspace{1cm} \text{and}\hspace{1cm} h' = g^{-1}h+u.
    \end{equation*}
\end{proposition}
The double complex below helps us to organise the material.

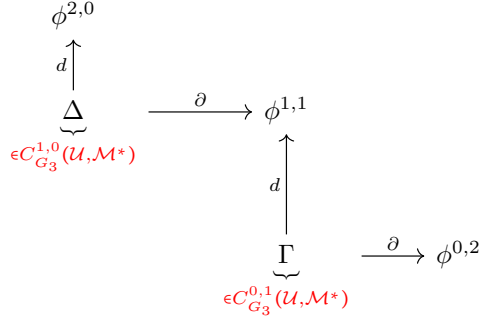
\begin{figure}[!htbp]
    \centering
   \begin{tikzcd}
        \phi^{2,0} & & \\
        \underbrace{\Delta}_{\textcolor{red}{\in C^{1,0}_{G_3}(\mathcal{U},\mathcal{M}^*) }} \arrow[u,"d"]{} \arrow[r,"\partial"] &\phi^{1,1} &\\
        & 
        \underbrace{\Gamma}_{\textcolor{red}{\in C^{0,1}_{G_3}(\mathcal{U},\mathcal{M}^*) }} \arrow[u,"d"]\arrow[r,"\partial"]  & \phi^{0,2}
    \end{tikzcd} 
    \label{figure: double complex}
    \caption{The group-equivariant \v{C}ech double complex}
\end{figure}

We will use each relation in turn, ascending from the bottom right to upper left. Our three relations are
\begin{enumerate}[label=\color{red}\protect\circled{\arabic*}]
    \item $\partial\Gamma=\phi^{0,2}\in \prod \mathcal{M}^*(V_a\cap V_b\cap V_c)$
    \item $d\Gamma=\phi^{1,1}\cdot\partial\Delta\in\Hom(G_3,\prod\mathcal{M}^*(V_a\cap V_b))$
    \item $d\Delta=\phi^{2,0}\in\Hom(G_3\times G_3,\prod\mathcal{M}^*(V_a))$
\end{enumerate}
\begin{proof}[Proof of Proposition \ref{prop: cocycle expression}]
    The proof will use the three relations coming from the double complex depicted in Figure \ref{figure: double complex} and the assumptions of the Proposition, which we compile here for later use:
    \begin{enumerate}[label=(\roman*)]
        \item $z$ is an eigenvector of $g$. \label{item (i)}
        \item $h'\equiv h$ modulo $L$ so that $h'=h+v$ for some $v\in L$.\label{item (ii)}
        \item $gh\equiv h $ modulo $L$ so that $h'=g^{-1}h+u$ for some $u$ in $L$.\label{item (iii)}
    \end{enumerate}
    We immediately deduce the following facts:
    \begin{enumerate}[resume,label=(\roman*)]
        \item $v-u =g^{-1}h-h\in \ker(a)$ since $h\in \ker(a)\cap\ker(ga)$.\label{item (iv)}
        \item $g^{-1}v-u=g^{-1}h'-h'\in \ker(a')$ since $h'\in\ker(a')\cap\ker(ga')$.\label{item (v)}
    \end{enumerate}
    Moreover, recalling the construction of our opens $V_a$ as the quotient of $\mathbb{C}\times U_a$ by a \textit{diagonal} $\mathbb{C}^*$-action, it follows from our assumption that $z$ is an eigenvector of $g$ that $(z(h),g^{-1}z) = (gz(h),z)$.
    
We now begin the proof, starting with our first cocycle relation.

\noindent\textcolor{red}{\circled{1}} Given a triple of functionals $a,b$ and $c$, we have:
\begin{equation*}
    (\partial\Gamma)_{a,b,c} = \Gamma_{a,b}\Gamma_{b,c}\Gamma_{c,a}=\phi^{0,2}_{a,b,c}.
\end{equation*}
Now consider the decomposition of the figure below into simplices:
\begin{figure}[!htbp]
    \centering
    \begin{tikzpicture}[scale = 0.8]
    \tkzDefPoint(0,0){a}  \tkzLabelPoint[below left](a){$a$}
    \tkzDefPoint(3,0){b}  \tkzLabelPoint[below right](b){$b$}
    \tkzDefSquare(a,b)    \tkzGetPoints{b'}{a'}
    \tkzLabelPoint[above left](a'){$a'$} \tkzLabelPoint[above right](b'){$b'$}
    \tkzDrawPolygon(a,b,b',a')
    \tkzDrawSegments[dashed](a,b')
\end{tikzpicture}
\caption{Decomposition into simplices}
\end{figure}
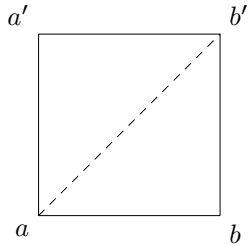

\noindent The first identity below corresponds to the lower right simplex, the second to the upper left simplex.
\begin{align}\label{eq: Cech cocycle relations 1}
    & (\partial\Gamma)_{a,b,b'}(z(h),z)=\overbrace{\Gamma_{a,b}(z(h),z)}^{\textcolor{red}{\ast}}\Gamma_{b,b'}(z(h),z)\Gamma_{b',a}(z(h),z) =\phi^{0,2}_{a,b,b'}(z(h),z)\\
    & (\partial\Gamma)_{a,a',b'}(z(h'),z)=\Gamma_{a,a'}(z(h'),z)\underbrace{\Gamma_{a',b'}(z(h'),z)}_{\textcolor{red}{\ast}}\Gamma_{b',a}(z(h'),z) =\phi^{0,2}_{a,a',b'}(z(h'),z) \label{eq: Cech cocycle relations 2}
\end{align}
 We are interested in the quotient of the two \textcolor{red}{$\ast$}. Combining equations \eqref{eq: Cech cocycle relations 1} and \eqref{eq: Cech cocycle relations 2} above, we can express the quotient:
\begin{equation}\label{eq: quotients of gamma}
    \frac{\Gamma_{a,b}(z(h),z)}{\Gamma_{a',b'}(z(h'),z)}= \frac{\Gamma_{a,a'}(z(h'),z)}{\Gamma_{b,b'}(z(h),z)}\cdot \frac{\Gamma_{b',a}(z(h'),z)}{\Gamma_{b',a}(z(h),z)}\cdot \frac{\phi^{0,2}_{a,b,b'}(z(h),z)}{\phi^{0,2}_{a,a',b'}(z(h'),z)}.
\end{equation}
There are two observations to make. First, we exploit the readily verified $\SL(\Lambda)$-equivariance of $\Gamma_{a,b}$, explicitly that $\Gamma_{ga,gb}(w,gz)=\Gamma_{a,b}(w,z)$ for $g\in \SL(\Lambda)$; equivalently that $(d\Gamma)_{a,b}(g)(w,z)= 1$.
Together with the assumption \ref{item (i)} that $z$ is an eigenvector for the action of $g$, we deduce the following:
\begin{equation*}
    \Gamma_{g a,g a'}(z(h),z) = \Gamma_{a,a'}(z(h),g^{-1}z)=\Gamma_{a,a'}(gz(h),z).
\end{equation*}
We now use the assumptions \ref{item (ii)} and \ref{item (iii)}, namely that $h'=h+v=g^{-1}h+u$ for $u$, $v\in L$, and recalling also that $g$ acts on $z$ by pullback. Thus, the quotient in \eqref{eq: quotients of gamma} becomes:
\begin{equation*}
    \frac{\Phi\triple{h}}{\Phi\triple{h'}}= \frac{\Gamma_{a,a'}(gz(h)+z(u),z)}{\Gamma_{a,a'}(gz(h),z)}\cdot\frac{\Gamma_{b',a}(z(h)+z(v),z)}{\Gamma_{b',a}(z(h),z)}\cdot \frac{\phi^{0,2}_{a,b,b'}(z(h),z)}{\phi^{0,2}_{a,a',b'}(z(h'),z)}.
\end{equation*}
We observe that both quotients are of the form $(d\Gamma)_{a,b}(u)(w,z) = \Gamma_{a,b}(w+z(u),z)/\Gamma_{a,b}(w,z)$ for $u\in L$. Hence:
\begin{equation}\label{eq: dGamma}
     \frac{\Phi\triple{h}}{\Phi\triple{h'}}=(d\Gamma)_{a,a'} (u)(gz(h),z)\cdot (d\Gamma)_{b',a} (v)(z(h),z)\cdot \frac{\phi^{0,2}_{a,b,b'}(z(h),z)}{\phi^{0,2}_{a,a',b'}(z(h'),z)}.
\end{equation}
We now turn to our second relation.

\noindent \textcolor{red}{\circled{2}} The second relation $d\Gamma=\phi^{1,1}\cdot \partial\Delta$ allows us to rewrite the expression above in terms of quotients of $\Delta$. Precisely, the relation tells us that:
\begin{equation*}
    (d\Gamma)_{a,b}=\phi^{1,1}_{a,b}\cdot (\partial\Delta)_{a,b}= \phi^{1,1}_{a,b}\cdot \frac{\Delta_a}{\Delta_b}.
\end{equation*}
Thus, we may rewrite the identity in \eqref{eq: dGamma} as:
\begin{align}
    \frac{\Phi\triple{h}}{\Phi\triple{h'}}&=\phi^{1,1}_{a,a'}(u)(gz(h),z)\cdot \frac{\Delta_a(u)(gz(h),z)}{\Delta_{a'}(u)(gz(h),z)}\cdot \phi^{1,1}_{b',a}(v)(z(h),z)\cdot \frac{\Delta_{b'}(v)(z(h),z)}{\Delta_a(v)(z(h),z)}\cdot \frac{\phi^{0,2}_{a,b,b'}(z(h),z)}{\phi^{0,2}_{a,a',b'}(z(h'),z)} \nonumber
    \\ & =\underbrace{\frac{\Delta_a(u)(gz(h),z)}{\Delta_a(v)(z(h),z)}}_{1.}\cdot \underbrace{\frac{\Delta_{b'}(v)(z(h),z)}{\Delta_{a'}(u)(gz(h),z)}}_{2.}\cdot \phi^{1,1}_{a,a'}(u)(gz(h),z)\cdot \phi^{1,1}_{b',a}(v)(z(h),z)\cdot \frac{\phi^{0,2}_{a,b,b'}(z(h),z)}{\phi^{0,2}_{a,a',b'}(z(h'),z)}\label{eq: quotients 1 and 2}
\end{align}
Simplifying the quotients of $\Delta$ in \eqref{eq: quotients 1 and 2} will require us to use our third and final relation.

\noindent \textcolor{red}{\circled{3}} Our final relation is that $d\Delta=\phi^{2,0}$. In particular, for $u,v\in L$, we have the group cocycle relation:
\begin{equation}\label{eq: delta cocycle relation}
    \phi^{2,0}_{a}(u,v)(w,z)\cdot \Delta_a(u+v)(w,z)=\Delta_a(u)(w,z) \cdot\Delta_a(v)(w+z(u),z).
\end{equation}
The Proposition will follow from the next two claims which resolve each quotient.
\begin{claim}
    $\Delta_a(u)(gz(h),z)/\Delta_a(v)(z(h),z) = \phi^{2,0}_a(u,v-u)(z(h),z)$.
\end{claim}
\begin{proof}
    We use the fact \ref{item (iv)} that $g^{-1}h - h = v-u\in \ker(a)$. Using the cocycle relation, we can write:
    \begin{equation*}
         \phi^{2,0}_a(u,v-u)(w,z)\cdot \Delta_a(u+(v-u))(w,z)=\Delta_a(v-u)(w,z)\cdot \Delta_a(u)(w+z(v-u),z).
    \end{equation*}
    From the definition of $\Delta$, it is apparent that $\Delta_a(v-u)(w,z)=1$ since $v-u\in\ker(a)$. Thus, we have deduced the identity $ \phi^{2,0}_a(u,v-u)(w,z)\cdot \Delta_a(v)(w,z)=\Delta_a(u)(w+z(v-u),z)$ and evaluating at $w = z(h)$ yields:
    \begin{equation*}
        \phi^{2,0}_a(u,v-u)_a(z(h),z)\cdot\Delta_a(v)(z(h),z)=\Delta_a(u)(z(h+v-u),z)=\Delta_a(u)(z(g^{-1}h),z)=\Delta_a(u)(gz(h),z).
    \end{equation*}
    Thus, the first quotient in \eqref{eq: quotients 1 and 2} is equal to $\phi^{2,0}_a(u,v-u)(z(h),z)$.
\end{proof}
\begin{claim}
    $\Delta_{b'}(v)(z(h),z)/\Delta_{a'}(u)(gz(h),z) = (\psi^{f,g\cdot f}_a(v)(z(h),z)\cdot\phi^{2,0}_{a'}(u,g^{-1}v-u)(gz(h),z))^{-1}$.
\end{claim}
\begin{proof}
    From the action of $\SL(\Lambda)$ on the set of framings and the construction of $\psi^{f,f'}$, we deduce the equalities:
    \begin{equation*}
        \Delta^{f}_{g^{-1}a}(g^{-1} v)(w,g^{-1}z)= \Delta^{g\cdot f}_{a}(v)(w,z) = \Delta^f_a(v)(w,z)\cdot\psi^{f,g\cdot f}_a(v)(w,z).
    \end{equation*}
    Applying this with $a'=g^{-1}b'$ and $w=z(h)$ yields:
    \begin{equation*}
        \Delta^f_{a'}(g^{-1}v)(gz(h),z)=\Delta^f_{g^{-1}b'}(g^{-1}v)(z(h),g^{-1}z) = \Delta^f_{b'}(v)(z(h),z)\cdot\psi^{f,g\cdot f}_{b'}(v)(z(h),z).
    \end{equation*}
    Thus, the quotient becomes:
    \begin{equation}\label{eq: delta quotient 1}
        \frac{\Delta_{b'}(v)(z(h),z)}{\Delta_{a'}(u)(gz(h),z)} =  \frac{\Delta_{a'}(g^{-1}v)(gz(h),z)}{\Delta_{a'}(u)(gz(h),z)} \cdot \frac{1}{\psi^{f,g\cdot f}_{b'}(v)(z(h),z)}.
    \end{equation}
    As before, we use the cocycle relation \eqref{eq: delta cocycle relation}, together with fact \ref{item (v)} that $g^{-1}h'-h' = g^{-1}v-u\in\ker(a')$ to deduce:
    \begin{equation*}
        \phi^{2,0}_{a'}(u,g^{-1}v-u)(w,z)\cdot\Delta_{a'}(u+(g^{-1}v-u))(w,z)=\Delta_{a'}(u)(w,z)\cdot\underbrace{\Delta_{a'}(g^{-1}v-u)(w+z(u),z)}_{=1}.
    \end{equation*}
    Evaluating at $w = z(g^{-1}h)=gz(h)$ yields:
    \begin{equation}\label{eq: delta quotient 2}
        \frac{\Delta_{a'}(g^{-1}v)(gz(h),z)}{\Delta_{a'}(u)(gz(h),z)} = \frac{1}{\phi^{2,0}_{a'}(u,g^{-1}v-u)(gz(h),z)}.
    \end{equation}
    Combining equations \eqref{eq: delta quotient 1} and \eqref{eq: delta quotient 2},
     yields the result that the second quotient appearing in \eqref{eq: quotients 1 and 2} is given by:
     \begin{equation*}
         \frac{\Delta_{b'}(v)(z(h),z)}{\Delta_{a'}(u)(gz(h),z)}= \frac{1}{ \psi^{f,g\cdot f}_{b'}(v)(z(h),z)\cdot\phi^{2,0}_{a'}(u,g^{-1}v-u)(gz(h),z)}.
     \end{equation*}
\end{proof}
Combining all the work above finishes the proof of the proposition.
\end{proof}
\subsection{Root of unity}
By combining the cocycle expression above with the limit formula stated in Theorem \ref{theorem: limit formula}, we are able to deduce the independence of $h$ up to a root of unity.
\begin{corollary}\label{corollary: root of unity}
    For $h$ and $h'$ two admissible vectors, the quotient ${\bf\Gamma}_{\mathfrak{a},h}(L)/{\bf\Gamma}_{\mathfrak{a},h'}(L)$ is a root of unity.
\end{corollary}
\begin{proof}
    The statement that the cocycle $D(\Delta,\Gamma)$ is valued in $\mathcal{O}^*$ can actually be strengthened: the cocycle is actually valued in $\exp(2\pi i \cdot \mathbb{Q}(z(L))[w])$. As discussed in Remark \ref{remark: admissible vector gives admissible triple}, both $(h,\langle\sigC\rangle,\varepsilon)$ and $(h',\langle\sigC\rangle,\varepsilon)$ are admissible triples, and so:
    \begin{equation*}
        \frac{\Phi(h,\langle \sigC\rangle,\varepsilon )}{\Phi(h',\langle \sigC\rangle ,\varepsilon )}=\frac{\Gamma_{a,b}(\sigC(h),\sigma_{\mathbb{C}};L)}{\Gamma_{a',b'}(\sigC(h'),\sigma_{\mathbb{C}};L)}\in \exp(2\pi i \cdot \mathbb{Q}(\sigma_{\mathbb{C}}(L))[\sigC(h)]).
    \end{equation*}
    Thus, we conclude that:
    \begin{equation}\label{eq: corolllary root of unity}
         \frac{{\bf\Gamma}_{\mathfrak{a},h}(L)}{{\bf\Gamma}_{\mathfrak{a},h'}(L)} =\frac{\Gamma_{a,b}(h,\mathfrak{a}^{-1}L)}{\Gamma_{a',b'}(h',\mathfrak{a}^{-1}L)}\cdot\frac{\Gamma_{a',b'}(h',L)^{\operatorname{N}(\mathfrak{a})}}{\Gamma_{a,b}(h,L)^{\operatorname{N}(\mathfrak{a})}} \in\exp(2\pi i\cdot  \sigma_{\mathbb{C}}(F)).
    \end{equation}
    The limit formula implies that $|{\bf\Gamma}_{\mathfrak{a},h}(L)|$ is independent of $h$, so that $|{\bf\Gamma}_{\mathfrak{a},h}(L)/{\bf\Gamma}_{\mathfrak{a},h'}(L)|=1$.

    With these two inputs, we are able to deduce the result. We write $x={\bf\Gamma}_{\mathfrak{a},h}(L)/{\bf\Gamma}_{\mathfrak{a},h'}(L)$. Then:
    \begin{equation*}
        \underbrace{\text{Limit formula}}_{|x|=1}\quad + \quad \underbrace{(\ref{eq: corolllary root of unity})}_{\substack{x=\exp(2\pi i \alpha)\\\alpha\in \sigma_{\mathbb{C}}(F)}} \quad \implies \quad \underbrace{\text{Corollary \ref{corollary: root of unity}}}_{\alpha\in \sigma_{\mathbb{C}}(F)\cap\mathbb{R}=\mathbb{Q} \implies x \text{ is a root of unity}} 
    \end{equation*}
\end{proof}
In the next section, we will prove that this root of unity is $1$, \textit{independently} of the limit formula, by explicitly computing the cocycles involved.
\newpage
\section{Vanishing theorems}\label{sec: vanishing theorems}
We now introduce a superlattice $V\supset L$, with respect to which we smooth our elliptic Gamma functions. We only consider those lattice pairs $(V,L)$ for which the quotient $V/L$ is cyclic of order $N$.

To define smoothed avatars of all the cocycles appearing in Proposition \ref{prop: cocycle expression}, we need a pair of framings for $L$ and $V$, which we denote $\underline{f} = (f_L,f_V)$. Moreover, whilst $\Hom(V,\mathbb{Z})\hookrightarrow\Hom(L,\mathbb{Z})$, we cannot embed the primitive functionals: $\Hom(V,\mathbb{Z})_{\text{prim}}\not\hookrightarrow\Hom(L,\mathbb{Z})_{\text{prim}}$. To address this, we define the following subset:
\begin{equation*}
    \mathcal{V}\defeq   \{a\in\Hom(V,\mathbb{Z}): V = L +\ker(a)\}
\end{equation*}
The relevance of this set $\mathcal{V}$ is clarified by the following lemma.
\begin{lemma}
    We have the equality $\mathcal{V} \cap (V^{\vee})_{\text{prim}} = (L^{\vee})_{\text{prim}}\cap (V^{\vee})_{\text{prim}}$. 
\end{lemma}
\begin{proof}
    By applying the snake lemma to:
    \[\begin{tikzcd}
	0 & L & {V} & {V/L} & 0 \\
	\\
	0 & {\mathbb{Z}} & {\mathbb{Z}} & 0 & 0
	\arrow[from=1-1, to=1-2]
	\arrow[from=1-2, to=1-3]
	\arrow["a", from=1-2, to=3-2]
	\arrow[from=1-3, to=1-4]
	\arrow["a", from=1-3, to=3-3]
	\arrow[from=1-4, to=1-5]
	\arrow[from=1-4, to=3-4]
	\arrow[from=3-1, to=3-2]
	\arrow[from=3-2, to=3-3,equal]
	\arrow[from=3-3, to=3-4]
	\arrow[from=3-4, to=3-5]
\end{tikzcd}\]
    we deduce that $a\in\mathcal{V}$ is equivalent to $a(L)=a(V)$, from which the equality follows immediately.
\end{proof}
We now smooth the cochain $(\Delta,\Gamma)$. Since we must simultaneously consider it as a cochain for $L$ and $V$, we restrict our group $G_3$ to the congruence subgroup $G=L\rtimes (\SL(L^{\vee})\cap \SL(V^{\vee}))\subset L\rtimes \SL(L^{\vee})$. Choosing coordinates such that $V^{\vee} = \langle e^1,e^2,Ne^3\rangle \subset L^{\vee}=\bigoplus \mathbb{Z}e^i$, an automorphism $g$ of $L^{\vee}$ restricts to an automorphism of $V^{\vee}$ if and only if its matrix form with respect to the basis $\{e^i\}$ has the following form modulo $N$:
\begin{equation*}
    [g]_{\{e^1,e^2,e^3\}} \equiv \begin{pmatrix}
        * & * & * \\ * & * & * \\ 0 & 0 & *
    \end{pmatrix} 
\end{equation*}
Notice that $\SL(L^{\vee})\cap \SL(V^{\vee})$ acts on $\mathcal{V}$ and so we obtain a $G$-invariant open cover $X=\bigcup_{a\in \mathcal{V}} V_{a}\subseteq X_3 = \bigcup_{a\in (L^{\vee})_{\text{prim}}}V_a$.
\begin{remark}
    Note that $X=X_3$: for all $\langle z\rangle \in \base$, we have $z\in U_a$ for some $a\in \mathcal{V}$. Writing $L=\oplus \mathbb{Z}e_i$ and $V=\langle e_1,e_2,e_3/N\rangle$, it is clear that the dual functionals $e^1,e^2\in \mathcal{V}$. With respect to this basis, we write $[z]_{\{e^{i}\}}= \begin{pmatrix}
        z_1 & z_2 & z_3
    \end{pmatrix}$. By the assumption $z\notin \mathbb{P}(\Lambda\otimes\mathbb{R})$, it follows that either $z_2/z_3$ or $z_1/z_3\notin\mathbb{R}$. Thus, $z\in U_a$ for at least one of $\pm e^1,\pm e^2$. 

    Despite this, we will continue to use the notation $X$ instead of $X_3$ to indicate that we are using a restricted open cover to compute $G$-equivariant cohomology.
\end{remark}
\begin{definition}[Smoothed Elliptic Gamma Cochain]
    The smoothed elliptic Gamma cochain is the following $G$-equivariant $1$-cochain valued in $\mathcal{M}^*$ for the $G$-invariant open cover $\mathcal{U} = \{V_a\}_{a\in\mathcal{V}}$:
    \begin{equation*}
        (\widetilde{\Delta}^{\underline{f}},\widetilde{\Gamma})\in C^1_{G}(\mathcal{U},\mathcal{M}^*) = C^1(G,\check{C}^0(\mathcal{U},\mathcal{M}^*)) \oplus C^0(G,\check{C}^1(\mathcal{U},\mathcal{M}^*)) 
    \end{equation*}
    where:
    \begin{equation*}
        \widetilde{\Delta}^{\underline{f}}_a(u)(w,z) \defeq \frac{\Delta^{f_L}_a(u)(w,z;L)^{N}}{\Delta^{f_{V}}_a(u)(w,z;V)} \quad \quad \widetilde{\Gamma}_{a,b}(w,z) \defeq \frac{\Gamma_{a,b}(w,z;L)^{N}}{\Gamma_{a,b}(w,z;V)}
    \end{equation*}
\end{definition}
Using the smoothed cochain, we obtain smoothed cocycles by applying the differential $D(\widetilde{\Delta},\widetilde{\Gamma}) = (\widetilde{\phi^{2,0}},\widetilde{\phi^{1,1}},\widetilde{\phi^{0,2}})$. These are related to the unsmoothed cocycles as you would expect:
\begin{equation*}
    \widetilde{\phi^{2,0}}(-)=\frac{\phi^{2,0}(-;L)^{N}}{\phi^{2,0}(-;V)} \hspace{1cm} \widetilde{\phi^{1,1}}(-)=\frac{\phi^{1,1}(-;L)^{N}}{\phi^{1,1}(-;V)} \hspace{1cm} \widetilde{\phi^{0,2}}(-)=\frac{\phi^{0,2}(-;L)^{N}}{\phi^{0,2}(-;V)}
\end{equation*}
where we have suppressed the inputs for clarity. In other words, we have $D(\widetilde{\Delta},\widetilde{\Gamma}) = \widetilde{D(\Delta,\Gamma)}$. 

Similarly, we define a smoothed change of framing cochain. For two pairs of framings $\underline{f}$ and $\underline{f'}$, we define:
\begin{equation*}
    \widetilde{\psi^{\underline{f},\underline{f'}}}(-)\defeq\frac{\psi^{f_L,f'_L}(-;L)^{N}}{\psi^{f_{V},f'_{V}}(-;V)}
\end{equation*}
which by construction will satisfy the smoothed version of \eqref{eq: change of framings}, namely:
\begin{equation}\label{eq: relate smoothed psi to smoothed delta}
    \widetilde{\psi^{\underline{f},\underline{f'}}} = \frac{\widetilde{\Delta}^{\underline{f'}}}{\widetilde{\Delta}^{\underline{f}}}.
\end{equation}
Finally, we define the smoothed analogue of $\Phi$.
\begin{definition}
    Let $g\in \operatorname{SL}(L^{\vee})\cap \SL(V^{\vee})$, for which $\chi_g\in\mathbb{Q}[x]$ is irreducible and has negative discriminant,  and $z\in\Lambda\otimes\mathbb{C}-\Lambda\otimes\mathbb{R}$ be an eigenfunctional of $g$. Then, for those nonzero elements $h\in L\otimes\mathbb{Q}$ such that associated primitive functional $a$ lies in $\mathcal{V}$, we define:
    \begin{equation*}
        \widetilde{\Phi}\triple{h} \defeq \widetilde{\Gamma}_{a,ga}(z(h),z).
    \end{equation*}
\end{definition}
Since $h\in \ker(a)\otimes\mathbb{Q}$, a sufficient condition for $a\in \mathcal{V}$ is that $\langle h \rangle_{\mathbb{Q}}$ generates $V/L$. We thus refine our notion of an admissible triple to include this condition.
\begin{definition}[Admissible triple for the pair $(V,L)$]  
    Consider a nonzero rational vector $h\in L\otimes\mathbb{Q}$, strictly complex functional $z\in \Lambda\otimes\mathbb{C}-\Lambda\otimes\mathbb{R}$ and automorphism $g\in \SL(L^{\vee})\cap \SL(V^{\vee})$ that has no rational eigenvalues. We call $\triple{h}$ an admissible triple for the pair $(V,L)$ if:
    \begin{equation*}
        g\cdot \langle z \rangle =\langle z\rangle \quad\quad  g\cdot  h\equiv h \text{ modulo }L \quad\quad \langle h\rangle_{\mathbb{Q}} \text{ generates }V/L.
    \end{equation*}
\end{definition}
\begin{remark}
    In contrast to the definition for an admissible triple for $L$, for $h'\in h+L$, the triple $\triple{h'}$ is no longer necessarily admissible for the pair $(V,L)$.
\end{remark}
We are now ready to state a smoothed version of Proposition \ref{prop: cocycle expression}. We fix a pair of framings $\underline{f}=(f_L,f_{V})$.
\begin{proposition}\label{prop: smoothed cocycle expression}
    Let $\triple{h}$ and $\triple{h'}$ be two admissible triples for the pair $(V,L)$. If $h'\equiv h$ modulo $L$:      
    \begin{equation*}
        \frac{\widetilde{\Phi}\triple{h}}{\widetilde{\Phi}\triple{h'}}=\widetilde{\psi_{b'}^{\underline{f},g \cdot \underline{f}}}(v)(z(h),z)^{-1}\cdot \frac{\widetilde{\phi^{2,0}_a}(u,v-u)(z(h),z)}{\widetilde{\phi^{2,0}_{a'}}(u,g^{-1}v-u)(gz(h),z) }\cdot \widetilde{\phi^{1,1}_{a,a'}}(u)(gz(h),z)\cdot \widetilde{\phi^{1,1}_{b',a}}(v)(z(h),z)\cdot \frac{\widetilde{\phi^{0,2}_{a,b,b'}}(z(h),z)}{\widetilde{\phi^{0,2}_{a,a',b'}}(z(h'),z)}
    \end{equation*}
    where as before $b=ga$, $b'=ga'$, and $u$, $v$ are the elements of $L$ determined by:
    \begin{equation*}
        h'=h+v \hspace{1cm} \text{and}\hspace{1cm} h' = g^{-1}h+u.
    \end{equation*}
\end{proposition}
\begin{proof}
    Since $g$ acts on $\mathcal{V}$, from our assumption that $a$, $a'\in \mathcal{V}$, it follows that $b$, $b'\in\mathcal{V}$, so that all the smoothed cocycles appearing are well-defined. The result follows immediately from Proposition \ref{prop: cocycle expression} by construction.
\end{proof}
\begin{remark}\label{remark: geometric / evaluation}
    We quickly discuss the assumptions of Proposition \ref{prop: smoothed cocycle expression}. The assumptions that are carried over from Proposition \ref{prop: cocycle expression} are \emph{geometric} assumptions, in that they imply the triples $\triple{h}$ and $\triple{h'}$ define homologous loops in $X\times_G EG$ as discussed in Section \ref{subsec: geometric}. For the smoothed expression, we only need to impose that $a$, $a'\in \mathcal{V}$. However, we have imposed the stronger condition that 
    $\langle h\rangle_{\mathbb{Q}}$, $\langle h'\rangle_{\mathbb{Q}}$ generate $V/L$; equivalently that $(a,b)$, $ (a',b')\in \mathcal{W}$, for $\mathcal{W}$ defined in \eqref{eq: def of W}. This is an \emph{evaluation} assumption, in the sense that it ensures the functionals are well-placed with respect to the smoothing pair $(V,L)$. To appreciate the necessity of this condition, it is instructive to consider the extremal case of $h\in L$, which we discuss in Appendix \ref{appendix: extremal}.
\end{remark}
Since the elliptic Gamma function, and consequently $\widetilde{\Phi}\triple{h}$, is independent of any framing, we can pick one that suits our purpose. This brings us to the notion of a \textit{compatible} pair.
\begin{definition}[Compatible pair]
    We call the pair of framings $\underline{f}= (f_L,f_{V})$ a compatible pair if they satisfy the following condition for all $a\in \mathcal{V}$: that $f_L(a)=(x_1,x_2,x_3)$ and $f_{V}(a)=(x_1,x_2,\frac{x_3}{N})$.
\end{definition}
The motivation for this definition is that for a compatible pair $\widetilde{\Delta}^{\underline{f}}$ is a product of $N$-smoothed $\theta_0^N$:
\begin{equation}\label{eq: smooth delta smooth theta}
    \widetilde{\Delta}^{\underline{f}}_a(u)(w,z) = \prod_{k=0}^{a(u)-1}\theta_0^N\left(\frac{w+kz(x_1)}{z(x_3)},\frac{z(x_2)}{z(x_3)}\right).
\end{equation}
This will allow us to leverage the $\mathbb{Z}^2\rtimes \Gamma_0(N)$-invariance of $\theta_0^N$ which is one of the key inputs into the following Theorem, from which the main result of our paper follows immediately.
\begin{theorem}\label{theorem: main theorem in main body}
    Let $\triple{h}$ and $\triple{h'}$ be two admissible triples for the pair $(V,L)$. If $h'\equiv h$ modulo $L$ and $(N,6\cdot \disc(\chi_g))=1$, then:
    \begin{equation*}
        \widetilde{\Phi}\triple{h} = \widetilde{\Phi}\triple{h'}.
    \end{equation*}
\end{theorem}
In the appendix, we address the case in which $N$ is no longer required to be coprime to $6$. 
\begin{proof}
    Since the identity does not depend on the choice of framing, it suffices to prove it for a particular $\underline{f}$. So we pick $\underline{f}$ to be a compatible pair. Propositions \ref{prop: smoothed compatible pairs} and \ref{prop: smoothed phi2,0} imply that the $\widetilde{\psi^{\underline{f},g\cdot\underline{f}}}$ and $\widetilde{\phi^{2,0}}$ terms vanish. Under the assumption that $d(h,h')=3$, Propositions \ref{prop: smoothed phi1,1} and \ref{prop: smoothed phi0,2} imply the remaining contributions from $\widetilde{\phi^{1,1}}$ and $\widetilde{\phi^{0,2}}$ vanish. 

    Finally, by Corollary \ref{corollary: d=3} to Proposition \ref{prop: image of Psi}, if $(N,6\cdot \disc(\chi_g))=1$, for any pair $h$ and $h'$ that form part of an admissible triple, we can always find an $h''$ such that $d(h,h'')=d(h',h'')=3$. This completes the proof.
\end{proof} 
\begin{corollary}
    Let $\mathfrak{f}$, $\sigC$, $\mathfrak{b}$, $L$ and $\varepsilon$ be as in the statement of Theorem \ref{theorem: arithmetic theorem in introduction}. Suppose $\mathfrak{a}$ is a prime ideal of prime norm and coprime to $6\operatorname{N}(\mathfrak{f})\cdot \disc(\chi_{\varepsilon})$. Then, for $h$ and $h'$ admissible elements, we have:
    \begin{equation*}
        {\bf\Gamma}_{\mathfrak{a},h}(L) = {\bf\Gamma}_{\mathfrak{a},h'}(L).
    \end{equation*}
\end{corollary}
\begin{proof}
    As discussed in Remark \ref{remark: admissible vector gives admissible triple}, the notion of admissibility in the arithmetic context implies that $(h,\langle \sigC\rangle,\varepsilon)$ is an admissible triple for $L$. The second ingredient to admissibility implies that $(h,\langle \sigC\rangle,\varepsilon)$ is an admissible triple for the pair $(\aL,L)$. Moreover, for a prime ideal $\mathfrak{a}$, the prime norm assumption is equivalent to $\aL/L$ being cyclic. Then, the result follows from the identity $\widetilde{\Phi}(h,\langle \sigC\rangle ,\varepsilon) = {\bf\Gamma}_{\mathfrak{a},h}(L)^{-1}$.
\end{proof}
\begin{remark}
    The assumption $\mathfrak{a}$ is coprime to $6\operatorname{N}(\mathfrak{f})\cdot \disc(\chi_{\varepsilon})$ excludes a finite set of primes. As $\disc(F/\mathbb{Q})|\disc(\chi_{\varepsilon})$, the condition excludes all primes of $F$ that lie above a rational prime that ramifies in $F/\mathbb{Q}$.
\end{remark}
From here on, we assume that $N=[V:L]$ is coprime to $6$, before returning to this assumption in the appendix. Before we begin the proofs of the vanishing results for the smoothed cocycles, let us note the varying complexity of the results. Apart from $\widetilde{\phi^{0,2}}$, the vanishing of all the smoothed avatars in Proposition \ref{prop: smoothed cocycle expression} are fairly immediate consequences of the $\mathbb{Z}^2\rtimes \Gamma_0(N)$-invariance of $\theta_0^N$, as stated in Proposition \ref{prop: j^N trivial}. The technical part of the proof is the vanishing of $\widetilde{\phi^{0,2}}$ and the reduction to the case $d=3$.

\subsection{Vanishing of $\widetilde{\psi^{\underline{f},\underline{f'}}}$}
We first consider the effect of smoothing upon the change of framing cochain.
\begin{proposition}\label{prop: smoothed compatible pairs}
    Suppose $\underline{f}$ and $\underline{f'}$ are compatible pairs. Then, the smoothed cochain $\widetilde{\psi^{\underline{f},\underline{f'}}}$ is trivial.
\end{proposition}
\begin{corollary}
    If $\underline{f}$ and $\underline{f'}$ are two compatible pairs, we have $\widetilde{\Delta}^{\underline{f}} =  \widetilde{\Delta}^{\underline{f'}}$.
\end{corollary}
\begin{proof}
    Immediate from \eqref{eq: relate smoothed psi to smoothed delta}.
\end{proof}
\begin{corollary}
    For $\underline{f}$ a compatible pair and $g\in \SL(L^{\vee})\cap \SL(V^{\vee})$, we have $\widetilde{\psi^{\underline{f},g\cdot \underline{f}}}$ is trivial.
\end{corollary}
\begin{proof}
     Recall that $g$ acts on $\mathcal{V}$. In particular, since $\underline{f}$ is compatible, we know that $g\cdot \underline{f}$ is also compatible.
\end{proof}
\begin{proof}[Proof of Proposition \ref{prop: smoothed compatible pairs}]
    Let $f_L(a)=(x_1,x_2,x_3)$ and $f'_L(a)=(y_1,y_2,y_3)$ for some $a\in \mathcal{V}$. It suffices to consider the evaluation of the cochain at $u\in L$ since $\Delta$ is defined via restriction to $L$. Using the identities of 
    \eqref{eq: relate smoothed psi to smoothed delta} and \eqref{eq: smooth delta smooth theta}, we can express:
    \begin{equation*}
     \widetilde{\psi^{\underline{f},\underline{f'}}_a}(u)(w,z)= \prod_{k=0}^{a(u)-1}\theta_0^N\left(\frac{w+kz(y_1)}{z(y_3)},\frac{z(y_2)}{z(y_3)}\right)\cdot \theta_0^N\left(\frac{w+kz(x_1)}{z(x_3)},\frac{z(x_2)}{z(x_3)}\right)^{-1}.
    \end{equation*}
    The result then follows from the two equalities:
    \begin{equation}\label{eq: equalities of smoothed theta}
        \theta_0^N\left(\frac{w+kz(y_1)}{z(y_3)},\frac{z(y_2)}{z(y_3)}\right)=\theta_0^N\left(\frac{w+kz(x_1)}{z(y_3)},\frac{z(y_2)}{z(y_3)}\right)=\theta_0^N\left(\frac{w+kz(x_1)}{z(x_3)},\frac{z(x_2)}{z(x_3)}\right).
    \end{equation}
    For the first equality, we use $y_1 -x_1\in \ker(a)\cap L$ and $\theta_0^N(-,\tau)$ is $\mathbb{Z}\oplus\mathbb{Z}\cdot \tau$-periodic. 
    
    Moreover, from the identity $\mathbb{Z}y_2\oplus\mathbb{Z}y_3 = \ker(a)\cap L=\mathbb{Z}x_2\oplus\mathbb{Z}x_3$, we can produce an automorphism $\delta = \big(\begin{smallmatrix} a&b\\c&d  \end{smallmatrix}\big)\in\SL_2(\mathbb{Z})$ such that $\delta\cdot (\begin{smallmatrix}
        y_2 \\y_3
    \end{smallmatrix}) = (\begin{smallmatrix}
        x_2 \\x_3
    \end{smallmatrix})$. To deduce the second equality in \eqref{eq: equalities of smoothed theta}, it remains to show that $\delta\in \Gamma_0(N)$. This follows from the compatibility of our framings. We know $x_3 = c y_2 +dy_3$ for $c,d\in \mathbb{Z}$ and so $x_3/N = (c y_2 +dy_3)/N\in \ker(a)\cap V = \mathbb{Z}y_2\oplus\mathbb{Z}\frac{y_3}{N}$. In particular, $c/N\in \mathbb{Z}$ and thus $\delta\in \Gamma_0(N)$.
\end{proof}
\subsection{Vanishing of $\widetilde{\phi^{2,0}}$}
We now turn to the smoothed cocycle $\widetilde{\phi^{2,0}}$. 
\begin{proposition}\label{prop: smoothed phi2,0}
    Suppose $\underline{f}$ is a compatible pair. Then, the smoothed cocycle $\widetilde{\phi^{2,0}}$ is trivial.
\end{proposition}
\begin{proof}
    Let us first consider $\phi^{2,0}$ restricted to the translation subgroup $L^2\subset G^2$. To begin with, we express the unsmoothed cocycle $\phi^{2,0}(-;L)$ in terms of $\theta_0$. 
    \begin{align}
        \phi^{2,0}_a(u,v)(w,z) &= \frac{\Delta_a(v)(w+z(u),z)\cdot \Delta_a(u)(w,z)}{\Delta_a(u+v)(w,z)} \nonumber \\&= \prod_{k=0}^{a(v)-1}\theta_0\left(\frac{w+z(u)+kz(x_1)}{z(x_3)},\frac{z(x_2)}{z(x_3)}\right)\cdot \theta_0\left(\frac{w+a(u)z(x_1)+kz(x_1)}{z(x_3)},\frac{z(x_2)}{z(x_3)}\right)^{-1}. \label{eq: phi^{2,0} explicit}
    \end{align}
    Since $\underline{f}$ is a compatible pair, we can use \eqref{eq: smooth delta smooth theta} to deduce:
    \begin{equation*}
        \widetilde{\phi^{2,0}_a}(u,v)(w,z) =  \prod_{k=0}^{a(v)-1}\theta_0^N\left(\frac{w+z(u)+kz(x_1)}{z(x_3)},\frac{z(x_2)}{z(x_3)}\right)\cdot \theta_0^N\left(\frac{w+a(u)z(x_1)+kz(x_1)}{z(x_3)},\frac{z(x_2)}{z(x_3)}\right)^{-1}.
    \end{equation*}
    We have $u-a(u)x_1\in \ker(a) = \mathbb{Z}x_2\oplus\mathbb{Z}x_3$ and so by the $\mathbb{Z}^2$-invariance of $\theta_0^N$ each term in the product is $1$.

    Although $\Delta$ is defined by restriction to the translation subgroup, $\phi^{2,0}$ does depend on the lattice automorphism. However, one can verify the following relation:
    \begin{equation*}
        \phi^{2,0}_a((u_1,g_1),(u_2,g_2))(w,z)=\phi^{2,0}_a(u_1,g_1u_2)(w,z)\cdot \psi^{f,g_1\cdot f}_a(g_1u_2)(w+z(u_1),z).
    \end{equation*}
    In particular, the vanishing of $\widetilde{\phi^{2,0}}$ in the general case reduces to the computation above for $L^2$ and Proposition \ref{prop: smoothed compatible pairs} on the vanishing of $\widetilde{\psi^{\underline{f},\underline{f'}}}$.
\end{proof}
\begin{remark}\label{remark: explicit phi2,0}
    We can also verify the proposition by giving an explicit description of $\phi^{2,0}(-;L)$ in terms of exponentials of polynomials. Using the expression in \eqref{eq: phi^{2,0} explicit}, we express $\phi^{2,0}$ in terms of the cocycle $j$ for $\theta_0$:
    \begin{equation*}
        \phi^{2,0}_a(u,v)(w,z) =\prod_{k=0}^{a(v)-1}j(u-a(u)x_1)\left(\frac{w+a(u)z(x_1)+kz(x_1)}{z(x_3)},\frac{z(x_2)}{z(x_3)}\right).
    \end{equation*}
    To use the explicit description of $j$ given in Lemma \ref{lemma: explicit cocycle j}, we first express $u$ and $v$ in terms of the basis specified by the framing $f_L(a) = (x_1,x_2,x_3)$: $u=\sum m_ix_i$ and $v=\sum n_i x_i$. In particular, $u-a(u)x_1 = u-m_1x_1 = m_2x_2+m_3x_3$.
    \begin{align*}
         \phi^{2,0}_a(u,v)(w,z) =\prod_{k=0}^{n_1-1}e^{2\pi i Q_{(m_2 ,m_3)}(\zeta,\tau)} = e^{2\pi i P^{2,0}_a(u,v)(w,z)}
    \end{align*}
    where $P^{2,0}_a(u,v)(w,z) = \sum_{k=0}^{n_1-1}Q_{(m_2 ,m_3)}(\zeta,\tau)$ and $(\zeta,\tau) = \left(\frac{w+a(u)z(x_1)+kz(x_1)}{z(x_3)},\frac{z(x_2)}{z(x_3)}\right)$. Explicitly, we have:
    \begin{equation*}
        P^{2,0}_a(u,v)(w,z) = \frac{n_1m_2}{2z(x_3)}\cdot\left(-2w+z(x_1)(1-n_1-2m_1)+z(x_2)(1-m_2)+z(x_3)\right).
    \end{equation*}
    Then, for $\underline{f}$ a compatible pair, we can explicitly compute:
    \begin{equation*}
        \widetilde{\phi^{2,0}_a}(u,v)(w,z)=e^{2\pi i \left(N\cdot P^{2,0}_a(u,v)(w,z;L)- P^{2,0}_a(u,v)(w,z;V) \right)} = e^{2\pi i\cdot n_1m_2\cdot \frac{N -1}{2}}= 1
    \end{equation*}
    since $N$ is odd. This verifies Proposition \ref{prop: smoothed phi2,0} directly.
\end{remark}
\subsection{Vanishing of $\widetilde{\phi^{1,1}}$}
We now turn to $\widetilde{\phi^{1,1}}\in C^{1,1}_G(\mathcal{U},\mathcal{O}^*)$. The choice of restricted cover means that this $(1,1)$-cochain is a map $ \widetilde{\phi^{1,1}}:G\rightarrow\prod_{(a,b)\in \mathcal{V}\times\mathcal{V}}\mathcal{O}^*(V_a\cap V_b)$. We do not claim that this cocycle is trivial in general, rather that it is only trivialised once we impose further conditions on the pair $(a,b)$. Rather than simply requiring the kernels of each functional to generate the quotient $V/L$, we further impose that the \textit{intersection} of the kernels generate the quotient. More precisely, we define:
\begin{equation}\label{eq: def of W}
    \mathcal{W} \defeq \{(a,b)\in \mathcal{V}^2: V = L + \ker(a,b)\}
\end{equation}
\begin{remark}
    When $(a,b)\in \mathcal{W}$, $\widetilde{\Gamma}_{a,b}$ is a product of $N$-smoothed $\Gamma$-functions, as in \eqref{eq: smoothed elliptic gamma}.
\end{remark}
\begin{proposition}\label{prop: smoothed phi1,1}
    Suppose $\underline{f}$ is a compatible pair. Then, for $(a,b)\in \mathcal{W}$, the smoothed cocycle $\widetilde{\phi^{1,1}_{a,b}}$ is trivial.
\end{proposition}
\begin{proof}
    Since $(d\Gamma)_{a,b}(g)$ is trivial for all pairs $a,b$ and $g\in \SL(L^{\vee})$, it follows that $(d\Gamma)_{a,b}((u,g))=(d\Gamma)_{a,b}(u)$. Combining this with $\Delta_a((u,g))= \Delta_a(u)$, we have $\phi^{1,1}_{a,b}((u,g))=\phi^{1,1}_{a,b}(u)$. Thus, it suffices to consider $\widetilde{\phi^{1,1}}$ restricted to $L$ for $u\in L$. We assume for simplicity that $u$ is in the positive cone of $a$ and $b$, with the computations in the other cases proceeding similarly.
    
    The proof will be by reduction to the case of $(a,b)$ of modulus $1$, for the notion of modulus introduced in Definition \ref{def: direction vector and modulus}. If $(a,b)$ has modulus $1$, then we can choose the pair $(\alpha,\beta)$ so that $L=\langle \alpha,\beta,\gamma\rangle $. In particular, we can construct a framing with $f'_L(a) = (\alpha,\beta,\gamma)$ and $f'_L(b) = (\beta,\alpha,-\gamma)$. By our assumption that $(a,b)\in\mathcal{W}$, the framing $f_L'$ extends to a compatible pair $\underline{f'}$.

    Using this basis, we write $u=m\alpha +n\beta +l\gamma$, for $m$, $n$, $l\in \mathbb{Z}$ and $m$, $n >0$. We repeatedly apply the difference relation $\Gamma(x+\tau,\tau,\sigma) = \theta_0(x,\sigma)\cdot \Gamma(x,\tau,\sigma)$ to deduce:
    \begin{equation*}
        (d\Gamma)_{a,b}(u)(w,z) = \frac{\Gamma(x+m\tau+n\sigma,\tau,\sigma)}{\Gamma(x,\tau,\sigma)} = \prod_{j=0}^{m-1}\theta_0(x+j\tau+n\sigma,\sigma)\cdot \prod_{k=0}^{n-1}\theta_0(x+k\sigma,\tau)
    \end{equation*}
    where $x = \frac{w}{z(\gamma)}$, $\tau =\frac{z(\alpha)}{z(\gamma)}$ and $\sigma = \frac{z(\beta)}{z(\gamma)}$. From $(a,b)\in \mathcal{W}$, we know $V=\langle \alpha,\beta,\frac{\gamma}{N}\rangle$ and so obtain:
    \begin{align*}
        (d\widetilde{\Gamma})_{a,b}(u)(w,z) &= \prod_{j=0}^{m-1}\theta_0^N(x+j\tau+n\sigma,\sigma)\cdot \prod_{k=0}^{n-1}\theta_0^N(x+k\sigma,\tau)
        \\ & = \prod_{j=0}^{a(u)-1}\theta_0^N\left(\frac{w+jz(\alpha)}{z(\gamma)},\frac{z(\beta)}{z(\gamma)}\right)\cdot \prod_{k=0}^{b(u)-1}\theta_0^N\left(\frac{w+kz(\beta)}{-z(\gamma)},\frac{z(\alpha)}{-z(\gamma)}\right)^{-1} = (\partial\widetilde{\Delta})^{\underline{f'}}_{a,b}(u)(w,z)
    \end{align*}
    where we used $\theta_0(x,\tau) = \theta_0(-x,-\tau)^{-1}$ and the $\mathbb{Z}^2$-invariance of $\theta_0^N$. Since both $\underline{f}$ and $\underline{f'}$ are compatible pairs, it follows from Proposition \ref{prop: smoothed compatible pairs} that $\widetilde{\Delta}^{\underline{f'}} =\widetilde{\Delta}^{\underline{f}}$. Thus, $(d\widetilde{\Gamma})_{a,b} = (\partial\widetilde{\Delta})^{\underline{f}}_{a,b}$ which is equivalent to the vanishing of $\widetilde{\phi^{1,1}_{a,b}}$.

    It remains to reduce to this case for $(a,b)$ of arbitrary modulus. We do so by forming a chain of pairs of reducing modulus, ultimately expressing $\widetilde{\Gamma}_{a,b}$ as a product over modulus $1$ pairs.
    
    To see how this chain is formed, for a pair $(a,b)$ of modulus $s$, we may assume $a=e^1$ and $b=re^1+se^2$ for $0\leq r <s$. But then letting $c=e^2$, we have the coplanar relation $\Gamma_{a,b}  = \Gamma_{a,c}\cdot\Gamma_{c,b}$, where the modulus of $(a,c)$ is $1$ and the modulus of $(c,b)=r <s$. Repeating this process, yields $\Gamma_{a,b}$ as a product of Gamma functions over modulus $1$ pairs. The resulting telescoping product of $(\partial \widetilde{\Delta})^{\underline{f}}_{a_i,a_{i+1}}$ cancels to give the result for $(a,b)$. Notice that every pair in the chain is in $\mathcal{W}$; indeed, $\ker(a,b)=\ker(a_i,a_{i+1})$ since they are coplanar.

    One subtlety is that the above argument only proves the triviality of $\widetilde{\phi^{1,1}}$ in the intersection $\bigcap V_{a_i}\subset V_{a}\cap V_b$ for the functionals $a_i$ appearing in the chain. However, we can always choose the signs of $a_i$ so that the total intersection is an open subset of $V_a \cap V_b$ and then deduce the result by analytic continuation.
\end{proof}
The definition of an admissible triple for the pair $(V,L)$ includes the assumption that both $\langle h\rangle_{\mathbb{Q}}$ and $\langle h'\rangle_{\mathbb{Q}}$ generate the quotient $V/L$. In particular, both $(a,b)$ and $(a',b')\in\mathcal{W}$. However, the $\widetilde{\phi^{1,1}}$ terms in Proposition \ref{prop: smoothed cocycle expression} involve the pairs $(a,a')$ and $(b',a)$. For these to vanish, we must further assume that $(a,a')$, $ (b',a)\in\mathcal{W}$; note that $(a,a')\in\mathcal{W}$ if and only if $(b,b')\in \mathcal{W}$. In other words, we must assume that every edge in one (equivalently, both) of the two simplices, depicted in the figure below, forms a pair in $\mathcal{W}$. 

To this end, let us define the following subset of triples of functionals:
\begin{equation*}
    \mathcal{Q}\defeq  \{(a,b,c)\in\mathcal{V}^3: (a,b),(b,c),(c,a)\in\mathcal{W}\}
\end{equation*}
\begin{definition}[$d(h,h')$]\label{def: d(h,h'}
    Let $\triple{h}$ and $\triple{h'}$ be two admissible triples for the pair $(V,L)$, which uniquely specify the functionals $a$, $a'$, $b=ga$, and $b=ga'$. By admissibility, we know $(a,b)$, $(a',b')\in \mathcal{W}$. We define $d(h,h')\in\{1,2,3\}$ to be the number of pairs in either the upper left or lower right simplex that lie in $\mathcal{W}$:
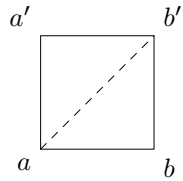
\begin{figure}[!htbp]\label{figure: simplices}
    \centering
    \begin{tikzpicture}[scale = 0.5]
    \tkzDefPoint(0,0){a}  \tkzLabelPoint[below left](a){$a$}
    \tkzDefPoint(3,0){b}  \tkzLabelPoint[below right](b){$b$}
    \tkzDefSquare(a,b)    \tkzGetPoints{b'}{a'}
    \tkzLabelPoint[above left](a'){$a'$} \tkzLabelPoint[above right](b'){$b'$}
    \tkzDrawPolygon(a,b,b',a')
    \tkzDrawSegments[dashed](a,b')
\end{tikzpicture}
\end{figure}
\end{definition}
\begin{remark}
    This number is independent of which simplex we choose because $(a,a')\in\mathcal{W}$ if and only if $(b,b')\in \mathcal{W}$. Moreover, we have $d(h,h')=3\iff (a,a',b')\in\mathcal{Q}\iff (a,b,b')\in\mathcal{Q}$.
\end{remark}
\begin{remark}
    We have dropped both $\langle z\rangle$ and $g$ from the notation as we are only varying $h$. Moreover, for fixed $g$, the invariant $d$ is in fact independent of the choice of eigenline as switching $\langle z\rangle$ to $\langle \bar{z}\rangle$ amounts to changing the sign of the associated functionals $a$ and $a'$ which has no effect on the kernel computation.
\end{remark}
We summarise the results up to this point in the following Corollary.
\begin{corollary}
     Let $\triple{h}$ and $\triple{h'}$ be two admissible triples for the pair $(V,L)$, such that $h'\equiv h$ modulo $L$. If $d(h,h')=3$, then:
     \begin{equation*}
        \frac{\widetilde{\Phi}\triple{h}}{\widetilde{\Phi}\triple{h'}}=\frac{\widetilde{\phi^{0,2}_{a,b,b'}}(z(h),z)}{\widetilde{\phi^{0,2}_{a,a',b'}}(z(h'),z)}.
    \end{equation*}
\end{corollary}
\begin{proof}
    Both $\widetilde{\Phi}$ and $\widetilde{\phi^{0,2}}$ are independent of a framing by definition and so take $\underline{f}$ to be a compatible pair.
\end{proof}
\subsection{Vanishing of $\widetilde{\phi^{0,2}}$}
The result follows, modulo the reduction to the $d=3$ case, from this final proposition, which is the most technically difficult part of the proof. An analogous result for the multiple elliptic Gamma functions $\Gamma_n$ has been independently arrived at in the recent work of Morain \cite{Morain_II}. We have highlighted the particular places where the proof methods overlap.
\begin{proposition}\label{prop: smoothed phi0,2}
    Suppose $(a,b,c)\in\mathcal{Q}$. Then the smoothed cocycle $\widetilde{\phi^{0,2}_{a,b,c}}$ is trivial. 
\end{proposition}
We require the following explicit description of the cocycle $\phi^{0,2}_{a,b,c}$.
\begin{proposition}[\cite{FHRZ08}Theorem 3.10]\label{prop: explicit phi0,2}
    For $(w,z)\in V_{a_1}\cap V_{a_2}\cap V_{a_3}$, the unsmoothed cocycle $\phi^{0,2}$ is given by:
    \begin{equation*}
        \phi^{0,2}_{a_1,a_2,a_3}(w,z)=e^{-\frac{2\pi i }{3!}P^{0,2}_{a_1,a_2,a_3}(w,z)}
    \end{equation*}
    for $P^{0,2}_{a_1,a_2,a_3}(w,z)\in \mathbb{Q}(z(L))[w]$. Explicitly, if $(a_1,a_2,a_3)$ is a positively oriented basis of $\Lambda\otimes\mathbb{Q}$, we have:
    \begin{equation*}
        P^{0,2}_{a_1,a_2,a_3}(w,z) = \sum_{\delta\in F_{a_1,a_2,a_3}(L)} B_{3,3}(w+z(\delta),z(\underline{x})).
    \end{equation*}
    Here, $x = (x_1,x_2,x_3)$, where $x_1$, $x_2$, and $x_3$ are the direction vectors of the wedges $(a_2,a_3)$, $(a_3,a_1)$, and $(a_1,a_2)$ respectively. Finally, $F_{a_1,a_2,a_3}(L)$ is the following choice of representatives for the quotient $L/\langle x_1,x_2,x_3\rangle$. 
    \begin{equation*}
        F_{a_1,a_2,a_3}(L) = \{\delta\in L: \hspace{0.2cm} 0 \leq a_j(\delta)<a_j(x_j),\hspace{0.2cm} j\in \{1,2,3\} \}.
    \end{equation*}
\end{proposition}
\begin{proof}
    This follows directly from the functional identity in \eqref{eq: cocycle identity}:
    \begin{equation*}
        \Gamma\left(\frac{w}{x_3},\frac{x_1}{x_3},\frac{x_2}{x_3}\right)\Gamma\left(\frac{w}{x_1},\frac{x_2}{x_1},\frac{x_3}{x_1}\right)\Gamma\left(\frac{w}{x_2},\frac{x_3}{x_2},\frac{x_1}{x_2}\right) = e^{-\frac{2\pi i}{3!}B_{3,3}(w,\underline{x})}
    \end{equation*} 
    together with the fact that $F_{a_1,a_2,a_3}(L)$ can be used as an index set to define each of $\Gamma_{a_i,a_j}$. Indeed, $F_{a_1,a_2,a_3}(L) $ modulo $\mathbb{Z}x_3$ is precisely $ F_{a_1,a_2}^{x_1,x_2}(L)/\mathbb{Z}x_3$, the index set introduced in Definition \ref{def: elliptic Gamma cochain}. We have the analogous relation for $(a_2,a_3)$ and $(a_3,a_1)$.
\end{proof}
\begin{remark}
    If $(a_1,a_2,a_3)$ are linearly dependent, then $\phi^{0,2}_{a_1,a_2,a_3} =1$ so the claim of Proposition \ref{prop: smoothed phi0,2} is immediate. Moreover, from $\Gamma_{a,b}=\Gamma_{b,a}^{-1}$, we deduce $\phi^{0,2}_{a_1,a_2,a_3} = (\phi^{0,2}_{a_1,a_3,a_2})^{-1}$. Thus, it suffices to consider the case $(a_1,a_2,a_3)$ is a positively oriented basis of $\Lambda\otimes\mathbb{Q}$.
\end{remark}
\begin{remark}
    For $\lambda\in \mathbb{C}^*$, if we make the substitution $t\mapsto \lambda t$ in the defining generating series for $B_{n,m}$ in \eqref{eq: exp generating function}, we deduce that the diagonal Bernoulli polynomials satisfy the following homogeneity relation: 
    \begin{equation}\label{eq: Bernoulli homogeneity}
        B_{n,n}(\lambda w,\lambda\underline{x}) = B_{n,n}(w,\underline{x}).
    \end{equation}
\end{remark}
From the description of $\phi^{0,2}$ above, we can express the smoothed cocycle as:
\begin{equation*}
     \widetilde{\phi^{0,2}_{a_1,a_2,a_3}}(w,z)=e^{-\frac{2\pi i}{3!} \left(N\cdot P^{0,2}_{a_1,a_2,a_3}(w,z;L)- P^{0,2}_{a_1,a_2,a_3}(w,z;V) \right)}.
\end{equation*}
Using the notation for direction vectors introduced in Definition \ref{def: direction vector and modulus}, the assumption $(a_1,a_2,a_3)\in\mathcal{Q}$ becomes:
\begin{equation*}
    (a_2,a_3)_{V}=(a_2,a_3)_L/N \hspace{0.5cm} \text{and}\hspace{0.5cm}(a_3,a_1)_{V}=(a_3,a_1)_L/N\hspace{0.5cm} \text{and}\hspace{0.5cm}(a_1,a_2)_{V}=(a_1,a_2)_L/N.
\end{equation*}
Then, by the homogeneity property of $B_{3,3}$ described in \eqref{eq: Bernoulli homogeneity}:
\begin{equation*}
    P^{0,2}_{a_1,a_2,a_3}(w,z;V) = \sum_{\mu\in F_{a_1,a_2,a_3}(V)} B_{3,3}(w+z(\mu),z(\underline{x}/N)) = \sum_{\mu\in F_{a_1,a_2,a_3}(V)} B_{3,3}(N\cdot(w+z(\mu)),z(\underline{x})).
\end{equation*}
Since Proposition \ref{prop: smoothed phi0,2} is equivalent to showing $N\cdot P^{0,2}_{a_1,a_2,a_3}(w,z;L)\equiv P^{0,2}_{a_1,a_2,a_3}(w,z;V)$ modulo $6\mathbb{Z}$, we must show:
\begin{equation}\label{eq: smoothed sum of Bernoulli}
     N\cdot \sum_{\delta\in F_{a_1,a_2,a_3}(L)} B_{3,3}(w+z(\delta),z(\underline{x})) \equiv \sum_{\mu\in F_{a_1,a_2,a_3}(V)} B_{3,3}(N\cdot(w+z(\mu)),z(\underline{x}))\hspace{0.5cm} \text{mod } 6\mathbb{Z}
\end{equation}
\subsubsection{Distribution Relation}
A critical step in the proof is the following rank $2$ Distribution Relation for $B_{3,3}$. This terminology is supposed to recall the rank $1$ Distribution Relation satisfied by the single variable Bernoulli polynomials $B_k$:
\begin{equation*}
    \sum_{m=0}^{N-1}B_{k}\left(w+\frac{m}{N}\right)N^{k-1}=B_k(Nw) .
\end{equation*}
If we set $w=0$, we get the following relation, where $b_k=B_k(0)$ is the $k$-th Bernoulli number.
\begin{equation}\label{eq: rank 1 DR}
    \sum_{m=0}^{N-1}B_{k}\left(\frac{m}{N}\right)N^{k-1}=b_k .
\end{equation}
To avoid any ambiguity, we define the following subobjects of $(\mathbb{Z}/N)^3$.
\begin{definition}[Line and Plane]\label{def: line and plane}
    A line in $(\mathbb{Z}/N)^3$ will be a submodule isomorphic to $\mathbb{Z}/N$. A plane in $(\mathbb{Z}/N)^3$ will be a submodule isomorphic to $(\mathbb{Z}/N)^2$.
\end{definition}
\begin{proposition}\label{prop: Bernoulli DR}
    Consider a rank $3$ lattice $U=\bigoplus_{j=1}^3 \mathbb{Z}x_j$. If $P$ is a plane in $\frac{1}{N}U/U\cong (\mathbb{Z}/N)^3$ that trivially intersects each coordinate axis and $(N,6)=1$, then for all $w\in\mathbb{C}$ and complex functionals $z:U\otimes\mathbb{C}\rightarrow\mathbb{C}$ such that $z(x_j)\neq 0$ we have the following distribution relation:
    \begin{equation}\label{eq: Bernoulli DR}
         N\cdot\sum_{\nu\in P}B_{3,3}(w+z(\nu),z(\underline{x})) \equiv B_{3,3}(Nw,z(\underline{x})) \text{ modulo } 6\mathbb{Z}
    \end{equation}
    where the coset representatives are taken inside of $\{0,\frac{1}{N}\dots,\frac{N-1}{N}\}^3$ and $\underline{x}=(x_1,x_2,x_3)$.
\end{proposition}
In other words, we claim that the difference is a constant in $6\mathbb{Z}$. By the continuity of $B_{3,3}$ and discreteness of $\mathbb{Z}$, it follows that this constant is necessarily independent of both $w$ and $z$. This will be verified by explicitly computing the constant.

The proof will proceed by using the exponential generating series to rewrite the difference as a sum involving the single-variable Bernoulli polynomials $B_k(x)$ and then using the Rank $1$ Distribution Relation stated in \eqref{eq: rank 1 DR} to simplify the difference.
\begin{proof}
    By writing $\nu=\frac{1}{N}\sum_{j=1}^3\nu_jx_j$, we can express the left hand side as:
    \begin{align*}
        \sum_{\nu\in P }NB_{3,3}(w+z(\nu),z(\underline{x}))&=\text{Coeff}\left(t^3;3!N\sum_{\nu\in P}e^{(w+z(\nu))t}\prod\frac{t}{e^{z(x_j)t}-1}\right) \\ & =\frac{3!N}{\prod_{j=1}^3z(x_j)}\cdot\text{Coeff}\left(t^3;\sum_{m\geq 0}\frac{w^m}{m!}t^m\sum_{\nu\in P}\prod_{j=1}^3\frac{z(x_j)te^{\frac{\nu_j}{N}z(x_j)t}}{e^{z(x_j)t}-1}\right).
    \end{align*}
    The (single-variable) Bernoulli Polynomials are defined via $\frac{te^{xt}}{e^{t}-1}=\sum_{k\geq 0}B_k(x)\frac{t^k}{k!} $ so that we recognise:
    \begin{equation*}
        \frac{z(x_j)te^{\frac{\nu_j}{N}z(x_j)t}}{e^{z(x_j)t}-1} = \sum_{k\geq 0}B_k\left(\frac{\nu_j}{N}\right)\frac{(z(x_j)t)^k}{k!} .
    \end{equation*}
    Thus, we can explicitly express the coefficient of $t^3$:
    \begin{equation}\label{eq: Bernoulli LHS}
        \sum_{\nu\in P }NB_{3,3}(w+z(\nu),z(\underline{x})) = \frac{3!N}{\prod_{j=1}^3z(x_j)}\sum_{l=0}^3\frac{w^l}{l!}\sum_{k_1+k_2+k_3=3-l}\Psi(k_1,k_2,k_3)
    \end{equation}
    where $\Psi(k_1,k_2,k_3) = \sum_{\nu\in P}\prod_{j=1}^3B_{k_j}\left(\frac{\nu_j}{N}\right)\frac{z(x_j)^{k_j}}{k_j!}$ contributes to the coefficient of $t^{k_1+k_2+k_3}$ in $\sum_{\nu\in P}\prod_{j=1}^3\frac{z(x_j)te^{\frac{\nu_j}{N}z(x_j)t}}{e^{z(x_j)t-1}}$.
    
    For the right hand side, we carry out the same computation, this time using the generating series for the Bernoulli numbers $b_k$ to rewrite:
    \begin{equation*}
        \frac{z(x_j)t}{e^{z(x_j)t}-1}=\sum_{k\geq 0}b_k\frac{(z(x_j)t)^k}{k!}.
    \end{equation*}
    This produces the equality:
    \begin{equation}\label{eq: Bernoulli RHS}
        B_{3,3}(Nw,z(\underline{x})) = \frac{3!}{\prod_{j=1}^3z(x_j)}\sum_{l=0}^3\frac{(Nw)^l}{l!}\sum_{k_1+k_2+k_3=3-l}\psi(k_1,k_2,k_3)
    \end{equation}
    where $\psi(k_1,k_2,k_3) = \prod_{j=1}^3b_{k_j}\frac{z(x_j)^{k_j}}{k_j!}$. Thus, combining \eqref{eq: Bernoulli LHS} and \eqref{eq: Bernoulli RHS}, we can express the difference as:
    \begin{equation*}
        \sum_{\nu\in P }NB_{3,3}(w+z(\nu),z(\underline{x})) -  B_{3,3}(Nw,z(\underline{x})) = \frac{3!}{\prod_{j=1}^3z(x_j)}\sum_{l=0}^3\frac{w^l}{l!}\sum_{k_1+k_2+k_3=3-l}N\Psi(k_1,k_2,k_3)-N^l\psi(k_1,k_2,k_3).
    \end{equation*}
    Lemma \ref{lemma: vanishing of non constant terms} below shows that the difference $N\Psi(k_1,k_2,k_3)-N^l\psi(k_1,k_2,k_3)$ vanishes in all cases except when $k_1=k_2=k_3=1$. In particular, the only non-zero contribution comes when $l=0$. Thus, the difference is independent of $w$ and can be written as:
    \begin{equation}\label{eq: Bernoulli difference}
         \sum_{\nu\in P }NB_{3,3}(w+z(\nu),z(\underline{x})) -  B_{3,3}(Nw,z(\underline{x})) = \frac{3!}{\prod_{j=1}^3z(x_j)}\left(N\Psi(1,1,1)-\psi(1,1,1)\right)
    \end{equation}
   where we can explicitly write:
   \begin{equation*}
       \Psi(1,1,1) = \sum_{\nu\in P}\prod_{j=1}^3B_1\left(\frac{\nu_j}{N}\right)\cdot z(x_j) 
       \hspace{2.5cm} \psi(1,1,1) = \prod_{j=1}^3b_1\cdot z(x_j).
   \end{equation*}
   We have $B_1(x) = -\frac{1}{2}+x$ and $b_1 = B_1(0)=-\frac{1}{2}$ and thus:
   \begin{equation*}
       N\Psi(1,1,1)-\psi(1,1,1) = \prod_{j=1}^{3}z(x_j)\cdot \left[N\sum_{\nu\in P}\left(-\frac{1}{8}+\frac{1}{4}\frac{\nu_1+\nu_2+\nu_3}{N}-\frac{1}{2}\frac{\nu_1\nu_2+\nu_2\nu_3+\nu_3\nu_1}{N^2}+\frac{\nu_1\nu_2\nu_3}{N^3}\right) +\frac{1}{8}\right].
   \end{equation*}
    Since projection onto each coordinate plane induces a bijection with $P$, we explicitly compute:
    \begin{equation*}
        \sum_{\nu\in P}\nu_1+\nu_2+\nu_3 = 3N\sum_{m=0}^{N-1}m=\frac{3}{2}N^2(N-1)\hspace{1.5cm}  \sum_{\nu\in P}\nu_1\nu_2+\nu_2\nu_3+\nu_3\nu_1 = 3\sum_{m,n=0}^{N-1}mn = \frac{3}{4}N^2(N-1)^2.
    \end{equation*}
    Thus, the difference in \eqref{eq: Bernoulli difference} becomes:
    \begin{equation*}
         \sum_{\nu\in P }NB_{3,3}(w+z(\nu),z(\underline{x})) -  B_{3,3}(Nw,z(\underline{x})) = 3!\left(-\frac{(N-1)^3}{8}+\frac{1}{N^2}\sum_{\nu\in P}\nu_1\nu_2\nu_3\right)
    \end{equation*}
    and is seen to be a rational constant independent of both $w$ and $z$. Since $(N,2)=1$, it follows that $(N-1)^3\in8\mathbb{Z}$ and Lemma \ref{lemma: consant in 6Z} below shows that $\sum_{\nu\in P}v_1v_2v_3\in N^2{\mathbb{Z}}$. In particular, we have shown that:
    \begin{equation*}
        \sum_{\nu\in P }NB_{3,3}(w+z(\nu),z(\underline{x})) -  B_{3,3}(Nw,z(\underline{x})) = 3!\left(-\frac{(N-1)^3}{8}+\frac{1}{N^2}\sum_{\nu\in P}\nu_1\nu_2\nu_3\right)\in 6 \mathbb{Z}
    \end{equation*}
    is an integer multiple of $6$ independent of $w$ and $z$.
\end{proof}
The following Lemma is a special case of \cite{Morain_II}[Lemma 25]. Though proven independently, the method of proofs are completely analogous. 
\begin{lemma}\label{lemma: vanishing of non constant terms}
    Let $k_i\geq 0$ integers and set $l=3-(k_1+k_2+k_3)\in\mathbb{Z}$. If one of the $k_i = 0$, then $N\Psi(k_1,k_2,k_3)-N^l\psi(k_1,k_2,k_3)=0$. In particular, among the $(k_1,k_2,k_3)$ with $k_1+k_2+k_3 \leq 3$, the only non-zero contribution is when $k_1 = k_2=k_3=1$.
\end{lemma}
\begin{proof}
    Without loss of generality, suppose that $k_3 = 0$. Thus, $l=3-(k_1+k_2)$. Since the plane $P$ trivially intersects each coordinate axis, projection induces a bijection between $P$ and the $x_1x_2$-coordinate plane. Thus:
    \begin{align*}
        N\Psi(k_1,k_2,k_3)-N^l\psi(k_1,k_2,k_3) & = N\sum_{m,n=0}^{N-1}B_{k_1}\left(\frac{m}{N}\right)\frac{z(x_1)^{k_1}}{k_1!}B_{k_2}\left(\frac{n}{N}\right)\frac{z(x_2)^{k_2}}{k_2!} -N^{l} b_{k_1}\frac{z(x_1)^{k_1}}{k_1!}b_{k_2}\frac{z(x_2)^{k_2}}{k_2!} \\ & =\frac{z(x_1)^{k_1}}{k_1!} \frac{z(x_2)^{k_2}}{k_2!}\cdot\left(N\sum_{m=0}^{N-1}B_{k_1}\left(\frac{m}{N}\right)\sum_{n=0}^{N-1}B_{k_2}\left(\frac{n}{N}\right)-N^{l}b_{k_1}b_{k_2}\right) .
    \end{align*}
    The result then follows from the rank $1$ Distribution Relation for single variable Bernoulli Polynomials:
    \begin{equation*}
        \sum_{m=0}^{N-1}B_{k_i}\left(\frac{m}{N}\right) = N^{1-{k_i}}b_{k_i}.
    \end{equation*}
\end{proof}
\begin{lemma}\label{lemma: consant in 6Z}
    For $(N,6)=1$ and a plane $P\subset (\mathbb{Z}/N)^3$ that trivially intersects each coordinate axis, the sum of the product of coordinates is divisible by $N^2$:
    \begin{equation*}
        \sum_{(\nu_1,\nu_2,\nu_3)\in P}\nu_1\nu_2\nu_3\in N^2\mathbb{Z}
    \end{equation*}
    where $\nu_i\in\{0,\dots,N-1\}$.
\end{lemma}
\begin{proof}
The key input from our assumption that $(N,6)=1$ is that both:
\begin{equation}\label{eq: sum of squares}
    \sum_{m=0}^{N-1}m^2=N\frac{(N-1)(2N-1)}{6}\in N\mathbb{Z} \hspace{1cm}\text{and}\hspace{1cm} \sum_{m=0}^{N-1}m=N\frac{N-1}{2}\in N\mathbb{Z} .
\end{equation}
In particular, we deduce the following useful facts, where, as above, we identify $(\mathbb{Z}/N)^3 = \{0,\dots,N-1\}^3$.
\begin{claim}\label{claim: projection to coord plane}
    For $(N,6)=1$ and a plane $P$ that trivially intersects each coordinate axis, we have $\sum_{\nu\in P}\nu_1\nu_2\in N^2\mathbb{Z}$, where $\nu = (\nu_1,\nu_2,\nu_3)$.
\end{claim}
For the proof of Lemma \ref{lemma: consant in 6Z}, we only need the sum to be in $N\mathbb{Z}$. However, to complete the proof of Proposition \ref{prop: smoothed phi0,2}, we will later require the stronger statement of the claim.
\begin{proof}[Proof of Claim \ref{claim: projection to coord plane}]
    Projection induces a bijection between $P$ and the $x_1x_2$-coordinate plane and so the sum is $\left(\sum_{m=0}^{N-1}m\right)^2\in N^2\mathbb{Z}$.
\end{proof}
\begin{claim}\label{claim: intersection with coord plane}
    For $(N,6)=1$ and a plane $P$ that trivially intersects each coordinate axis, we have $\sum_{\nu\in P,\nu_3=0}\nu_1\nu_2\in N\mathbb{Z}$, where $\nu = (\nu_1,\nu_2,\nu_3)$.
\end{claim}
\begin{proof}[Proof of Claim \ref{claim: intersection with coord plane}]
    By our assumption on $P$, the intersection of $P$ with the $x_1x_2$-coordinate plane is the line $\langle(u_1,u_2,0)\rangle$ for some $u_1$, $u_2\in (\mathbb{Z}/N)^*$. Thus, the sum is congruent to $u_1u_2\sum_{m=0}^{N-1}m^2\equiv 0$ modulo $N$.
\end{proof}
Since $N$ is odd, the claim that $S=\sum_{\nu \in P}\nu_1\nu_2\nu_3\in N^2\mathbb{Z}$ is equivalent to $2S\in N^2\mathbb{Z}$.
\begin{equation*}
    2S = \sum_{\nu_1,\nu_2,\nu_3\neq 0}\nu_1\nu_2\nu_3 + (N-\nu_1)(N-\nu_2)(N-\nu_3) \equiv N\sum_{\nu_1,\nu_2,\nu_3\neq 0}\nu_1\nu_2+\nu_2\nu_3+\nu_3\nu_1\hspace{0.5cm} \text{mod }N^2 
\end{equation*}
By symmetry, it suffices to prove that $\sum_{\nu_1,\nu_2,\nu_3\neq 0}\nu_1\nu_2\in N\mathbb{Z}$. We can write:
\begin{equation*}
    \sum_{\nu_1,\nu_2,\nu_3\neq 0}\nu_1\nu_2 =\sum_{\nu_3\neq 0}\nu_1\nu_2= \sum_{\nu\in P}\nu_1\nu_2- \sum_{\nu_3= 0}\nu_1\nu_2\in N\mathbb{Z}
\end{equation*}
by Claims \ref{claim: projection to coord plane} and \ref{claim: intersection with coord plane} respectively.
\end{proof}
\begin{remark}
    A near verbatim proof shows the rank $3$ Distribution Relation:
    \begin{equation*}
        \sum_{j_1,j_2,j_3=0}^{N-1}B_{3,3}\left(w+\frac{\sum j_ix_i}{N},\underline{x}\right) = B_{3,3}(Nw,\underline{x}).
    \end{equation*}
    Notice that this is an \textit{equality} of functions. One can obtain this equality modulo $6\mathbb{Z}$ directly by `taking' the \v{C}ech coboundary of the Distribution Relation:
    \begin{equation*}
       \prod_{j,k,l=0}^{N-1}\Gamma\left(x+\frac{j+k\tau+l\sigma}{N},\tau,\sigma\right)=\Gamma(Nx,\tau,\sigma).
    \end{equation*}
\end{remark}
\subsubsection{Sketch of proof}
We begin with a `moral' sketch of the proof. We have dropped the functionals from the notation of our coset representatives $F(L)$ for ease of notation and write $U=\langle{x_1,x_2,x_3\rangle}$ for $(x_1,x_2,x_3) = ((a_2,a_3)_L,(a_3,a_1)_L,(a_1,a_2)_L)$.
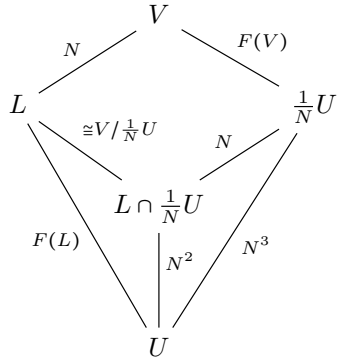
\begin{figure}[!htbp]
    \centering
\[\begin{tikzcd}
	& V \\
	L && {\frac{1}{N}U} \\
	& {L\cap \frac{1}{N}U} \\
	\\
	& U
	\arrow["N"', no head, from=1-2, to=2-1]
	\arrow["{F(V)}", no head, from=1-2, to=2-3]
	\arrow["{\cong V/\frac{1}{N}U}"{pos=0.45}, no head, from=2-1, to=3-2]
	\arrow["{F(L)}"', no head, from=2-1, to=5-2]
	\arrow["N"', no head, from=2-3, to=3-2]
	\arrow["{N^3}", no head, from=2-3, to=5-2]
	\arrow["N^2"{pos=0.3}, no head, from=3-2, to=5-2]
\end{tikzcd}\]
    \caption{Lattice of sublattices}
    \label{fig: lattice of sublattice}
\end{figure}

From this, we deduce the short exact sequence:
\[\begin{tikzcd}
	0 & {L\cap\frac{1}{N}U/U} & {L/U} & {V/\frac{1}{N}U} & 0
	\arrow[from=1-1, to=1-2]
	\arrow[from=1-2, to=1-3]
	\arrow[from=1-3, to=1-4]
	\arrow[from=1-4, to=1-5]
\end{tikzcd}\]
where $L\cap\frac{1}{N}U/U\cong (\mathbb{Z}/N)^2$ is a plane in $\frac{1}{N}U/U$ that trivially intersects each coordinate axis: indeed, by the $d=3$ assumption, $\frac{1}{N}U/L\cap\frac{1}{N}U\cong V/L\cong \mathbb{Z}/N$ and each $\frac{x_i}{N}\in \frac{1}{N}U/U$ generates this quotient. Thus, morally we can decompose $\delta\in F(L)$ as a sum $\delta = \nu+\mu$ for $\nu\in L\cap \frac{1}{N}U/U$ and $\mu\in F(V)$. 
\begin{equation*}\label{eq: misleading}
    \{\delta:\delta\in F(L)\} = \{\nu+\mu:\nu\in L\cap\frac{1}{N}U/U,\mu\in F(V)\} \tag{$\star$}.
\end{equation*}
Then, we can rewrite $N\cdot P^{0,2}_{a_1,a_2,a_3}(w,z;L)$ as follows:
\begin{align*}
    N\cdot P^{0,2}_{a_1,a_2,a_3}(w,z;L) = &\sum_{\delta\in F_{a_1,a_2,a_3}(L)} N\cdot B_{3,3}(w+z(\delta),z(\underline{x})) \overset{(\star)}{=} \sum_{\mu\in F_{a_1,a_2,a_3}(V)}N\cdot \sum_{\nu\in L\cap\frac{1}{N}U/U} B_{3,3}(w+z(\mu)+z(\nu),z(\underline{x})) \\ \equiv & \sum_{\mu\in F_{a_1,a_2,a_3}(V)} B_{3,3}(N\cdot(w+z(\mu)),z(x)) = P^{0,2}_{a_1,a_2,a_3}(w,z;V) \hspace{0.5cm}\text{ mod } 6\mathbb{Z} 
\end{align*}
by the Distribution Relation of Proposition \ref{prop: Bernoulli DR}. Therefore, $e^{-\frac{2\pi i }{3!}\text{(LHS - RHS)}}=1$, proving the triviality of $\widetilde{\phi^{0,2}_{a_1,a_2,a_3}}$.

Why is this not a formal proof? The misleading aspect is the decomposition in \eqref{eq: misleading}. To streamline notation, let $f=|F(V)|$ so that $|F(L)|=fN^2$. Thus we have $L\subset V\subset \frac{1}{Nf}U$ and can embed:
\begin{equation}\label{eq: two rows SES}
\begin{tikzcd}
	0 & {L\cap\frac{1}{N}U/U} & {L/U} & {V/\frac{1}{N}U} & 0 \\
	0 & {\left(\mathbb{Z}/N\mathbb{Z}\right)^3} & {\left(\mathbb{Z}/Nf\mathbb{Z}\right)^3} & {\left(\mathbb{Z}/f\mathbb{Z}\right)^3} & 0
	\arrow[from=1-1, to=1-2]
	\arrow[from=1-2, to=1-3]
	\arrow[hook, from=1-2, to=2-2]
	\arrow[from=1-3, to=1-4]
	\arrow[hook, from=1-3, to=2-3]
	\arrow[from=1-4, to=1-5]
	\arrow[hook, from=1-4, to=2-4]
	\arrow[from=2-1, to=2-2]
	\arrow["\iota",from=2-2, to=2-3]
	\arrow[from=2-3, to=2-4]
	\arrow["s",shift right, curve={height=-6pt}, dashed, from=2-4, to=2-3]
	\arrow[from=2-4, to=2-5]
\end{tikzcd}
\end{equation}
The \textit{set-theoretic} section $s$ is the diagonal map given on each coordinate by $\mathbb{Z}/f\mathbb{Z}\rightarrow\mathbb{Z}/Nf\mathbb{Z}$, $1\mapsto1$. Considering just one coordinate for now for clarity, this gives rise to the decomposition:
\begin{equation}\label{eq: decomposition of 1/Nf}
    \{z:z\in \mathbb{Z}/Nf\mathbb{Z}\} = \{
    \iota(x)+s(y):x\in \mathbb{Z}/N\mathbb{Z},y\in\mathbb{Z}/f\mathbb{Z}\}.
\end{equation}
Taking $z\in \{0,\dots,Nf-1\}$, $x\in \{0,\dots,N-1\}$ and $y\in \{0,\dots,f-1\}$, the decomposition above is simply that $z = fx+y$ for a unique choice of $x$ and $y$. That is we reduce $z$ modulo $f$ to its unique representative $z-fx=y$ in $\{0,\dots,f-1\}$. This is just the Euclidean algorithm.

The issue is the section $s$, and consequent decomposition of \eqref{eq: decomposition of 1/Nf}, does not restrict to the top row of \eqref{eq: two rows SES}. However, we will show that the decomposition does hold when we project onto a coordinate plane, so that the subtlety is concentrated in the cubic terms $\sum_{\delta\in F(L)}\delta^3-\sum_{\mu\in F(V)}\sum_{\nu\in L\cap\frac{1}{N}U/U}(\mu+\nu)^3$. The final step will be to control this difference and thus show that:
\begin{equation*}
    \sum_{\delta\in F(L)} N\cdot B_{3,3}(w+z(\delta),z(\underline{x})) \equiv \sum_{\mu\in F(V)}N\cdot \sum_{\nu\in L\cap\frac{1}{N}U/U} B_{3,3}(w+z(\mu)+z(\nu),z(\underline{x})) \hspace{0.5cm}\text{ mod } 6\mathbb{Z} 
\end{equation*}
from which Proposition \ref{prop: smoothed phi0,2} will follow.
\subsubsection{Rectifying the proof strategy}
We begin with a more explicit description of the objects and maps in the commutative diagram in \eqref{eq: two rows SES}. Remember that $B_{3,3}(w,z(\underline{x}))$ is not $z(U)$-invariant where $U=\langle x_1,x_2,x_3\rangle$ so we must always be careful to refer to a \textit{specific} set of coset representatives rather than simply a quotient group.
\begin{itemize}
    \item $F(L) = \{\delta = \frac{1}{Nf}(\delta_1x_1+\delta_2x_2+\delta_3x_3):   0\leq \delta_i< Nf \text{ , }\delta\in L\}$ is a set of coset representatives for $L/U$.
    \item $F(V) = \{\mu = \frac{1}{f}(\mu_1\frac{x_1}{N}+mu_2\frac{x_2}{N}+\mu_3\frac{x_3}{N}):   0\leq \mu_i< f \text{ , }\mu\in V\}$ is a set of coset representatives for $V/\frac{1}{N}U$.
    \item $F(L\cap\frac{1}{N}U) = \{\nu =\frac{1}{N}(\nu_1x_1+\nu_2x_2+\nu_3x_3):   0\leq \nu_i< N \text{ , }\nu\in L\}$ is a set of coset representatives for $L\cap\frac{1}{N}U/U$.
\end{itemize}
\begin{remark}
    The notation $F(L\cap\frac{1}{N}U)$ is not simply meant to mimic that of $F(L)$ and $F(V)$. For our functionals $a_1$, $a_2$, and $a_3$, the set $F(L\cap\frac{1}{N}U)$ is precisely $F_{a_1,a_2,a_3}(L\cap\frac{1}{N}U)$ in the notation of Proposition \ref{prop: explicit phi0,2}.
\end{remark}
Then the vertical maps in \eqref{eq: two rows SES} are given by $\delta\mapsto (\delta_1,\delta_2,\delta_3)$ and similarly for $\mu$ and $\nu$. Note that the square on the right of \eqref{eq: two rows SES} commutes because the horizontal map $L/U\rightarrow V/\frac{1}{N}U$ corresponds to quotienting by $\frac{f}{Nf}U$, which precisely reduces the $\delta_i$ modulo $f$.

Now let us rectify our previous proof strategy. The reason why the section $s$ of \eqref{eq: two rows SES} does not restrict to $V/\frac{1}{N}U\rightarrow L/U$ is that for $\mu\in V/\frac{1}{N}U$, there is no reason to expect $s(\mu)\in L/U$; whilst $V=L+\frac{1}{N}U$, the specific coset representatives given by $F(V)$ need not lie in $L$. One sufficient condition for this to occur is for $L/U$ to contain all preimages, that is to say that $L\cap\frac{1}{N}U/U=\left(\mathbb{Z}/N\mathbb{Z}\right)^3$. This is evidently not true: the two groups have different cardinality. However, we can recover this equality by projecting first to a coordinate plane. This attempt to restrict the section motivates first applying the projection $\pi_{i,j}$ onto the $x_ix_j$-coordinate plane: $\pi_{i,j}(\delta)=(\delta_i,\delta_j)$ for the notation $\delta_i$ as above. We deduce the following commutative diagram with exact rows:
\begin{equation}\label{eq: three rows SES}
\begin{tikzcd}
	0 & {L\cap\frac{1}{N}U/U} & {L/U} & {V/\frac{1}{N}U} & 0 \\
	0 & {\pi_{i,j}(L\cap\frac{1}{N}U/U)} & {\pi_{i,j}(L/U)} & {\pi_{i,j}(V/\frac{1}{N}U)} & 0 \\
	0 & {\left(\mathbb{Z}/N\mathbb{Z}\right)^2} & {\left(\mathbb{Z}/Nf\mathbb{Z}\right)^2} & {\left(\mathbb{Z}/f\mathbb{Z}\right)^2} & 0
	\arrow[from=1-1, to=1-2]
	\arrow[from=1-2, to=1-3]
	\arrow[two heads, from=1-2, to=2-2]
	\arrow[from=1-3, to=1-4]
	\arrow[two heads, from=1-3, to=2-3]
	\arrow[from=1-4, to=1-5]
	\arrow[two heads, from=1-4, to=2-4]
	\arrow[from=2-1, to=2-2]
	\arrow["\iota",from=2-2, to=2-3]
	\arrow[equal, from=2-2, to=3-2]
	\arrow[from=2-3, to=2-4]
	\arrow[hook, from=2-3, to=3-3]
	\arrow["s", shift right, curve={height=-6pt}, dashed, from=2-4, to=2-3]
	\arrow[from=2-4, to=2-5]
	\arrow[hook, from=2-4, to=3-4]
	\arrow[from=3-1, to=3-2]
	\arrow["\iota",from=3-2, to=3-3]
	\arrow[from=3-3, to=3-4]
	\arrow["s", shift right, curve={height=-6pt}, dashed, from=3-4, to=3-3]
	\arrow[from=3-4, to=3-5]
\end{tikzcd}
\end{equation}
\begin{remark}
    It is not immediate that the middle row, the `image sequence', is exact. In general, for a commutative diagram with exact rows:
    \begin{equation}\label{eq: SL digression}
    \begin{tikzcd}
    	0 & A & B & C & 0 \\
    	0 & {A'} & {B'} & {C'} & 0
    	\arrow[from=1-1, to=1-2]
    	\arrow[from=1-2, to=1-3]
    	\arrow["a"', from=1-2, to=2-2]
    	\arrow[from=1-3, to=1-4]
    	\arrow["b"', from=1-3, to=2-3]
    	\arrow[from=1-4, to=1-5]
    	\arrow["c"', from=1-4, to=2-4]
    	\arrow[from=2-1, to=2-2]
    	\arrow[from=2-2, to=2-3]
    	\arrow[from=2-3, to=2-4]
    	\arrow[from=2-4, to=2-5]
    \end{tikzcd}
    \end{equation}
    it is not always true that the `image sequence`:
    \[\begin{tikzcd}
    	0 & a(A) & b(B) & c(C) & 0 
    	\arrow[from=1-1, to=1-2]
    	\arrow[from=1-2, to=1-3]
    	\arrow[from=1-3, to=1-4]
    	\arrow[from=1-4, to=1-5]
    \end{tikzcd}\]
    is exact. In fact, exactness is equivalent to the vanishing of the coboundary map $\partial$ in the Snake Lemma associated to \eqref{eq: SL digression}. This is equivalent to the more easily verified condition $a(A)=b(B)\cap A'$.

    In the case of \eqref{eq: three rows SES}, this means checking that $\pi_{i,j}(L\cap\frac{1}{N}U/U) = \pi_{i,j}(L/U)\cap \left(\mathbb{Z}/N\mathbb{Z}\right)^2$, which is immediate by the surjectivity on the left.
\end{remark}
As argued previously, by the $d=3$ assumption, the plane $L\cap\frac{1}{N}U/U\subset\frac{1}{N}U/U\cong(\mathbb{Z}/N)^3$ trivially intersects each coordinate axis. Thus, the projection onto any coordinate plane is bijective. Critically, this means the section $s:\left(\mathbb{Z}/f\mathbb{Z}\right)^2\rightarrow\left(\mathbb{Z}/Nf\mathbb{Z}\right)^2$ restricts to $s:\pi_{i,j}(V/\frac{1}{N}U)\rightarrow \pi_{i,j}(L/U)$, as depicted in the diagram above. In particular, the decomposition of \eqref{eq: decomposition of 1/Nf} restricts to give a similar decomposition of $\pi_{i,j}(L/U)$. 
\begin{proposition}\label{prop: decomposition after projection}
    For any pair $(i,j)$ and $\delta\in F(L)$, there exists a unique pair $(\nu,\mu)\in F(L\cap{\frac{1}{N}U)}\times F(V)$ such that $\pi_{i,j}(\delta)=\pi_{i,j}(\iota(\nu)+s(\mu))$. That is to say we can decompose $\pi_{i,j}(L/U)$ as:
    \begin{equation*}
        \{(\delta_i,\delta_j):\delta\in F(L)\} = \left\{(\iota(\nu_i)+s(\mu_i),\iota(\nu_j)+s(\mu_j)):\nu\in F\left(L\cap\frac{1}{N}U\right) ,\mu\in F(V)\right\}.
    \end{equation*}
\end{proposition}
\begin{proof}
    We can apply the decomposition of \eqref{eq: decomposition of 1/Nf} to $\pi_{i,j}(L/U)$ since $\pi_{i,j}(L\cap\frac{1}{N}U/U) = \left(\mathbb{Z}/N\mathbb{Z}\right)^2$. Essentially, we are simply observing that we can reduce both $\delta_i$ and $\delta_j$ simultaneously modulo $f$. This uniquely specifies an element $\nu\in F(L\cap\frac{1}{N}U)$ and then we will have $\pi_{i,j}(\delta-\iota(\nu))\in s(\pi_{i,j}(V/\frac{1}{N}U))$. Since a section $s$ is necessarily injective, this condition uniquely specifies $\mu\in F(V)$.
\end{proof}
\begin{corollary}\label{corollary: sums linear and quadratic}
    We have the identities:
    \begin{equation*}
        \sum_{\delta\in F(L)}\delta = \sum_{\mu\in F(V)}\sum_{\nu\in F(L\cap\frac{1}{N}U)}\mu+\nu \in V \hspace{1cm} \text{and} \hspace{1cm} \sum_{\delta\in F(L)}\delta^2 = \sum_{\mu\in F(V)}\sum_{\nu\in F(L\cap\frac{1}{N}U)}(\mu+\nu)^2\in \operatorname{Sym}^2(V).
    \end{equation*}
\end{corollary}
\begin{proof}
    Since  $x_1x_2x_3\not\in V\oplus\operatorname{Sym}^2(V)$, each term appearing above contains at most two indices. Thus, we can assume $\delta = \iota(\nu)+s(\mu)$. Then, the first identity is given by:
    \begin{align*}
        \sum_{\delta\in F(L)}\delta=\sum_{\delta\in F(L)}\sum_{i=1}^3\frac{\delta_i}{Nf}x_i=&\sum_{\mu\in F(V)}\sum_{\nu\in F(L\cap\frac{1}{N}U)}\sum_{i=1}^3\frac{s(\mu_i)+\iota(\nu_i)}{Nf}x_i
        \\ &=\sum_{\mu\in F(V)}\sum_{\nu\in F(L\cap\frac{1}{N}U)}\sum_{i=1}^3\frac{\mu_i}{f}\frac{x_i}{N}+\frac{\nu_i}{N}x_i=\sum_{\mu\in F(V)}\sum_{\nu\in F(L\cap\frac{1}{N}U)}\mu+\nu
    \end{align*}
    since $\iota(\nu_i)=f\nu_i$ and $s(\mu_i)=\mu_i$ on explicit representatives. The second identity is proved similarly.
\end{proof}
\begin{corollary}\label{corollary: cubic error term}
We have the following decomposition of $N\cdot P^{0,2}_{a_1,a_2,a_3}(w,z;L)$:
    \begin{align*}
        \sum&_{\delta\in F(L)}B_{3,3}(w+z(\delta),z(\underline{x})) -\sum_{\mu\in F_{a_1,a_2,a_3}(V)} \sum_{\nu\in F(L\cap\frac{1}{N}U)}  B_{3,3}(w+z(\mu)+z(\nu),z(\underline{x}))  \\ & = \frac{1}{z(x_1)z(x_2)z(x_3)}\left(\sum_{\delta\in F(L)}\delta^3-\sum_{\mu\in F(V)}\sum_{\nu\in F(L\cap\frac{1}{N}U)}(\mu+\nu)^3 \right).
    \end{align*}
\end{corollary}
\begin{proof}
    The leading term of $B_{3,3}(w,z(\underline{x}))$ is $\dfrac{w^3}{\prod z(x_j)}$. The result then follows immediately from Corollary \ref{corollary: sums linear and quadratic}.
\end{proof}
Combining the above with the Distribution Relation yields:
\begin{equation*}
    N\cdot P^{0,2}_{a_1,a_2,a_3}(w,z;L)-P^{0,2}_{a_1,a_2,a_3}(w,z;V)\equiv \frac{1}{z(x_1)z(x_2)z(x_3)}\left(\sum_{\delta\in F(L)}\delta^3-\sum_{\mu\in F(V)}\sum_{\nu\in F(L\cap\frac{1}{N}U)}(\mu+\nu)^3 \right)\hspace{0.5cm}\text{mod } 6\mathbb{Z}
\end{equation*}
Thus, it remains to prove that the `error term' to our proof sketch, namely the right hand side above, is in fact in $6\mathbb{Z}$. This is the final step of our proof and is proven in the next, and final, subsection.
\subsubsection{Final steps}
As in the proof of the Distribution Relation, in order to show that the constant term lay in $6\mathbb{Z}$, we required Lemma \ref{lemma: consant in 6Z}: that for $(N,6)=1$ and a plane $P$ that trivially intersects each coordinate axis, the sum $\sum_{(\nu_1,\nu_2,\nu_3)\in P}\nu_1\nu_2\nu_3\in N^2\mathbb{Z}$. We now state a stronger version of this claim.
\begin{lemma}\label{lemma: sum of a_1a_2(a_3+X)}
For $(N,6)=1$, any $X\in \mathbb{Z}$ and a plane $P\subset(\mathbb{Z}/N)^3$ that trivially intersects each coordinate axis, the following sum is divisible by $N^2$:
\begin{equation*}
    \sum_{(\nu_1,\nu_2,\nu_3)\in P}\nu_1\nu_2[\nu_3+X]_N\in N^2\mathbb{Z}
\end{equation*}
where $\nu_i\in\{0,\dots,N-1\}$.
\end{lemma}
Here $[-]_N$ denotes taking the residue class modulo $N$ in the set $\{0,\dots,N-1\}$. 
\begin{proof}
    We begin by recalling the two claims that were introduced in the proof of Lemma \ref{lemma: consant in 6Z} and which we will require again. Namely, we have for $(N,6)=1$ and a plane $P$ that trivially intersects each coordinate axis:
    \begin{equation*}
         \sum_{\nu\in P}\nu_1\nu_2\in N^2\mathbb{Z}  \hspace{0.5cm}\text{and} \hspace{0.5cm} \sum_{\nu\in P,\nu_3=0}\nu_1\nu_2\in N\mathbb{Z}
    \end{equation*}
    To these two, we add a third claim:
    \begin{claim}\label{claim: sum over level set}
        For $(N,6)=1$, a plane $P$ that trivially intersects each coordinate axis and some $c\in \{0,\dots,N-1\}$ we have $\sum_{\nu\in P,\nu_3=c}\nu_1\nu_2\in N\mathbb{Z}$.
    \end{claim}
    \begin{proof}[Proof of Claim \ref{claim: sum over level set}]
        By our assumption on $P$, the projection $\pi_3:P\twoheadrightarrow\mathbb{Z}/N$ surjects. Thus, $\pi_3^{-1}(c)=(y_1,y_2,c) +\langle (u_1,u_2,0)\rangle$ for some $u_1$, $u_2\in (\mathbb{Z}/N)^*$. Therefore, the sum is congruent to:
        \begin{equation*}
            \sum_{\nu\in P,\nu_3=c}\nu_1\nu_2 \equiv  \sum_{m=0}^{N-1}(y_1+mu_1)(y_2+mu_2) \equiv  Ny_1y_2+(u_1y_2+y_1u_2)\sum_{m=0}^{N-1}m+u_1u_2\sum_{m=0}^{N-1}m^2\equiv 0 \text{ mod } N\mathbb{Z}
        \end{equation*}
        by our assumption that $(N,6)=1$.
    \end{proof}
    As in the proof of Lemma \ref{lemma: consant in 6Z}, the claim $S=\sum_{\nu\in P}\nu_1\nu_2[\nu_3+X]_N\in N^2\mathbb{Z}$ is equivalent to $2S\in N^2\mathbb{Z}$.
    \begin{align*}
        2S &= \sum_{\nu_1,\nu_2\neq 0}\nu_1\nu_2[X+\nu_3]_N + (N-\nu_1)(N-\nu_2)[X-\nu_3]_N
        \\
        &\equiv \sum_{\nu_1,\nu_2\neq 0}\nu_1\nu_2([X+\nu_3]_N+[X-\nu_3]_N) - N(\nu_1+\nu_2)[X-\nu_3]_N\hspace{0.5cm} \text{mod }N^2 
    \end{align*}
    Since we have $[X+\nu_3]_N+[X-\nu_3]_N\equiv 2X$ modulo $N$, it follows that $[X+\nu_3]_N+[X-\nu_3]_N=  2X + NC_{\nu_3}$ for some integer $C_{\nu_3}\in\mathbb{Z}$ that depends only on $\nu_3$. Thus, we can write:
    \begin{equation*}
        2S\equiv 2X\sum_{\nu_1,\nu_2\neq 0}\nu_1\nu_2 +N\cdot\left(\sum_{\nu_1,\nu_2\neq 0}\nu_1\nu_2C_{\nu_3}-(\nu_1+\nu_2)[X-\nu_3]_N\right)\hspace{0.5cm} \text{mod }N^2 
    \end{equation*}
    We have seen that the first sum is divisible by $N^2$ and so it remains to show that:
    \begin{equation*}
        \sum_{\nu_1,\nu_2\neq 0}\nu_1\nu_2C_{\nu_3}+ \nu_3(\nu_1+\nu_2)-X(\nu_1+\nu_2)\overset{?}{\in}N\mathbb{Z}
    \end{equation*}
    Since $C_{\nu_3}$ depends only on $\nu_3$, the sum over the first term can be decomposed into the level sets of $\nu_3$, on which we can use Claim \ref{claim: sum over level set}. That $\sum_{\nu_1,\nu_2\neq 0 }\nu_1\nu_3\in N\mathbb{Z}$ was dealt with in the proof of Lemma \ref{lemma: consant in 6Z}. For the final terms, use $\pi_{1,2}:P\xrightarrow{\sim}(\mathbb{Z}/N)^2$ and so $\sum_{\nu_1,\nu_2\neq 0 }\nu_1=\sum_{\nu_1,\nu_2\neq 0 }\nu_2 = (N-1)\sum_{m=1}^{N-1}m\in N\mathbb{Z}$, which completes the proof.
\end{proof}
\begin{proposition}\label{prop: final error term}
    $\frac{N}{z(x_1)z(x_2)z(x_3)}\left(\sum_{\delta\in F(L)}\delta^3-\sum_{\mu\in F(V)}\sum_{\nu\in F(L\cap\frac{1}{N}U)}(\mu+\nu)^3\right)\in 6\mathbb{Z}$. In particular, we have:
    \begin{equation*}
        N\cdot P^{0,2}_{a_1,a_2,a_3}(w,z;L)  \equiv \sum_{\mu\in F_{a_1,a_2,a_3}(V)} \sum_{\nu\in F(L\cap\frac{1}{N}U)} N\cdot B_{3,3}(w+z(\mu)+z(\nu),z(\underline{x})) \hspace{0.5cm} \text{mod } 6\mathbb{Z}
    \end{equation*}
\end{proposition}
\begin{proof}
    From Proposition \ref{prop: decomposition after projection}, we can in fact deduce more than Corollary \ref{corollary: cubic error term}. Not only can we decompose the sums over the linear and quadratic terms as in Corollary \ref{corollary: sums linear and quadratic}, but we can do the same for those terms in $\operatorname{Sym}^3{V}$ that contain at most two indices. In other words, the coefficients of $x_i^3$ and $x_i^2x_j$ agree in the two sums $\sum_{\delta\in F(L)}\delta^3$ and $\sum_{\mu\in F(V)}\sum_{\nu\in L\cap\frac{1}{N}U/U}(\mu+\nu)^3$. Thus, the difference is:
    \begin{equation*}
        N\cdot \text{Coeff}\left(x_1x_2x_3;\sum_{\delta\in F(L)}\delta^3-\sum_{\mu\in F(V)}\sum_{\nu\in F(L\cap\frac{1}{N}U)}(\mu+\nu)^3\right)= \frac{3!N}{(Nf)^3}\left(\sum_{\delta\in F(L)}\prod_{i=1}^3\delta_i-\sum_{\mu\in F(V)}\sum_{\nu\in F(L\cap\frac{1}{N}U)}\prod_{i=1}^3(\mu_i+f\nu_i)\right)
    \end{equation*}
    since $\nu=\sum_{i=1}^3\frac{\nu_i}{N}x_i=\sum_{i=1}^3\frac{f\nu_i}{Nf}x_i$. Thus, proving the proposition is equivalent to showing that:
    \begin{equation}\label{eq: difference in cubics}
        \sum_{\delta\in F(L)}\prod_{i=1}^3\delta_i-\sum_{\mu\in F(V)}\sum_{\nu\in L\cap\frac{1}{N}U/U}\prod_{i=1}^3(\mu_i+f\nu_i)\in N^2f^3\mathbb{Z}
    \end{equation}
    We begin by strengthening Proposition \ref{prop: decomposition after projection} so as to explicitly describe $F(L)$ in terms of $F(V)$ and $F(L\cap\frac{1}{N}U)$. Writing $p:L/U\twoheadrightarrow V/\frac{1}{N}U$, given on representatives in $\left(\mathbb{Z}/Nf\mathbb{Z}\right)^3$ by reduction mod $f$, we have $F(L)=\sqcup_{\mu\in F(V)} p^{-1}(\mu)$. Then the content of Proposition \ref{prop: decomposition after projection} tells us that for each $\mu\in F(V)$, there is a unique $\delta\in F(L)$ such that $\pi_{1,2}(\delta)=s(\pi_{1,2}(\mu))$: that is to say that $\delta\equiv\mu$ mod $f$ and moreover $\delta_1 = \mu_1$, $\delta_2 = \mu_2$. Let us write $D_{\mu}$ for the difference $\frac{1}{f}(\delta_3-\mu_3)$ so that $\delta = \mu + f(0,0,D_{\mu})$. Thus, we can explicitly describe $p^{-1}(\mu)$:
    \begin{equation*}
        p^{-1}(\mu)=\delta+\iota(F(L\cap\frac{1}{N}U)) = \{(\mu_1+f\nu_1,\mu_2+f\nu_2,\mu_3+f[\nu_3+D_{\mu}]_N): \nu\in F(L\cap\frac{1}{N}U) \}.
    \end{equation*}
    And so the difference in \eqref{eq: difference in cubics} becomes:
    \begin{equation*}
        \sum_{\mu\in F(V)}\sum_{\nu\in F(L\cap\frac{1}{N}U)}(\mu_1+f\nu_1)(\mu_2+f\nu_2)\underbrace{(\mu_3+f[\nu_3+D_{\mu}]_N-(\mu_3+f\nu_3))}_{=f([\nu_3+D_{\mu}]_N-\nu_3)}.
    \end{equation*}
    \begin{claim}
    We have the following identities:
    \begin{equation*}
        \sum_{\nu\in F(L\cap\frac{1}{N}U)}[\nu_3+D_{\mu}]_N-\nu_3 =\sum_{\nu\in F(L\cap\frac{1}{N}U)}\nu_1\cdot([\nu_3+D_{\mu}]_N-\nu_3)= \sum_{\nu\in F(L\cap\frac{1}{N}U)}\nu_2\cdot([\nu_3+D_{\mu}]_N-\nu_3)=0.
    \end{equation*}
    \end{claim}
    \begin{proof}
        Since $\pi_{i,j}(L\cap\frac{1}{N}U/U)=\left(\mathbb{Z}/N\mathbb{Z}\right)^2$, as long as only two indices appear, we can reindex $\nu\in F(L\cap\frac{1}{N}U)$ so that, for example, $\sum_{\nu\in F(L\cap\frac{1}{N}U)}\nu_2\cdot[\nu_3+D_{\mu}]_N=\sum_{\nu\in F(L\cap\frac{1}{N}U)}\nu_2\nu_3$.
    \end{proof}
    Thus, the difference simplifies to:
    \begin{equation}\label{eq: final N^2}
        f^3\cdot \sum_{\mu\in F(V)}\sum_{\nu\in F(L\cap\frac{1}{N}U)}\nu_1\nu_2([\nu_3+D_{\mu}]_N-\nu_3)\overset{?}{\in}N^2f^3\mathbb{Z}
    \end{equation}
    where the $\nu_i\in\{0,\dots,N-1\}$. As we have seen in Lemma \ref{lemma: consant in 6Z}, we know that $\sum_{\nu\in F(L\cap\frac{1}{N}U)}\nu_1\nu_2\nu_3\in N^2\mathbb{Z}$. That $\sum_{\nu\in F(L\cap\frac{1}{N}U)}\nu_1\nu_2[\nu_3+D_{\mu}]_N\in N^2\mathbb{Z}$ is given by Lemma \ref{lemma: sum of a_1a_2(a_3+X)}. This completes the proof of the Proposition.
\end{proof}
Finally, we collate all our results to complete the proof of Proposition \ref{prop: smoothed phi0,2}.
\begin{proof}[Proof of Proposition \ref{prop: smoothed phi0,2}]
    By the explicit description of $\phi^{0,2}$ in Proposition \ref{prop: explicit phi0,2}, we know that:
    \begin{equation*}
         \widetilde{\phi^{0,2}_{a,b,c}}(w,z)=e^{-\frac{2\pi i}{3!} \left(N\cdot P^{0,2}_{a,b,c}(w,z;L)- P^{0,2}_{a,b,c}(w,z;V) \right)}.
    \end{equation*}
    Now we have:
    \begin{align*}
        N\cdot P^{0,2}_{a,b,c}(w,z;L) = &\sum_{\delta\in F_{a,b,c}(L)} N\cdot B_{3,3}(w+z(\delta),z(\underline{x})) \\ \equiv & \sum_{\mu\in F_{a,b,c}(V)}N\cdot \sum_{\nu\in L\cap\frac{1}{N}U/U} B_{3,3}(w+z(\mu)+z(\nu),z(\underline{x})) \hspace{0.5cm}\text{ mod } 6\mathbb{Z} \\ \equiv & \sum_{\mu\in F_{a,b,c}(V)} B_{3,3}(N\cdot(w+z(\mu)),z(x)) = P^{0,2}_{a,b,c}(w,z;V) \hspace{0.5cm}\text{ mod } 6\mathbb{Z} 
    \end{align*}
    where the second congruence follows from Proposition \ref{prop: final error term} and the third congruence follows from the Distribution Relation of Proposition \ref{prop: Bernoulli DR}. And so we conclude $\widetilde{\phi^{0,2}_{a,b,c}}(w,z)=1$.
\end{proof}
In particular, we have proven the following result.
\begin{corollary}\label{corollary: proof with d=3}
    Let $\triple{h}$ and $\triple{h'}$ be two admissible triples for the pair $(V,L)$, such that $h'\equiv h$ modulo $L$. If $d(h,h')=3$, then:
    \begin{equation*}
        \widetilde{\Phi}\triple{h}=\widetilde{\Phi}\triple{h'}.
    \end{equation*}
\end{corollary}
\newpage
\section{Reduction to the $d=3$ case}\label{sec: reduction}
In this section, we fix the pair $\langle z\rangle $ and $g$ that satisfy $g\cdot \langle z\rangle = \langle z\rangle$. To complete the proof of Theorem \ref{theorem: main theorem in main body}, we proceed as follows.  Given two admissible triples $\triple{h}$, $\triple{h'}$, satisfying $h\equiv h'$ modulo $L$, we will find a third admissible triple $\triple{h''}$, also satisfying $h''\equiv h'\equiv h $ modulo $L$, such that $d(h,h'')=d(h',h'')=3$. Then by Corollary \ref{corollary: proof with d=3}:
\begin{equation*}
    \frac{\widetilde{\Phi}\triple{h}}{\widetilde{\Phi}\triple{h'}}=\frac{\widetilde{\Phi}\triple{h}}{\widetilde{\Phi}\triple{h''}}\cdot \frac{\widetilde{\Phi}\triple{h''}}{\widetilde{\Phi}\triple{h'}}=1.
\end{equation*}
Before we outline the proof strategy, let us better understand the membership of the sets $\mathcal{V}$ and $\mathcal{W}$:
\begin{align*}
    \mathcal{V} &=  \{a\in\Hom(V,\mathbb{Z}): V = L +\ker(a)\} \\
    \mathcal{W} &=  \{(a,b)\in \mathcal{V}^2: V = L +\ker(a,b)\} 
\end{align*}
We record the following criterion for when $-\otimes\mathbb{Z}/N$ commutes with kernels for an arbitrary morphism $\mathbb{Z}^n\xrightarrow{\phi}\mathbb{Z}^m$.
\begin{lemma}
    Consider a morphism $\phi:\mathbb{Z}^n\rightarrow\mathbb{Z}^m$. Then, we have the equivalences:
    \begin{equation*}
        \ker(\phi\otimes\mathbb{Z}/N)=\ker(\phi)\otimes\mathbb{Z}/N\iff \operatorname{Tor}^{\mathbb{Z}}_1(\coker(\phi),\mathbb{Z}/N) = \coker(\phi)[N]=0 \iff \operatorname{im}(\phi\otimes \mathbb{Z}/N) = (\mathbb{Z}/N)^{\operatorname{rank}_{\mathbb{Z}}(\phi)} 
    \end{equation*}
\end{lemma}
\begin{proof}
    Tensor each of the short exact sequences $0\rightarrow\ker(\phi)\rightarrow\mathbb{Z}^n\rightarrow\text{im}(\phi)\rightarrow 0$ and $0\rightarrow\text{im}(\phi)\rightarrow\mathbb{Z}^m\rightarrow\coker(\phi)\rightarrow 0$ with $\mathbb{Z}/N$. Since $\text{im}(\phi)$ is a free $\mathbb{Z}$-module, we then use $\operatorname{Tor}^{\mathbb{Z}}_1(\text{im}(\phi),\mathbb{Z}/N)=0$.
\end{proof}

\begin{lemma}\label{lemma: sufficient condition}
    Let $a$, $b\in V^{\vee}$.
    \begin{enumerate}
        \item If $a\otimes\mathbb{Z}/N\in (V/NV)^{\vee}$ is primitive, then $a\in \mathcal{V}$ if and only if $ a\otimes\mathbb{Z}/N|_{L/NV}\in (L/NV)^{\vee} $ is primitive.
        \item The following is a sufficient condition for $(a,b)\in \mathcal{W}$:
        \begin{equation*}
            \langle a\otimes\mathbb{Z}/N|_{L/NV},b\otimes\mathbb{Z}/N|_{L/NV}\rangle = (L/NV)^{\vee}\implies (a,b)\in \mathcal{W}.
        \end{equation*}
    \end{enumerate}
\end{lemma}
\begin{proof}
    We have the easily verifiable equivalences for $a\in \mathcal{V}$:
    \begin{equation*}
        V = L+\ker(a)\iff V/NV = L/NV +(\ker(a)+NV)/NV \iff V/NV = L/NV +\ker(a)\otimes \mathbb{Z}/N.
    \end{equation*}
    By the assumption $a\otimes\mathbb{Z}/N\in (V/NV)^{\vee}$ is primitive, we know $\ker(a\otimes\mathbb{Z/N}) = \ker(a)\otimes\mathbb{Z}/N$. In particular,
    \begin{equation*}
        a\in\mathcal{V}\iff V/NV = L/NV +\ker(a\otimes\mathbb{Z}/N)\iff a\otimes\mathbb{Z}/N|_{L/NV} \text{ is primitive}.
    \end{equation*}
    Now consider the second part of the claim. As above, $(a,b)\in \mathcal{W}$ if and only if $V/NV = L/NV +\ker(a,b)\otimes \mathbb{Z}/N$.  The hypothesis implies that $\operatorname{rank}_{\mathbb{Z}}(a,b)=2$. Then, by the Lemma above:
    \begin{equation*}
        \ker((a,b)\otimes\mathbb{Z}/N) =\ker(a,b)\otimes\mathbb{Z}/N \iff \text{im}((a,b)\otimes\mathbb{Z}/N) =  (\mathbb{Z}/N)^2 \iff \langle a\otimes\mathbb{Z}/N,b\otimes\mathbb{Z}/N\rangle \cong (\mathbb{Z}/N)^2.
    \end{equation*}
    Since our hypothesis $\langle a\otimes\mathbb{Z}/N|_{L/NV},b\otimes\mathbb{Z}/N|_{L/NV}\rangle = (L/NV)^{\vee}$ implies $\langle a\otimes\mathbb{Z}/N,b\otimes\mathbb{Z}/N\rangle \cong (\mathbb{Z}/N)^2$, we deduce $\ker((a,b)\otimes\mathbb{Z}/N) =\ker(a,b)\otimes\mathbb{Z}/N $ and then use:
    \begin{align*}
        V/NV = L/NV +\ker((a,b)\otimes\mathbb{Z}/N) &\iff (a,b)\otimes\mathbb{Z}/N:L/NV\xrightarrow{\sim}(\mathbb{Z}/N)^2 \\&\iff \langle a\otimes\mathbb{Z}/N|_{L/NV},b\otimes\mathbb{Z}/N|_{L/NV}\rangle = (L/NV)^{\vee}.
    \end{align*}
    This completes the proof of the Lemma.
\end{proof}
\begin{remark}
    By taking the annihilator of these submodules, we can state our conditions in terms of submodules of $L/NV$ rather than the dual. For instance:
    \begin{equation}\label{eq: submodules of dual}
        \langle a\otimes\mathbb{Z}/N|_{L/NV},b\otimes\mathbb{Z}/N|_{L/NV}\rangle = (L/NV)^{\vee}\iff \ker(a\otimes\mathbb{Z}/N|_{L/NV})\cap\ker(b\otimes\mathbb{Z}/N|_{L/NV}) = \{0\}.
    \end{equation}
\end{remark}
We now outline the proof strategy. Write $a\sim b$ for the relation on $\mathcal{V}\times \mathcal{V}$ defined by $(a,b)\in\mathcal{W}$: this is evidently reflexive, symmetric and $g$-equivariant. Fix $h$ and $h'\in L\otimes\mathbb{Q}$ such that $gh\equiv h\equiv h'\equiv gh'$ modulo $L$ and both $\langle h
\rangle_{\mathbb{Q}}$ and $\langle h'\rangle_{\mathbb{Q}}$ generate $V/L$. From these vectors, we obtain primitive functionals $a$, $a'\in\mathcal{V}$. Among the $a''$ arising from eligible $h''$, we seek one for which each of the edges in the following diagram corresponds to a pair of functionals in $\mathcal{W}$.
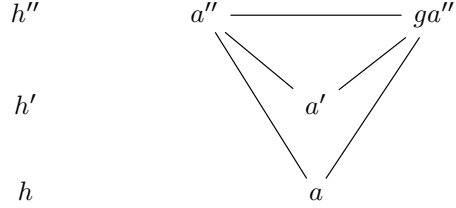
\[\begin{tikzcd}
	{h''} && {a''} && {ga''} \\
	{h'} &&& {a'} \\
	h &&& a
	\arrow[no head, from=1-3, to=1-5]
	\arrow[no head, from=2-4, to=1-3]
	\arrow[no head, from=2-4, to=1-5]
	\arrow[no head, from=3-4, to=1-3]
	\arrow[no head, from=3-4, to=1-5]
\end{tikzcd}\]
By the assumption $\langle h''\rangle_{\mathbb{Q}}$ generates $V/L$, we necessarily have $a''\sim ga''$. Using the $g$-equivariance, our objective becomes finding an $a''$ such that:
\begin{equation*}
    a'' \sim a,g^{-1}a,a',g^{-1}a'.
\end{equation*}
Thus, we obtain $4$ lines $\langle a\rangle$, $\langle g^{-1}a\rangle$, $\langle a'\rangle $, $\langle g^{-1}a'\rangle \subset  (L/NV)^{\vee}$; as in Definition \ref{def: line and plane}, a line here refers to a submodule isomorphic to $\mathbb{Z}/N$. Using the sufficient condition given in Lemma \ref{lemma: sufficient condition} and its reformulation in \eqref{eq: submodules of dual}, the strategy will be to find an $h''$ such that $\ker(a''\otimes\mathbb{Z}/N|_{L/NV}) = (\langle h'',g^{-1}h''\rangle_{\mathbb{Q}}\cap L)/ NV\subset L/NV $ has trivial intersection with each of the four lines (or rather their annihilators) above. 

We do this by proving any line in $L/NV$ comes from an eligible $h''$ under this construction and that there are sufficiently many lines: for the former, we control the action of $g$ modulo $N$ using $(N, \disc(\chi_g))=1$; for the latter, we require $(N,6)=1$. Since the argument in the prime case is simpler, we begin with this case.
\subsection{$N$ is prime}
Suppose $N$ is a prime and write $R$ for the set of eligible $h''$, that is:
\begin{equation*}
    R\defeq \{h'': h''\equiv h \text{ mod } L\text{, } \langle h''\rangle_{\mathbb{Q}} \text{ generates } V/L\} \subset h +L. 
\end{equation*}
Let us express the association of $h''$ to a line in $L/NV$ as a map:
\begin{equation*}
    R\xrightarrow{\Psi} \mathbb{P}(L/NV) \quad h''\longmapsto\ker(a''\otimes\mathbb{Z}/N|_{L/NV})=(\langle h'',g^{-1}h''\rangle_{\mathbb{Q}}\cap L)/ NV .
\end{equation*}
Our objective is to show there exists $h''\in R$ such that $\Psi(h'')\in \mathbb{P}(L/NV) - \{l_i\}_{i=1}^4$ for four lines $l_i\in \mathbb{P}(L/NV)$; over a field, two lines intersect trivially if and only if they are distinct. For $(N,\disc(\chi_g))=1$, we will deduce $\Psi(R)  =\mathbb{P}(L/NV) $ surjects. Then, assuming $(N,6)=1$, it follows that $|\mathbb{P}(L/NV)| = N+1 > 4$.

To describe $R$, we define the following auxiliary constants for $h$: set $m=\operatorname{order}_{L\otimes\mathbb{Q}/\mathbb{Z}}(h)$ and $n$ be the largest integer such that $mh\in N^nV$, so that $mh\in N^nV - N^{n+1}V$. Then we have:
\begin{center}
    $\langle h \rangle_{\mathbb{Q}}$ generates $V/L$ if and only if $\frac{mh}{N^n}\in V - L$.
\end{center}
\begin{lemma}
    Suppose that $\langle h\rangle_{\mathbb{Q}}$ generates $V/L$. Then:
    \begin{enumerate}
        \item We always have $n\geq 1$. If $N|m$, then $n=1$.
        \item We have $\operatorname{order}_{L\otimes\mathbb{Q}/\mathbb{Z}}(h)=\operatorname{order}_{NV\otimes\mathbb{Q}/\mathbb{Z}}(h)$.
    \end{enumerate}
\end{lemma}
\begin{proof}
    Both claims follow from the following easily verified fact: if $u\in L$ is a lattice element whose rational span generates $V/L$, then $u\in N^kL\iff u \in N^{k+1}V$ for all $k\geq 0$.
\end{proof}
With these auxiliary constants, we describe a subset of $R$, splitting into the cases $N$ coprime to $m$ or not.
\begin{lemma}\label{lemma: subsets of R}
    If $N|m$, then $R = h+L$. If $(N,m)=1$, then $R\supset h+N^nL$.
\end{lemma}
\begin{proof}
    When $N|m$, we know $n=1$ and so $\frac{m}{N^n}\in \mathbb{Z}$. Thus, in both cases ($h''$ in $h+L$ or $h+N^nL$ respectively) we have $\frac{mh''}{N^n}\equiv \frac{mh}{N^n}\in V/L$. 
\end{proof}

\begin{proposition}\label{prop: image of Psi}
    Suppose $(N, \disc(\chi_g))=1$. Then, $\Psi(R) = \mathbb{P}(L/NV)$. 
\end{proposition}
    Thus, any line in $L/NV$ can be expressed as $\ker(a''\otimes\mathbb{Z}/N|_{L/NV})$ for $a''\in \mathcal{V}$ associated to $h''\in R$.
\begin{corollary}\label{corollary: d=3}
    Suppose $(N, 6\cdot \disc(\chi_g))=1$. Then, for any $h$, $h'\in R$, there exists an $h'' \in R$ such that $d(h,h'')=d(h',h'')=3$.
\end{corollary}
\begin{proof}
    From the Proposition above, we know $|\Psi(R)|=N+1 > 4$ and so given any $4$ lines in $\mathbb{P}(L/NV)$, say arising from $a$, $g^{-1}a$, $a'$, $g^{-1}a'$, we can always find an $h''$ such that $\ker(a''\otimes\mathbb{Z}/N|_{L/NV})$ is distinct from all $4$ of these lines in $\mathbb{P}(L/NV)$. Such an $a''$ will satisfy $a''\sim a$, $g^{-1}a$, $a'$, $g^{-1}a'$ and thus we will have $d(h,h'')=d(h',h'')=3$.
\end{proof}
\begin{corollary}
    Suppose $(N, 6\cdot \disc(\chi_g))=1$. Then, for any $h$, $h'\in R$, we have $\widetilde{\Phi}\triple{h} = \widetilde{\Phi}\triple{h'}$.
\end{corollary}
We prove Proposition \ref{prop: image of Psi} by showing that, restricted to the subsets of $R$ described above, we can explicitly compute $\Psi(h'')$ using an endomorphism $\psi\in \End(NV\otimes\mathbb{Q}/\mathbb{Z})$. Choosing $\lambda\in \mathbb{Z}$ such that $\lambda = g^{-1} \in \Aut(V/L)\cong(\mathbb{Z}/N)^*$, we define $\psi = -\lambda+g^{-1}\in \End(NV\otimes\mathbb{Q}/\mathbb{Z})$. By construction, $\psi|_{V/NV}:V/NV\rightarrow L/NV\hookrightarrow V/NV$.
\begin{lemma}\label{lemma: sufficient automorphism}
     Suppose $(N, \disc(\chi_g))=1$. Then, $\psi|_{L/NV}\in \Aut(L/NV)$.
\end{lemma}
In particular, any coset of $L/NV$ in $V/NV$ surjects onto $L/NV$ under $\psi$.
\begin{proof}
    Firstly, since $\det(g)=1$, we note that $\disc(\chi_{g^{-1}}) = \disc(\chi_g) $. Now observe:
    \begin{equation*}
        \chi_{g^{-1}}(x)\equiv \chi_{g^{-1}}(x; V/NV) = \chi_{g^{-1}}(x;V/L)\cdot \chi_{g^{-1}}(x;L/NV)=(x-\lambda)\cdot \chi_{g^{-1}}(x;L/NV)\in(\mathbb{Z}/N)[x].
    \end{equation*}
    By assumption $(N, \disc(\chi_g))=1$, we know $\chi_{g^{-1}}\in (\mathbb{Z}/N)[x]$ cannot have a double root. In particular, $\chi_{g^{-1}}(\lambda;L/NV)\neq 0$ and so $-\lambda+g^{-1}\in \Aut(L/NV)$.
\end{proof}
\begin{proof}[Proof of Proposition \ref{prop: image of Psi}]
    Consider the cases $N$ coprime to $m$ or not separately. First, when $(N,m)=1$, we have $R\supset h + N^nL$. Then:
    \begin{equation*}
        h''\in h+N^nL\implies \frac{mh''}{N^n}\in \frac{mh}{N^n}+mL\longmapsto \frac{mh}{N^n}+\frac{L}{NV}\subset \frac{V}{NV}
    \end{equation*}
    where $(N,m)=1$ to deduce $mL/NV = L/NV$. Under the map $\psi$, this coset surjects onto $L/NV$. In other words, any element of $L/NV$ can be expressed as $\psi(\frac{mh''}{N^n})$ for $h''\in R$. By construction $\psi(\frac{mh''}{N^n})\in \langle h'',g^{-1}h''\rangle_{\mathbb{Q}}$, so that when $\psi(\frac{mh''}{N^n})\neq 0 \in L/NV$, we have $\Psi(h'')  = \langle \psi(\frac{mh''}{N^n})\rangle\in\mathbb{P}(L/NV)$. It follows that $\Psi(R) = \mathbb{P}(L/NV)$. 

    It remains to consider the case $N|m$, for which $n=1$. If $v_N(m)=1$, then the above proof works verbatim as $\frac{m}{N}L/NV = L/NV$. However, when $v_N(m)\geq 2$, we have $\frac{mh''}{N}\equiv \frac{mh}{N}\in V/NV$ for all $h''\in R = h+L$ and so we cannot repeat the proof above. Instead, we shall use the following observation: when $N|m$, we have $\frac{mh}{N}\equiv g^{-1}\frac{mh}{N}\in V/L$, which implies that $\lambda\equiv 1\in (\mathbb{Z}/N)^*$. If we then consider $\psi = - 1 +g^{-1}\in \End(NV\otimes \mathbb{Q}/\mathbb{Z})$, this restricts to $ -\lambda+g^{-1}\in \End(V/NV)$ and further satisfies $\psi(h)=-h+g^{-1}h\in L/NV$. Then, for $h''\in h+L$, we compute $\psi(h'')\in \psi(h)+\psi(L/NV) = \psi(h)+L/NV = L/NV$. As above, we thus deduce that $\Psi(R) = \mathbb{P}(L/NV)$. 
\end{proof}
\subsection{General case}
The strategy in the general case remains the same: to find an $h''$ such that $\ker(a''\otimes\mathbb{Z}/N|_{L/NV}) = (\langle h'',g^{-1}h''\rangle_{\mathbb{Q}}\cap L)/ NV\subset L/NV $ has trivial intersection with any four lines in $L/NV$. Again, we do so by proving that any line in $L/NV$ comes from an eligible $h''$ and that there are sufficiently many lines in $L/NV$.

As before, we write $R$ for the set of eligible $h''$:
\begin{equation*}
    R\defeq \{h'': h''\equiv h \text{ mod } L\text{, } \langle h''\rangle_{\mathbb{Q}} \text{ generates } V/L\} \subset h +L. 
\end{equation*}
To make the condition $\langle h\rangle_{\mathbb{Q}}$ generates $V/L$ more explicit, we pick a basis for $L=\langle e_1,e_2,e_3\rangle$ such that $V = \langle e_1,e_2,\frac{e_3}{N}\rangle$. We write $[-]_L$ and $[-]_V$ for the isomorphisms with $\mathbb{Z}^3$ defined by these bases.
\begin{lemma}
    For $x\in L$, with coordinates $[x]_L= (x_1,x_2,x_3)$, the rational span $\langle x\rangle _{\mathbb{Q}}$ generates $V/L$ if and only if $v_p(Nx_3)\leq v_p(x_1)$, $v_p(x_2)$ for all $p|N$. Assuming this condition, an explicit generator in the rational span is given by $\frac{x}{N\prod p^{v_p(x_3)}}\in V$.
\end{lemma}
\begin{proof}
    We know $[x]_V = (x_1,x_2,Nx_3)$. The first step is to produce a primitive vector in $V/NV$ in the rational span. So let $n_p =\min(v_p(x_1),v_p(x_2),v_p(Nx_3))$ for each $p|N$. Then $\frac{x}{\prod p ^{n_p}}\in V/NV$ is primitive so $(\langle x\rangle _{\mathbb{Q}}\cap V)/NV = \langle \frac{x}{\prod p ^{n_p}}\rangle \cong \mathbb{Z}/N$. Reducing this line mod $L$, we obtain the subgroup generated by $\frac{Nx_3}{\prod p^{n_p}}\in V/L\cong \mathbb{Z}/N$. This generates $V/L$ if and only if $(\frac{Nx_3}{\prod p^{n_p}},N)=1$. This is equivalent to $n_p = v_p(Nx_3)$ for all $p|N$, which by the definition of $n_p$, is precisely the condition given in the Lemma above.
\end{proof}
We define $m=\operatorname{order}_{L\otimes\mathbb{Q}/\mathbb{Z}}(h)$; as before, one can check this equals $m=\operatorname{order}_{NV\otimes\mathbb{Q}/\mathbb{Z}}(h)$. Writing $[mh]_L = (x_1,x_2,x_3)$ and letting $t = \prod_{p|N}p^{v_p(x_3)}$, we have:
\begin{equation*}
    \langle h\rangle_{\mathbb{Q}}\text{ generates }V/L\iff \frac{mh}{Nt}\in V \text{ generates }V/L.
\end{equation*}
\begin{remark}
    Despite the presence of $t$ in the denominator, we have:
    \begin{equation*}
        v_p\left(\frac{m}{Nt}\right) = \begin{cases}
            <0: v_p(m)< v_p(N)
            \\0: v_p(m)= v_p(N)
            \\>0: v_p(m)> v_p(N)
        \end{cases}
    \end{equation*}
    for all $p$. The above is tautologous if $v_p(N) = 0$, so assume $p|N$. The only nontrivial case to check is when $v_p(m)\geq v_p(N)\geq 1$. But when $v_p(m)\geq 1$, we have $v_p(t)=v_p(x_3)=0$.
\end{remark}
We again use the endomorphism $\psi=-\lambda +g^{-1}\in \End(NV\otimes\mathbb{Q})$ for some $\lambda\in \mathbb{Z}$ satisfying $\lambda \equiv g^{-1}\in \Aut(V/L)\in (\mathbb{Z}/N)^* $. As before, we have $\psi:V/NV \rightarrow L/NV \hookrightarrow V/NV$.
\begin{lemma}
    For $(N,\disc(\chi_g))=1$, we have $\psi|_{L/NV}\in \Aut(L/NV)$.
\end{lemma}
\begin{proof}
    Since $\End(L/NV) = M_2(\mathbb{Z}/N)$, the assertion is equivalent to $\det(\psi) = \chi_{g^{-1}}(\lambda;L/NV)\in (\mathbb{Z}/N)^* \iff \chi_{g^{-1}}(\lambda;L/(pL+NV))\neq 0 \in \mathbb{F}_p$ for all $p|N$. Then, by our previous argument in Lemma \ref{lemma: sufficient automorphism}, we know that a sufficient condition is given by $p\nmid \disc(\chi_g^{-1}(x))=\disc(\chi_g)$ for each $p|N$.
\end{proof}
Rather than splitting into cases as we did when $N$ is prime, we will give one construction which works for all $N$. However, we will still use the insight that was the rationale for splitting into cases: namely, that we can deduce extra information about $\lambda$ when $v_p(m)\geq v_p(N)$. As shown above, for $v_p(m)\geq v_p(N)$, we have $\frac{m}{Nt}\in \mathbb{Z}_p$. In particular, for $v_p(m)\geq v_p(N)$, we have $\frac{mh}{Nt}\equiv g^{-1}\frac{mh}{Nt}\in V/L\otimes\mathbb{Z}_p$ so that $\lambda\equiv 1$ modulo $p^{v_p(N)}$. If we choose $\lambda$ so that it satisfies the stronger condition $\lambda \equiv 1$ modulo $p^{v_p(m)}$, then $\psi(h) = -h+g^{-1}\in L/NV\otimes\mathbb{Z}_p$ for $p$ such that $v_p(m)\geq v_p(N)$; recall that $h\in \frac{1}{p^{v_p(m)}}NV\otimes\mathbb{Z}_p$.

We now define the following constants:
\begin{align*}
    c &= \prod_{p|N}p^{v_p(x_3)}\cdot\max(1,p^{v_p(N)-v_p(m)})\in \mathbb{Z}
    \\ r&= \prod_{p|m,p\nmid N}p^{v_p(m)}\cdot c^{-1}\in c^{-1}\mathbb{Z}
\end{align*}
The role of these constants is clarified in the next Lemma, which is a straightforward computation:
\begin{lemma}
    The constants $c$ and $r$ are the unique positive rational numbers that satisfy:
    \begin{align*}
        c\mathbb{Z}_p=\begin{cases}
            \mathbb{Z}_p&:v_p(m)\geq v_p(N) \\
            \frac{Nt}{m}\mathbb{Z}_p&: v_p(m)< v_p(N)
        \end{cases}\quad\quad\quad 
        r\mathbb{Z}_p  = \begin{cases}
            \mathbb{Z}_p&:v_p(m)\geq v_p(N)\geq 1 \\
            \frac{m}{Nt}\mathbb{Z}_p&: v_p(m)< v_p(N), v_p(N)\neq 0
            \\
            m\mathbb{Z}_p&: v_p(N)=0
        \end{cases}
    \end{align*} 
    for all primes $p$. In particular, $h+cL\subset R$, $rcL/NV = L/NV$, and $\psi(rh)\in L/NV$.
\end{lemma}
\begin{proof}
    We omit the $p$-adic valuation verifications and prove the final claims. 

    Since $c\in \mathbb{Z}$, we know that $h+cL\subset h+L$ and therefore, it remains to verify that for each $h''\in h+cL$, its rational span generates the quotient. Indeed, we have:
    \begin{equation*}
        \frac{m}{Nt}h''\in \frac{m}{Nt}h+\frac{m}{Nt}cL\subset \frac{m}{Nt}h+L\longmapsto \frac{m}{Nt}h \in V/L.
    \end{equation*}
    Explicitly, $\frac{m}{Nt}c = \prod_{p|m,p\nmid N}p^{v_p(m)-v_p(N)}\in \mathbb{Z}$.
    
    The second claim follows immediately from $(rc,N)=1$.

    For the final claim, we must show that $\psi(rh)\in L/NV\otimes\mathbb{Z}_p$ for each prime $p$. We will use $\sim$ to denote up to multiplication by $\mathbb{Z}_p^*$. If $v_p(N)=0$, then $rh\sim  mh=0\in L/NV\otimes\mathbb{Z}_p$. Now assume $p|N$. If $v_p(m)<v_p(N)$, then $rh\sim \frac{m}{Nt}h\in V/NV\otimes\mathbb{Z}_p$ and so the result follows from $\psi:V/NV\rightarrow L/NV$. The final case is when $v_p(m)\geq v_p(N)$, so that $rh\sim h\in \frac{1}{m}NV/NV\otimes\mathbb{Z}_p$. By our assumption that $\lambda\equiv 1\in (\mathbb{Z}/p^{v_p(m)})^*$ in this case, it follows that $\psi(rh)\sim -h+g^{-1}h\in L/NV$. 
\end{proof}
Recall that a line in $L/NV$ is a submodule isomorphic to $\mathbb{Z}/N$.
\begin{proposition}
    Suppose $(N,\disc(\chi_g))=1$. Then, for any line $l$ in $L/NV$, there exists an $h''\in R$ such that $(\langle h'',g^{-1}h''\rangle_{\mathbb{Q}}\cap L)/NV = l$. If we further suppose that $(N,6)=1$, then for any pair $h$, $h'\in R$, there exists $h''\in R$ such that $d(h,h'')=d(h',h'')=3$.
\end{proposition}
\begin{proof}
    Combining all the above, for $h''\in h+cL$, we know that $-r\lambda h''+ rg^{-1}h''  = \psi(rh'')\in \psi(rh+rcL/NV) = \psi(rh)+\psi(L/NV)=L/NV$. In particular, we can find an $h''\in R$ whose associated line in $L/NV$ is $l$.

    Why do we need $(N,6)=1$? Given four lines in $L/NV$, we need there to exist at least one other line that intersects trivially with each of them. We know that:
    \begin{align*}
        l_1\cap l_2=\{0\}&\iff \det(l_1\wedge l_2)\in (\mathbb{Z}/N)^*\iff \det(l_1\wedge l_2)\in (\mathbb{F}_p)^* \text{ for all }p|N \\ &\iff \text{ the reductions of $l_1$ and $l_2$ are distinct mod $p$ for all $p|N$}.
    \end{align*}
    In particular, if we denote $p_{min}$ to be the smallest prime dividing $N$, then we need $p_{min}+1>4\iff p_{min}\geq 5$. Thus, we require $(N,6)=1$.
\end{proof}
\newpage
\section{Cohomological Framework}
In this section, we reflect on the cohomological aspects of our proof. Firstly, we provide a homological interpretation of the assumptions placed on the admissible triple $(h,\langle z\rangle,g)$. Secondly, we give a new evaluation procedure which has more natural cohomological properties. 
\subsection{Cohomological Summary}
Recall the setup is a pair of rank $3$ lattices $(V,L)$ with cyclic quotient $V/L\cong\mathbb{Z}/N$. We write $\Lambda = L^{\vee}$ and consider the complex manifold $X=\operatorname{Tot}(\mathcal{O}(1)\rightarrow (\base))$ which carries an action of the congruence subgroup $G = L\rtimes (\SL(L^{\vee})\cap \SL(V^{\vee}))$. We defined the following spaces:
\begin{align*}
    \mathcal{V} &\defeq   \{a\in\Hom(V,\mathbb{Z}): V = L +\ker(a)\} \\
    \mathcal{W} &\defeq \{(a,b)\in \mathcal{V}^2: V = L + \ker(a,b)\} \\ 
    \mathcal{Q} &\defeq  \{(a,b,c)\in\mathcal{V}^3: (a,b),(b,c),(c,a)\in\mathcal{W}\}
\end{align*}
In particular, we constructed a $G$-invariant open cover of $X$ indexed by primitive functionals $a\in \mathcal{V}$, denoted $\mathcal{U} = \{V_a\}_{a\in \mathcal{V}_{\text{prim}}}$. The work of \cite{FHRZ08} gave us the elliptic Gamma cocycle $(\Delta^f,\Gamma)\in Z^1_G(\mathcal{U},\mathcal{M}^*/\mathcal{O}^*)$ which was independent of $f$; note that it is valued in $\mathcal{M}^*/\mathcal{O}^*$. Equivalently:
\begin{equation*}
    \psi^{f,f'} = \frac{\Delta^{f'}}{\Delta^f}\in C^{1,0}_G(\mathcal{U},\mathcal{O}^*)\quad D(\Delta^f,\Gamma) = (\phi^{2,0},\phi^{1,1},\phi^{0,2})\in C^2_G(\mathcal{U},\mathcal{O}^*)
\end{equation*}
We summarise our results in the following Theorem.
\begin{theorem}\label{theorem: vanishing cocycle}
    Suppose $(N,6)=1$. Then, the smoothed elliptic Gamma cochain $(\widetilde{\Delta}^{\underline{f}},\widetilde{\Gamma})\in C^1_G(\mathcal{U},\mathcal{M}^*)$ is independent of a compatible pair $\underline{f}$. Moreover, for any compatible pair $\underline{f}$:
    \begin{equation*}
        D(\widetilde{\Delta},\widetilde{\Gamma})= \left(\widetilde{\phi^{2,0}},\widetilde{\phi^{1,1}}\Big\vert_{\mathcal{W}},\widetilde{\phi^{0,2}}\Big\vert_{\mathcal{Q}}\right) =1.
    \end{equation*}
\end{theorem}
\begin{remark}
    The notation $\widetilde{\phi^{1,1}}\Big\vert_{\mathcal{W}}$ indicates the cochain in $C^{1,1}_G(\mathcal{U},\mathcal{O}^*)$ restricted to double intersections $V_a\cap V_b$ for $(a,b)\in \mathcal{W}$, that is $\widetilde{\phi^{1,1}}\Big\vert_{\mathcal{W}}: G\rightarrow\prod_{(a,b)\in \mathcal{W}}\mathcal{O}^*(V_a\cap V_b)$. 
\end{remark}
Thus, in some restricted sense, we may think of $(\widetilde{\Delta},\widetilde{\Gamma})$ as defining a $1$-cocycle in $H^1_G(X,\mathcal{M}^*)$. Associated to an admissible triple $\triple{h}$ for the pair $(V,L)$, we defined an evaluation:
\begin{equation}\label{eq: evaluation}
    \Phi_{\triple{h}}(\widetilde{\Delta},\widetilde{\Gamma}) = \widetilde{\Gamma}_{a,ga}(z(h),z).
\end{equation}
This is what we previously denoted by $\widetilde{\Phi}\triple{h}$. For another admissible triple $\triple{h'}$ satisfying $h\equiv h'$ modulo $L$, we proved:
\begin{equation*}
    \frac{\Phi_{\triple{h}}}{\Phi_{\triple{h'}}}(\widetilde{\Delta},\widetilde{\Gamma}) = \text{expression in }D(\widetilde{\Delta},\widetilde{\Gamma})
\end{equation*}
from which one immediately deduces the following Corollary.
\begin{corollary}
    If $(a,b,b')$, $(a,a',b')\in\mathcal{Q}$, we have $\Phi_{(h,\langle z\rangle,g)}(\widetilde{\Delta},\widetilde{\Gamma})= \Phi_{(h',\langle z\rangle,g)}(\widetilde{\Delta},\widetilde{\Gamma})$.
\end{corollary}
Thus we introduced the invariant $d$, which satisfies $d(h,h')=3\iff (a,a',b')\in\mathcal{Q}\iff (a,b,b')\in\mathcal{Q}$. In particular, our independence of $h$ result follows from the reduction to this case.
\begin{corollary}
    If $(N, 6\cdot \text{disc}(\chi_g))=1$, then for any $h$, $h'\in R$, there exists an $h'' \in R$ such that $d(h,h'')=d(h',h'')=3$. In particular, we deduce $\Phi_{(h,\langle z\rangle,g)}(\widetilde{\Delta},\widetilde{\Gamma})= \Phi_{(h',\langle z\rangle,g)}(\widetilde{\Delta},\widetilde{\Gamma})$.
\end{corollary}
The notation of \eqref{eq: evaluation} is suggestive of a more general evaluation procedure, a pairing between a homology class given by $\triple{h}$ and a cohomology class (almost) given by $(\widetilde{\Delta},\widetilde{\Gamma})$. More precisely, we ask:
\begin{enumerate}
    \item Is there a homological interpretation of the data $(h,\langle z\rangle,g)$ and the compatibility conditions imposed?
    \item Is there an evaluation morphism defined on all $1$-cochains that restricts to this evaluation map on the cochain $(\widetilde{\Delta},\widetilde{\Gamma})$? Motivated by the fact that 
    $\Phi_{(h,\langle z\rangle,g)}(\widetilde{\Delta},\widetilde{\Gamma})/\Phi_{(h',\langle z\rangle,g)}(\widetilde{\Delta},\widetilde{\Gamma})$ is expressed in terms of $D(\widetilde{\Delta},\widetilde{\Gamma})$, we might hope that (the restriction to $1$-cycles of) such an evaluation morphism descends to cohomology.
\end{enumerate}
Before answering these questions, we end this section with a computation of the restriction to a generic triptic curve of Dixmier-Douady class associated to the smoothed Gamma gerbe. Although it is not true that $(\widetilde{\Delta},\widetilde{\Gamma})$ is a cocycle, the smoothing operation does at least trivialise the restricted Dixmier-Douady class.
\subsubsection{An explicit cohomological computation}
As in \cite{FHRZ08}, denote by $\mathcal{G}$ the Gamma gerbe $[D(\Delta^f,\Gamma)]\in H^2_G(X,\mathcal{O}^*)$. Under the boundary map in the long exact sequence associated to the exponential sheaf sequence, we obtain the Dixmier-Douady class $c(\mathcal{G})\in H^3_G(X,\underline{\mathbb{Z}})$. We say $\langle z\rangle \in \base$ is generic when $z:L\hookrightarrow \mathbb{C}$; this ensures the quotient stack $[\mathbb{C}/z(L)]$ is in fact a sheaf. We restrict $c(\mathcal{G})$ to obtain the Dixmier-Douady class of the restricted Gamma gerbe:
\begin{equation*}
    H^3_G(X,\underline{\mathbb{Z}})\rightarrow H^3_L(X,\underline{\mathbb{Z}})\rightarrow H^3_L(\pi^{-1}\langle z\rangle \cong \mathbb{C},\underline{\mathbb{Z}}) = H^3(L,\mathbb{Z}).
\end{equation*}
For the final equality, since $L$ acts freely on the contractible fiber $\pi^{-1}\langle z\rangle \cong \mathbb{C}$, it follows that $[\mathbb{C}/z(L)]$ is a model for $BL$ and thus computes group cohomology. In \cite{FHRZ08}[Proposition 6.10], it is proven that $c(\mathcal{G})$ is a generator of $H^3(L,\mathbb{Z})= \mathbb{Z}$. From the explicit description of $\phi^{2,0}\in C^2(G,\check{C}^0(\mathcal{U},\mathcal{O}^*))$, we can express the restriction of $c(\mathcal{G})$ as:
\begin{equation*}
    \psi = \frac{1}{2\pi i}d\log(\phi^{2,0}_a)\in C^3(L,\check{C}^0(V_a,\underline{\mathbb{Z}})) = C^3(L,\mathbb{Z})
\end{equation*}
for some $a\in \mathcal{V}$ such that $z\in U_a$. Using the framing $f(a)=(x_1,x_2,x_3)$ and writing $\lambda = \sum l_i x_i $, $\mu = \sum m_ix_i$, $\nu = \sum n_ix_i\in L$, the cocycle is then explicitly computed to be:
\begin{equation*}
    \psi(\lambda,\mu,\nu) = -n_1m_2l_3
\end{equation*}
which is a generator of $H^3(L,\mathbb{Z})$.
\begin{proposition}
    Smoothing kills the restricted Dixmier-Douady class. Explicitly, we have $\widetilde{\psi}=0$ so that $c(\widetilde{\mathcal{G}})=[\widetilde{\psi}]=0\in  H^3_L(\pi^{-1}\langle z\rangle\cong \mathbb{C},\underline{\mathbb{Z}})$.  
\end{proposition}
\begin{proof}
    Fix a compatible pair so that $f_L(a)=(x_1,x_2,x_3)$ and $f_V(a)=(x_1,x_2,\frac{x_3}{N})$. Then, we compute:
    \begin{equation*}
        \widetilde{\psi}(\lambda,\mu,\nu)=N\cdot \psi(\lambda,\mu,\nu;L)-\psi(\lambda,\mu,\nu;V) =  N\cdot (-n_1m_2l_3) -(-n_1m_2(N\cdot l_3))=0.
    \end{equation*}
    Alternatively, we can see this via our explicit computation of $\widetilde{\phi^{2,0}}$ in Remark \ref{remark: explicit phi2,0} as:
    \begin{equation*}
        \widetilde{\phi^{2,0}_a}(\mu,\nu)(w,z)= e^{2\pi i\cdot n_1m_2\cdot \frac{N -1}{2}}.
    \end{equation*}
    In particular, $\frac{1}{2\pi i}\log(\widetilde{\phi^{2,0}_a})=n_1m_2\frac{N-1}{2}$. Then applying the additive differential we obtain:
    \begin{equation*}
         \widetilde{\psi}(\lambda,\mu,\nu)=\frac{1}{2\pi i}d\log(\widetilde{\phi^{2,0}_a})(\lambda,\mu,\nu)= \frac{N-1}{2}\left(n_1m_2+(m_1+n_1)l_2-n_1(l_2+m_2)-m_1l_2\right)=0.
    \end{equation*}
\end{proof}
\subsection{Homological interpretation of the input data}\label{subsec: geometric}
In this section, we show that admissible triples $\triple{h}$ for $L$ define loops in $X\times_G EG$, where $EG$ is the universal bundle over the classifying space $BG$. Admissible triples satisfy the following compatibility conditions:
\begin{equation*}
    g\cdot \langle z \rangle =\langle z\rangle \quad\quad  g\cdot  h\equiv h \text{ modulo }L.
\end{equation*}
In particular, $gh-h\in L$ and so we can define an element $x=(gh-h,g)\in G$. We write $p=(z(h),z)\in X$, which is independent of the choice of generator for the line $\langle z \rangle$. Now take the simplicial model for $EG$ whose $n$-simplices are $(EG)_n=G^{n+1}$. We denote $(p,(e,x))\in (X\times EG)_1$ for the $1$-chain $(p,e)\mapsto (p,x)\in (X\times EG)_0$. 
\begin{proposition}
    An admissible triple $\triple{h}$ for $L$ defines a loop $(p,(e,x))$ in $X\times_G EG$, where $p=(z(h),z)\in X$ and $x=(gh-h,g)\in G$.
\end{proposition}
\begin{proof}
    To show the $1$-chain $(p,(e,x))$ in $X\times EG$ becomes a $1$-cycle in the quotient $X\times_G EG$, we must show $(p,x)=x\cdot (p,e)$. We compute:
    \begin{equation*}
        x\cdot p = (gh-g,g)\cdot (z(h),z) = (z(h)-gz(gh-h),gz)=(gz(h),gz)=(z(h),z)
    \end{equation*}
    where for the final equality we use that $gz=\lambda z$ for $\lambda\in \mathbb{C}^*$ and that the coordinates on the opens $V_a$ are homogeneous: $V_a = (\mathbb{C}\times U_a)/\mathbb{C}^*$.
\end{proof}
Recall that $X\times_G EG$ naturally has the structure of a fiber bundle over $BG$ with fiber $X$:
\begin{equation*}
    X\hookrightarrow X\times_G EG\twoheadrightarrow BG
\end{equation*}
Thus, another phrasing of conditions on $h$ and $z$ is that define a lift of the loop $x\in \pi_1(BG)=G$, itself a lift of $g\in \pi_1(B\SL(L^{\vee})\cap \SL(V^{\vee}))$, from the base space to the fiber bundle.

For another admissible triple $(h',\langle z\rangle,g)$, we obtain another $1$-cycle in $X\times_G EG$ lifting $x'=(gh'-h',g)\in \pi_1(BG)$. A natural question to ask is whether the assumption $h\equiv h'$ implies these two cycles are homologous.

\begin{proposition}
    The $1$-cycles associated to $\triple{h}$ and $\triple{h'}$ are homologous when $h\equiv h'$ modulo $L$. In particular, there is a canonical homology class in $H_1(X\times_G EG,\mathbb{Z})$ attached to $\triple{h+L}$ where $g\cdot h+L=h+L$ and $g\cdot \langle z\rangle$.
\end{proposition}
\begin{proof}
    The space $X$ retracts onto $\base$, which itself is fibered over the Grassmannian of oriented planes $\operatorname{Gr}_2^+(\Lambda\otimes\mathbb{R})\cong S^2$ with contractible fibers. In particular, $X$ is homotopy equivalent to $S^2$ and $\pi_0(X)=\pi_1(X)=1$. Then, from the long exact sequence on homotopy groups associated to a fiber bundle, projection induces isomorphisms $\pi_1(X\times_G EG)\cong \pi_1(BG)=G$ and $H_1(X\times_G EG,\mathbb{Z})\cong H_1(BG,\mathbb{Z})=G^{ab}$. It therefore suffices to show that the group elements $x$ and $x'$ differ by a commutator in $G$. 
    \begin{equation*}
        x\cdot (x')^{-1}=(gh-h,g)\cdot(gh'-h',g)^{-1} = (gh-h,g)\cdot (g^{-1}h'-h',g^{-1}) = (h'-h-g(h'-h),1)
    \end{equation*}
    By our assumption that $h'\equiv h$ modulo $L$, we have $v=h'-h\in L$. Abusing notation, we write $v$ and $g$ for the group elements $(v,1)$ and $(0,g)$ respectively. Then:
    \begin{equation*}
        x\cdot (x')^{-1} = vgv^{-1}g^{-1}=[v,g]\in [G,G]
    \end{equation*}
    which proves that the loops associated to $\triple{h}$ and $\triple{h'}$ are homologous.

    Alternatively, denoting $H=\SL(L^{\vee})\cap \SL(V^{\vee})$, we use the explicit description of the abelianisation:
    \begin{equation*}
        G^{ab} = (L\rtimes H)^{ab} = (L^{ab})_H\times H^{ab}= (L)_H\times H^{ab}
    \end{equation*}
    Evidently $x\cdot (x')^{-1}=(h'-h-g(h'-h),1)\mapsto (0,1)\in G^{ab}$, as $h'-h-g(h'-h)\in I_HL$ for the augmentation ideal $I_H$.
\end{proof}
In the next subsection, we explicitly show that the loops are homologous in $H_1(X\times_G EG,\mathbb{Z})$ by exhibiting their difference as the boundary of a $2$-chain. This proves the result without resorting to a computation of the fundamental group of $X$.
\subsubsection{Explicit homology of loops}
Recall the notation $p=(z(h),z)$, $p'=(z(h'),z)\in X_3$ and $x = (gh-h,g)$, $ x' = (gh'-h',g)\in G$. The key relations are $xp=p$ and $x'p'=p'$. We denote the corresponding cycles as $(p,[x])$ and $(p',[x'])$ in $Z_1(X\times_G EG)\subset C_1(X\times_G EG) = C_1(X\times EG)_G$. Explicitly, the boundary map is $\overline{\partial}_1(p,[x])=(x^{-1}p,\ast)-(p,\ast)=0$, where $(BG)_0=\{\ast\}$. For the notation of the boundary map, see the Remark below.
\begin{remark}
    Recall the identification $C_*(BG) = C_*(EG)_G$: the $G$-equivariant (co)-homology of $EG$ is precisely the (co)-homology of $BG$, which in turn computes the group (co)-homology of $G$. It is standard to make the following change of coordinates when passing from $C_*(EG)$ to the quotient $C_*(BG)$. For instance, on $2$-chains:
\[\begin{tikzcd}[column sep=0.2cm,row sep=0.3cm]
	{C_2(EG)} &&& {C_1(EG)} &&& {(g_0,g_1,g_2)} &&& {(g_1,g_2)-(g_0,g_2)+(g_0,g_1)}  \\ \\\\\\
    {C_2(BG)} &&& {C_1(BG)} &&& {[x|y]=[g_0^{-1}g_1|g_1^{-1}g_2] } &&& {[y]-[xy]+[x]=[g_1^{-1}g_2]-[g_0^{-1}g_2]+[g_0^{-1}g_1]}\\
	\arrow["\partial_2",from=1-1, to=1-4]
	\arrow["\pi",from=1-1, to=5-1]
	\arrow["\pi",from=1-4, to=5-4]
	\arrow[maps to, from=1-7, to=1-10]
	\arrow[maps to, from=1-7, to=5-7]
	\arrow[maps to, from=1-10, to=5-10]
	\arrow["\overline{\partial}_2",from=5-1, to=5-4]
	\arrow[maps to, from=5-7, to=5-10]
\end{tikzcd}\]
    This change of coordinates gives a `nicer' formula for $\overline{\partial}$: indeed, $\overline{\partial}_2([x|y]) = [y]-[xy]+[x]$. From this, one easily deduces $H_1(G,\mathbb{Z}) = H_1(BG,\mathbb{Z}) = \mathbb{Z}G/\langle [y]-[xy]+[x]\rangle = G^{ab} $.
\end{remark}
We perform a similar change of coordinates in passing from $C_*(X\times EG)$ to $C_*(X\times_G EG)$; explicitly for $2$-chains, this is $(p,(g_0,g_1,g_2))\mapsto (g_0^{-1}p,[g_0^{-1}g_1|g_1^{-1}g_2])$. To recover a representative in $C_2(X\times EG)$ from $(q,[x|y])$, we choose $g_0\in G$ and send $(q,[x|y])\mapsto (g_0q,(g_0,g_0x,g_0xy))$. As above, this choice of coordinates gives a simpler formula for the boundary map $\overline{\partial}_2$. We compute:
\begin{equation*}
    \pi\circ\partial_2(p,(g_0,g_1,g_2)) = (g_1^{-1}p,[g_1^{-1}g_2])-(g_0^{-1}p,[g_0^{-1}g_2])+(g_0^{-1}p,[g_0^{-1}g_1]).
\end{equation*}
Writing $(z,[x|y]) =(g_0^{-1}p,[g_0^{-1}g_1|g_1^{-1}g_2]) = \pi((p,(g_0,g_1,g_2)))  $, we arrive at the formula:
\begin{equation*}
    \overline{\partial}_2(z,[x|y]) = (x^{-1}z,[y])-(z,[xy])+(z,[x]).
\end{equation*}
Setting $g_0 = e$, we get the useful identity:
\begin{equation}
    \pi\circ\partial_2(p,(e,g_1,g_2))= (g_1^{-1}p,[g_1^{-1}g_2])-(p,[g_2])+(p,[g_1]).
\end{equation} 
\begin{proposition}
    The difference $ (p',[x'])-(p,[x])\in C_1(X\times_G EG)$ is a boundary cycle, that is it lies in the image $\overline{\partial}_2(C_2(X\times_G EG)) = \pi\circ\partial_2(C_2(X\times EG))$. We give an explicit preimage below.
\end{proposition}
\begin{proof}
    As before, we abuse notation and write $v$ for the element $(h'-h,1)\in G$. One easily verifies $v^{-1}\cdot x=(gh-h',g)$ and $v^{-1}\cdot p=p'$. Then:
     \begin{align*}
        &\pi\circ\partial_2(((p,(e,(gh'-h,g),x))-(p',(e,-v,x'))) \\&=  (p',[-v])-(p,[x])+(p,[(gh'-h,g])- \big((p,[(gh'-h,g)])-(p',[x'])+(p',[-v])\big) \\ & = (p',[x'])- (p,[x]).
    \end{align*}
    This explicitly verifies the identity $(p',[x'])=(p,[x])\in H_1(X\times_G EG,\mathbb{Z})$. 
\end{proof}
The following diagram helps track the relation between the group elements and provides an explicit picture of the boundary of $2$-chains:
    \[\begin{tikzcd}
    	p &&& {p'} \\
    	\\
    	p &&& {p'}
    	\arrow["v^{-1}", from=1-1, to=1-4]
    	\arrow["x", from=3-1, to=1-1]
    	\arrow["{(gh'-h,g)}"{description}, from=3-4, to=1-1]
    	\arrow["{x'}"', from=3-4, to=1-4]
    	\arrow["v^{-1}"', from=3-1, to=3-4]
    \end{tikzcd}\]
\subsection{Evaluation Morphism}
We have previously defined \eqref{eq: evaluation} an evaluation procedure
\begin{equation*}
    \Phi_{(h,\langle z\rangle,g)}(\widetilde{\Delta},\widetilde{\Gamma}) = \widetilde{\Gamma}_{a,ga}(z(h),z).
\end{equation*}
The aim now is to extend this to an evaluation map on all $1$-cochains $C^1_G(\mathcal{U},\mathcal{M}^*)$ which will descend to a map on cohomology.
\begin{definition}
    Consider an admissible triple $\triple{h}$ for $L$. Then, we define a group morphism $\Phi_{(h,\langle z\rangle,g)}:C^1_G(\mathcal{U},\mathcal{M}^*)\rightarrow \mathbb{C}^{*}$ via:
    \begin{equation*}
        (\phi^{1,0},\phi^{0,1})\mapsto \phi^{1,0}_a(x^{-1})(p)\cdot \phi^{0,1}_{a,ga}(p)
    \end{equation*}
    where $p=(z(h),z)$ and $x=(gh-h,g)$.
\end{definition}
\begin{lemma}
    We have $\Phi_{(h,\langle z\rangle,g)}(\widetilde{\Delta},\widetilde{\Gamma}) = \widetilde{\Gamma}_{a,ga}(z(h),z)$. In particular, this evaluation map extends our previous procedure.
\end{lemma}
\begin{proof}
    It suffices to show that $\widetilde{\Delta}_a(x^{-1})(p)=1$. We have $x^{-1}=(g^{-1}h-h,g^{-1})$ and by definition $\widetilde{\Delta}_a((g^{-1}h-h,g^{-1}))=\widetilde{\Delta}_a(g^{-1}h-h)=1\in\mathcal{M}^*(V_a)$ as $g^{-1}h-h\in \ker(a)$.
\end{proof}
\begin{proposition}
    $\Phi_{(h,\langle z\rangle,g)}:H^1_G(X,\mathcal{M}^*)\rightarrow \mathbb{C}^*$ is well-defined.
\end{proposition}
\begin{proof}
    We must show that $\Phi_{(h,\langle z\rangle,g)}$ is trivial on boundaries $B^1_G(\mathcal{U},\mathcal{M}^*)=D(C^{0}_G(\mathcal{U},\mathcal{M}^*))$. To that end, let $f\in C^{0}_G(\mathcal{U},\mathcal{M}^*)=\prod_{a\in \mathcal{V}}\mathcal{M}^*(V_a)$. Then:
    \begin{align*}
        \Phi_{(h,\langle z\rangle,g)}(Df)&=(df)_a(x^{-1})(p)\cdot (\partial f)_{a,ga}(p) \\& = \frac{f_{ga}(x\cdot p)}{f_a(p)}\cdot \frac{f_a(p)}{f_{ga}(p)} = 1
    \end{align*}
    since $x\cdot p = p$. Notice that the numerator for the left quotient is $f_{ga}(x\cdot p)$ not $f_a(x\cdot p)$ as $x\cdot p = (gz(h),gz) \in V_{ga}$.
\end{proof}
\subsubsection{Homologous as maps on cohomology}
It is natural to ask how the evaluation maps $\Phi_{\triple{h}}$ and $\Phi_{\triple{h'}}$ compare. Since we have shown that the loops associated to $\triple{h}$ and $\triple{h'}$ are homologous, it is reasonable to expect they give the same maps on cohomology. 
\begin{proposition}
    Suppose $h\equiv h'$ modulo $L$. Then, we have:
    \begin{equation*}
        \frac{\Phi_{(h,\langle z\rangle,g)}}{\Phi_{(h',\langle z\rangle,g)}}\left(\phi^{1,0},\phi^{0,1}\right) = \text{ expression in } D(\phi^{1,0},\phi^{0,1})
    \end{equation*}
    for some expression that is given explicitly below. Thus, if $(\phi^{1,0},\phi^{0,1})\in Z^1_G(\mathcal{U},\mathcal{M}^*)$, we have:
    \begin{equation*}
        \Phi_{(h,\langle z\rangle,g)}(\phi^{1,0},\phi^{0,1}) =\Phi_{(h',\langle z\rangle,g)}(\phi^{1,0},\phi^{0,1}) 
    \end{equation*}
    and in particular, $\Phi_{(h,\langle z\rangle,g)}=\Phi_{(h',\langle z\rangle,g)}: H^1_G(X,\mathcal{M}^*)\rightarrow\mathbb{C}^*$.
\end{proposition}
\begin{proof}
    We need to compute:
\begin{equation}\label{eq: initial quotient}
    \frac{\Phi_{(h,\langle z\rangle,g)}}{\Phi_{(h',\langle z\rangle,g)}}\left(\phi^{1,0},\phi^{0,1}\right) = \frac{\phi^{1,0}_a(x^{-1})(p)\cdot \phi^{0,1}_{a,ga}(p) }{\phi^{1,0}_{a'}(x'^{-1})(p')\cdot \phi^{0,1}_{a',ga'}(p')}.
\end{equation}
As in the proof for the cocycle expression for $(\Delta^f,\Gamma)$, we use the relations:
\begin{align*}
    & (\partial\phi^{0,1})_{a,ga,ga'}(p)=\overbrace{\phi^{0,1}_{a,ga}(p)}^{\textcolor{red}{\ast}}\phi^{0,1}_{ga,ga'}(p)\phi^{0,1}_{ga',a}(p) =\phi^{0,2}_{a,ga,ga'}(p)\\
    & (\partial\phi^{0,1})_{a,a',ga'}(p')=\phi^{0,1}_{a,a'}(p')\underbrace{\phi^{0,1}_{a',ga'}(p')}_{\textcolor{red}{\ast}}\phi^{0,1}_{ga',a}(p') =\phi^{0,2}_{a,a',ga'}(p') 
\end{align*}
to rewrite:
\begin{equation*}
    \frac{\phi^{0,1}_{a,ga}(p) }{ \phi^{0,1}_{a',ga'}(p') } = \frac{\phi^{0,2}_{a,ga,ga'}(p)}{\phi^{0,2}_{a,a',ga'}(p') }\cdot\frac{\phi^{0,1}_{a,a'}(p')}{\phi^{0,1}_{ga,ga'}(p)}\cdot \frac{\phi^{0,1}_{ga',a}(p') }{\phi^{0,1}_{ga',a}(p) }.
\end{equation*}
By definition, $(d\phi^{0,1})_{a,b}(v)(w,z) = \phi^{0,1}_{a,b}(w+z(v),z)/\phi^{0,1}_{a,b}(w,z)$, so setting $v = h'-h$ as usual, we have ensures $p'=(z(h'),z) = (z(h)+z(v),z)$ and thus:
\begin{equation}\label{eq: second quotient}
    \frac{\phi^{0,1}_{ga',a}(p') }{\phi^{0,1}_{ga',a}(p) } = (d\phi^{0,1})_{ga',a}(v)(p) = \phi^{1,1}_{ga',a}(v)(p)\cdot (\partial\phi^{1,0})_{ga',a}(v)(p)
\end{equation}
by definition of the differential $D(\phi^{1,0},\phi^{0,1})$. 

Moreover, we have $(gh'-h,g)\cdot p' = (z(h')-gz(gh'-h),gz)=(z(gh')-z(gh'-h),z)=p$. And thus:
\begin{align}
    \frac{\phi^{0,1}_{ga,ga'}(p)}{\phi^{0,1}_{a,a'}(p')} = \frac{\phi^{0,1}_{ga,ga'}((gh'-h,g)\cdot p')}{\phi^{0,1}_{a,a'}(p')} &= (d\phi^{0,1})_{a,a'}((gh'-h,g)^{-1})(p') \notag \\&= \phi^{1,1}_{a,a'}((g^{-1}h-h',g^{-1}))(p') \cdot (\partial \phi^{1,0})_{a,a'}((g^{-1}h-h',g^{-1}))(p') \label{eq: first quotient}
\end{align}
where we use $(gh'-h,g)^{-1}=(g^{-1}h-h',g^{-1})$. Now combining equations \eqref{eq: second quotient} and \eqref{eq: first quotient}, we can rewrite \eqref{eq: initial quotient}:
\begin{align*}
    \frac{\Phi_{(h,\langle z\rangle,g)}}{\Phi_{(h',\langle z\rangle,g)}}\left(\phi^{1,0},\phi^{0,1}\right) = &\frac{\phi^{1,0}_a(x^{-1})(p)}{\phi^{1,0}_{a'}(x'^{-1})(p')}\cdot \frac{\phi^{1,0}_{a'}((g^{-1}h-h',g^{-1}))(p')}{\phi^{1,0}_{a}((g^{-1}h-h',g^{-1}))(p')}\cdot \frac{\phi^{1,0}_{ga'}(v)(p)}{\phi^{1,0}_{a}(v)(p)} \\&\cdot \phi^{1,1}_{a,a'}((g^{-1}h-h',g^{-1}))(p')^{-1}\cdot \phi^{1,1}_{ga',a}(v)(p)\cdot \frac{\phi^{0,2}_{a,ga,ga'}(p)}{\phi^{0,2}_{a,a',ga'}(p')}.
\end{align*}
Now we use the final relation:
\begin{equation*}
    \phi^{2,0}_a(x_1,x_2)(q) =\frac{ \phi^{1,0}_a(x_1)(q)\cdot\phi^{1,0}_a(x_2)(x_1^{-1}\cdot q)}{\phi^{1,0}_a(x_1\cdot x_2)(q)}.
\end{equation*}
In particular, setting $x_1 = v $, $x_2 = (g^{-1}h-h',g^{-1})$, we observe that $x_1\cdot x_2 = (g^{-1}h -h,g^{-1})=x^{-1}$ and $x_1^{-1}p = (h-h',1)\cdot (z(h),z)=p'$. Thus, we have the following identity:
\begin{equation}\label{eq: first 2,0}
    \phi^{2,0}_a(v,(g^{-1}h-h',g^{-1}))(p) =\frac{\phi^{1,0}_{a}(v)(p)\cdot \phi^{1,0}_{a}((g^{-1}h-h',g^{-1}))(p')}{\phi^{1,0}_a(x^{-1})(p)} .
\end{equation}
Moreover, swapping the order of the inputs and evaluating at $p'$, we have:
\begin{equation}\label{eq: second 2,0}
    \phi^{2,0}_{a'}((g^{-1}h-h',g^{-1}),v)(p') = \frac{\phi^{1,0}_{a'}((g^{-1}h-h',g^{-1}))(p')\cdot \phi^{1,0}_{ga'}(v)(p)}{\phi^{1,0}_{a'}( (x')^{-1})(p')} 
\end{equation}
where we note that $(g^{-1}h-h',g^{-1})\cdot v = (x')^{-1}$ and $(g^{-1}h-h',g^{-1})^{-1}\cdot p' = (gh'-h,g)\cdot p'=(gz(h),gz)=p\in V_{ga'}$. Combining equations \eqref{eq: first 2,0} and \eqref{eq: second 2,0}, we finally obtain the explicit formula:
\begin{align*}
    \frac{\Phi_{(h,\langle z\rangle,g)}}{\Phi_{(h',\langle z\rangle,g)}}\left(\phi^{1,0},\phi^{0,1}\right) = \frac{\phi^{2,0}_a(v,(g^{-1}h-h',g^{-1}))(p)^{-1}}{ \phi^{2,0}_{a'}((g^{-1}h-h',g^{-1}),v)(p')}\cdot \phi^{1,1}_{a,a'}((g^{-1}h-h',g^{-1}))(p')^{-1}\cdot \phi^{1,1}_{ga',a}(v)(p)\cdot \frac{\phi^{0,2}_{a,ga,ga'}(p)}{\phi^{0,2}_{a,a',ga'}(p')}.
\end{align*}
This completes the proof of the Proposition.
\end{proof}
\begin{remark}
    The elements $v=(h'-h,1)$ and $(g^{-1}h-h',g^{-1})^{-1}=(gh'-h,g)$ are precisely the elements which we used to define the $2$-chain whose boundary gave us a homology between the loops associated to $\triple{h}$ and $\triple{h'}$.
\end{remark}
\begin{remark}
    The contribution of $\phi^{1,0}$ term to the evaluation $\Phi_{\triple{h}}$ allows us to define a morphism on cochains that descends to cohomology. Additionally, it is a `nicer' formula for the cochain $(\Delta^f,\Gamma)$ in the sense that now $\frac{\Phi_{(h,\langle z\rangle,g)}}{\Phi_{(h',\langle z\rangle,g)}}\left(\Delta^f,\Gamma\right)$ is expressed purely in terms of $D(\Delta^f,\Gamma)$. By adding a redundant $\Delta^f$ term, the change of framing term that previously appeared is now absorbed into the cocycle expression.
\end{remark}
\newpage
\appendix 
\section{Relaxing the coprimality assumption}\label{appendix: coprime assumption}
We still consider a pair of lattices $(V,L)$ with cyclic quotient $V/L \cong \mathbb{Z}/N$. In all the results above, we have assumed that $(N,6)=1$. Aside from in the reduction to the $d=3$ case, we used this coprimality assumption twice: to deduce that $\theta_0^N$ is $\Gamma_0(N)$-invariant and to prove the distribution relation for $B_{3,3}$. If we drop the coprimality assumption, the error to both claims is a root of unity. Crucially, we can bound the root of unity: independently of $N$ for $\theta_0^N$, and in terms of $N$ for the Distribution Relation.
\begin{proposition}
        For $N\in \mathbb{Z}$ and $\big(\begin{smallmatrix}a& b \\c&d\end{smallmatrix}\big)\in\Gamma_0(N)$, then:
    \begin{equation*}
        \frac{\theta_0^N(x,\tau)}{\theta_0^N\left(\frac{x}{c\tau+d},\frac{a\tau+b}{c\tau+d}\right)}\in \mu_{24}.
    \end{equation*} 
    More precisely, the root of unity ambiguity lies in $\mu_{k}$ where $k = \operatorname{gcd}_{(l,N)=1}(l^2-1)$ which we compute to be: 
    \begin{equation*}
        k=\begin{cases}
            \text{  }1 &:\text{  } (N,6)=1 \\
            \text{  }3 &: \text{  }2\nmid N, \text{  }3\mid N \\
            \text{  }8 &: \text{  }2\mid N, \text{             }3\nmid N \\
            \text{  }24 &: \text{  }6|N
        \end{cases}
    \end{equation*}
\end{proposition}
The proof follows the approach taken by \cite{Kato04}, in which the coprimality hypothesis $(N,6)=1$ is assumed. For a fixed $\tau$ and $g\in \Gamma_0(N)$, we will show both $\theta_0^N(-,\tau)$ and $g^*\theta_0^N(-,\tau)$ are sections of the same line bundle on $\mathbb{C}/\mathbb{Z}\oplus\mathbb{Z}\tau$ with the same divisor. Thus, their quotient is constant. Finally, we use the distribution relations to determine this constant.

Recall $\theta_0$ is a global section of $\mathcal{O}_{\mathcal{E}}([\mathcal{O}])$, the line bundle on the universal elliptic curve $\mathcal{E}$ over $\mathbb{H}$ associated to the divisor given by the zero section. 
\begin{lemma}
    $\theta_0^N$ is a meromorphic section of the line bundle $\mathcal{L}$:
    \begin{equation*}
        \mathcal{L}=\begin{cases}
            \mathcal{O}_{\mathcal{E}}&: (N,2)=1\\\mathcal{O}_{\mathcal{E}}([1/2])&: 2|N 
        \end{cases}
    \end{equation*}
    Moreover, for $g \in\Gamma_0(N)$, we have $g^{*}\theta_0^N$ is a section of the same line bundle $\mathcal{L}$.
\end{lemma}
The $\Gamma_0(N)$ is a strictly stronger assumption than needed: if $(N,2)=1$, we could take $g\in \SL_2(\mathbb{Z})$ and if $2|N$, it suffices to take $g\in \Gamma_0(2)$.
\begin{proof}
    We compute the divisor:
    \begin{equation*}
        \text{div}(\theta_0^N) = N\cdot [\mathcal{O}]-\sum_{j=0}^{N-1}\left[\frac{j}{N}\right]
    \end{equation*}
    from which the results follows from an application of Wilson's theorem: namely the sum of these divisors is either $[\mathcal{O}]$ or the unique element of order $2$ when it exists. The claim for $g^*\theta_0^N$ follows immediately from the fact that $\Gamma_0(N)$ preserves $\mathcal{L}$ in either case.
    
    The above only specifies the isomorphism class of the line bundle $\mathcal{L}$. To see the stronger claim that both $\theta_0^N$ and $g^*\theta_0^N$ are sections of the same line bundle, we explicitly verify they both define the same $\mathbb{Z}^2$-cocycle:
    \begin{align*}
        \theta_0^N(x+1,\tau)=\theta_0^N(x,\tau) &\quad \theta_0^N(x+\tau,\tau)= (-1)^{N-1}\cdot \theta_0^N(x,\tau) \\ g^*\theta_0^N(x+1,\tau)=g^*\theta_0^N(x,\tau) &\quad g^*\theta_0^N(x+\tau,\tau)= (-1)^{N-1}
        \cdot g^*\theta_0^N(x,\tau)
    \end{align*}
    which hold for $g\in \Gamma_0(2)\supset \Gamma_0(N)$. In particular, since these sections square to functions on the universal elliptic curve, it is clear that the line bundle $\mathcal{L}$ must be $2$-torsion.
\end{proof}
\begin{lemma}
    The quotient $\theta_0^N/g^*\theta_0^N$ is a non-vanishing function on $\mathbb{H}$ for $g\in \Gamma_0(N)$.
\end{lemma}
\begin{proof}
    In the previous Lemma, we saw $d\theta_0^N = dg^*\theta_0^N\in Z^1(\mathbb{Z}^2,\mathcal{O}^*(\mathbb{C}\times\mathbb{H}))$ and so $\theta_0^N/g^*\theta_0^N\in \mathcal{M}^*(\mathcal{E})$. Moreover, for $g\in \Gamma_0(N)$, we compute $\text{div}(g^{*}\theta_0^N)=g^*\text{div}(\theta_0^N)=\text{div}(\theta_0^N)$ and so $\theta_0^N/g^*\theta_0^N\in \mathcal{O}^*(\mathcal{E})=\mathcal{O}^*(\mathbb{H})$.
\end{proof}
Recall $\theta_0$ satisfies the functional identity on $\mathbb{C}\times\mathbb{H}$:
\begin{equation}\label{eq: theta pushforward} 
    \prod_{m,n=0}^{l-1}\theta_0\left(x+\frac{m\tau+n}{l},\tau\right)= \theta_0(lx,\tau).
\end{equation}
Both of these generalise to $\theta_0^N$ and $g^*\theta_0^N$.
\begin{lemma}
    Suppose $g\in \Gamma_0(N)$. Then, as meromorphic functions on $\mathbb{C}\times\mathbb{H}$, for all positive integers $(l,N)=1$:
    \begin{align*}
        \prod_{m,n=0}^{l-1}\theta_0^N\left(x+\frac{m\tau+n}{l},\tau\right) = \theta_0^N(lx,\tau)
        \quad \prod_{m,n=0}^{l-1}g^*\theta_0^N\left(x+\frac{m\tau+n}{l},\tau\right) = g^*\theta_0^N(lx,\tau)
    \end{align*}
\end{lemma}
\begin{proof}
    Since $\theta_0^N=\theta_0/N^*\theta_0$, it remains to show the analogous relation for $N^*\theta_0$. Explicitly:
    \begin{align*}
        \prod_{m,n=0}^{l-1}N^*\theta_0\left(x+\frac{m\tau+n}{l},\tau\right)&= \prod_{m,n=0}^{l-1}\theta_0\left(Nx+\frac{mM\tau+Nn}{l},N\tau\right)\\&=\prod_{m,n=0}^{l-1}\theta_0\left(Nx+\frac{mN\tau+n}{l},N\tau\right)=\theta_0(lNx,N\tau) =N^*\theta_0(lx,\tau) 
    \end{align*}
    where the first and last equalities are definitions. The penultimate equality is \eqref{eq: theta pushforward}. The second equality follows from $N\in (\mathbb{Z}/l)^*$ and $\theta_0(x+1,\tau)=\theta_0(x,\tau)$. This proves the first functional relation.

    Before proving the second relation, we note that if $(N,2)=1$, then it immediately follows from the first one: in this case, we know $\theta_0^N(-,\tau)$ is a function on the elliptic curve $\mathbb{C}/\mathbb{Z}\oplus\mathbb{Z}\tau$ and then the first relation can be rephrased as pushforward invariance $[l]_*\theta_0^N = \theta_0^N$. Then, we have $[l]_*g^*\theta_0^N = g^*[l]_*\theta_0^N=g^*\theta_0^N$ since $g$ preserves the $l$-torsion.

    So let us assume $2|N$ and write $g^{
    -1} =\big(\begin{smallmatrix}a& b \\c&d\end{smallmatrix}\big) $. Then, the left hand side is:
    \begin{equation*}
        \prod_{m,n=0}^{l-1}\theta_0^N\left(\frac{x}{c\tau+d}+\frac{m\tau+n}{l(c\tau+d)},\frac{a\tau+b}{c\tau+d}\right)=\prod_{m,n=0}^{l-1}\theta_0^N\left(w+\frac{(dm-cn)\sigma+(-bm+an)}{l},\sigma\right)
    \end{equation*}
    where $w=\frac{x}{c\tau+d}$ and $\sigma = \frac{a\tau+n}{c\tau+d}$. Though we know the product ranges over a set of representatives for the $l$-torsion on $\mathbb{C}/\mathbb{Z}\oplus\mathbb{Z}\sigma$, we can no longer conclude as $\theta_0^N$ is not $\mathbb{Z}\oplus\mathbb{Z}\sigma$-invariant. Instead, we use the identities:
    \begin{equation*}
        \theta_0^N(x+1,\tau)=\theta_0^N(x,\tau) \quad \theta_0^N(x+\tau,\tau)= (-1)^{N-1}\cdot \theta_0^N(x,\tau) 
    \end{equation*}
    to deduce:
    \begin{align*}
        \prod_{m,n=0}^{l-1}\theta_0^N\left(w+\frac{(dm-cn)\sigma+(-bm+an)}{l},\sigma\right) &= \prod_{m,n=0}^{l-1}(-1)^{x_{m,n}\cdot(N-1)}\theta_0^N\left(w+\frac{m\sigma+n}{l},\sigma\right)\\ &= \prod_{m,n=0}^{l-1}(-1)^{x_{m,n}}\cdot \theta_0^N(lw,\sigma) = (-1)^{\sum x_{m,n}}\cdot g^*\theta_0^N(lx,\tau)
    \end{align*}
    where $x_{m,n}=\frac{1}{l}(dm-cn - [dm-cn]_l)$. It remains to show $\sum{x_{m,n}}\in 2\mathbb{Z}$. We compute:
    \begin{equation*}
       \sum_{m,n=0}^{l-1}{x_{m,n}} = \frac{1}{l}\left(ld\cdot \sum_{m=0}^{l-1}m-lc\cdot\sum_{n=0}^{n-1}n-l\sum_{k=0}^{n-1}k\right)  = \sum_{m=0}^{l-1}m\cdot (d-1-c)\overset{?}{\in}2\mathbb{Z}
    \end{equation*}
    Since we are in the case $N|2$, it follows $g\in \Gamma_0(2)$. In particular, $2|c$ and $(2,d)=1$, so that $2|d-1-c$.
\end{proof}
\begin{corollary}
    For each $\tau\in \mathbb{H}$, the function $\theta_0^N/g^*\theta_0^N(-,\tau)$ on the elliptic curve $\mathbb{C}/\mathbb{Z}\oplus\mathbb{Z}\tau$ is invariant under pushforward by $[l]_*$ for all positive integers $l$ such that $(l,N)=1$.
\end{corollary}
\begin{proof}[Proof of Proposition]
    Combining the above Lemmas, we know that for $g\in \Gamma_0(N)$ and fixed $\tau$, the quotient $\theta_0^N/g^*\theta_0^N(-,\tau)=\lambda$ for $\lambda\in \mathbb{C}$ satisfying:
    \begin{equation*}
        [l]_*\lambda = \lambda^{l^2}=\lambda\iff  \lambda\in \mu_{l^2-1}  \end{equation*}
    for all $(l,N)=1$. In particular, $\theta_0^N/g^*\theta_0^N(-,\tau)\in \mu_{k}$ for $k=\operatorname{gcd}_{(l,N)=1}(l^2-1)$. The second part of the Proposition, namely the formula for $k$, follows from weak approximation together with the observations:
    \begin{itemize}
        \item $(\mathbb{Z}_2^*)^2=1+8\mathbb{Z}_2$. In particular, there exists $t_2\in\mathbb{Z}_2^*$ such that $v_2(t_2^2-1)=3$. 
        \item $(\mathbb{Z}_3^*)^2=1+3\mathbb{Z}_3$. In particular, there exists $t_3\in\mathbb{Z}_3^*$ such that $v_3(t_3^2-1)=1$.
        \item For $p\neq 2,3$, we have $(\mathbb{Z}_p^*)^2=(\mathbb{F}_p^*)^2\cdot 1+p\mathbb{Z}_p$ and $(\mathbb{F}_p^*)^2\neq\{1\}$. In particular, there exists $t_p\in\mathbb{Z}_p^*$ such that $v_p(t_p^2-1)=0$.
    \end{itemize}

    Finally, by the discreteness of $\mu_{24}$, we deduce the result for all $\tau$, namely that $\theta_0^N/g^*\theta_0^N\in \mu_{24}$.
\end{proof} 
\begin{remark}\label{remark: kato perspective}
    The style of argumentation above leads us to the following characterisation of $\theta_0$. It is the unique holomorphic function on $\mathbb{C}\times\mathbb{H}$ that satisfies the pushforward relation \eqref{eq: theta pushforward} and descends to a section of a specified line bundle on $\mathcal{E}$ with divisor $ [\mathcal{O}]$. Immediately, one deduces that $\theta_0/g^*\theta_0\in \mathcal{O}^*(\mathbb{C}\times\mathbb{H}) = \exp(\mathcal{O}(\mathbb{C}\times\mathbb{H}))$ for $g\in \SL_2(\mathbb{Z})$ as pullback by $\SL_2(\mathbb{Z})$ preserves the divisor of the section. The exact form of this automorphy factor, as stated in Lemma \ref{lemma: explicit cocycle j}, can then be deduced by the uniqueness property.
\end{remark}
Now consider the Distribution Relation for $B_{3,3}$. With no assumption on $N$, we still arrive at:
\begin{equation*}
    \sum_{\nu\in P }NB_{3,3}(w+z(\nu),z(\underline{x})) -  B_{3,3}(Nw,z(\underline{x})) = 3!\left(-\frac{(N-1)^3}{8}+\frac{1}{N^2}\sum_{\nu\in P}\nu_1\nu_2\nu_3\right).
\end{equation*}
In particular, taking $\exp(\frac{-2\pi i}{3!}-)$, we will obtain an element of $\mu_{8\cdot N^2}$. The only other time the coprimality $(N,6)=1$ condition is used is in the final steps of Proposition \ref{prop: final error term}, precisely \eqref{eq: final N^2}, to clear another $N^2$ denominator. In particular, dropping the $(N,6)=1$ condition, we are still able to prove
$\widetilde{\phi^{0,2}}\Big\vert_{\mathcal{Q}}\in \mu_{8\cdot N^2}$.
\begin{corollary}
    The smoothed elliptic Gamma cochain $(\widetilde{\Delta}^{\underline{f}},\widetilde{\Gamma})\in C^1_G(\mathcal{U},\mathcal{M}^*)$ is independent of a compatible pair $\underline{f}$ up to $\mu_{24}$: that is to say $\widetilde{\psi^{\underline{f},\underline{f'}}}\in \mu_{24}$. Moreover, for any compatible pair $\underline{f}$:
    \begin{equation*}
        D(\widetilde{\Delta}^{\underline{f}},\widetilde{\Gamma})=\left(\widetilde{\phi^{2,0}},\widetilde{\phi^{1,1}}\Big\vert_{\mathcal{W}},\widetilde{\phi^{0,2}}\Big\vert_{\mathcal{Q}}\right) \in \mu_{24\cdot N^2}.
    \end{equation*}
    If $(N,6)=1$, then this root of unity is always $1$.
\end{corollary}
We are thus able to marginally extend our results for very little work.
\begin{theorem}\label{theorem: total theorem}
    Fix a pair of rank $3$ lattices $(V,L)$ with cyclic quotient $V/L \cong \mathbb{Z}/N$. Consider triples $h\in L\otimes\mathbb{Q}$, $\langle z\rangle \in \mathbb{P}(L^{\vee}\otimes\mathbb{C})-\mathbb{P}(L^{\vee}\otimes\mathbb{R}) $, and an automorphism $g$ of the pair $(V,L)$, that are admissible for the pair $(V,L)$:
    \begin{itemize}
        \item $g\cdot h\equiv h$ modulo $L$.
        \item $\langle h\rangle_{\mathbb{Q}}$ generates the quotient $V/L$.
        \item $g\cdot \langle z\rangle =\langle z\rangle$.
    \end{itemize}
    Then for all admissible triples $\triple{h}$, $\triple{h'}$ for the pair $(V,L)$ such that $h\equiv h'$ modulo $L$ and $d(h,h')=3$:
    \begin{equation*}
        \frac{\widetilde{\Phi}\triple{h}}{ \widetilde{\Phi}\triple{h'}}\in \mu_{24\cdot N^2}
    \end{equation*}
    where:
    \begin{equation*}
        \widetilde{\Phi}\triple{h}=\widetilde{\Gamma}_{a,ga}(z(h),z)=\frac{\Gamma_{a,ga}(z(h),z;L)^N}{\Gamma_{a,ga}(z(h),z;V)}
    \end{equation*}
    for primitive functional $a\in V^{\vee}$ determined by the triple.

    Moreover, if $(N,6)=1$, then we can strengthen the result in two directions: firstly, this root of unity is always $1$; secondly, if we further assume $(N,\disc (\chi_g))=1$, then we can remove the $d=3$ assumption by always finding an admissible triple $\triple{h''}$ with $h''\equiv h'\equiv h$ modulo $L$ such that $d(h,h'')=d(h',h'')=3$.
\end{theorem}
\newpage
\section{An extremal case: $h\in L$}\label{appendix: extremal}
Recall that an admissible triple $\triple{h}$ for the pair $(V,L)$ satisfies the following conditions: $g\cdot\langle z \rangle=\langle z\rangle$, $gh\equiv h$ and $\langle h\rangle_{\mathbb{Q}}$ generates the quotient $V/L$. Theorem \ref{theorem: main theorem in main body} then states that if $(N,6\cdot\disc(\chi_g))=1$, then for any two admissible triples $\triple{h}$ and $\triple{h'}$ satisfying $h\equiv h'$ modulo $L$, we have:
\begin{equation*}
    \widetilde{\Gamma}_{a,b}(z(h),z)=\widetilde{\Gamma}_{a',b'}(z(h'),z).
\end{equation*}
The reader may note that we did \textit{not} exclude the case $h\in L$. Moreover, for any pair $h$, $h'\in L $, we have $gh\equiv h\equiv h'\equiv gh'$ modulo $L$. However, computations show that it is \textit{not} true that $\widetilde{\Phi}\triple{h}$ is constant for $h\in L$, demonstrating the need for the further, and in this case only, assumption that $\langle h\rangle_{\mathbb{Q}}$ generates $V/L$. In this case, we can give a much shorter proof that $\widetilde{\Phi}\triple{h}$ is independent, up to signs, of the choice of $h\in L$.
\begin{proposition}\label{prop: h in L}
    Suppose $g\cdot z=\lambda z$ for $z\in \Lambda\otimes\mathbb{C}-\Lambda\otimes\mathbb{R}$. Then, for all $h\in L$ such that $\langle h \rangle_{\mathbb{Q}}$ generates $V/L$:
    \begin{equation*}
        \widetilde{\Phi}\triple{h}^2 = \lambda^{N-1}.
    \end{equation*}
\end{proposition}
In particular, the Proposition directly verifies Theorem \ref{theorem: main theorem in main body}, up to signs, for $h\in L$. The proof reduces to the following computation of $\widetilde{\Gamma}$ at $0$.
\begin{lemma}
    Suppose $g\cdot z=\lambda z$ for $z\in \Lambda\otimes\mathbb{C}-\Lambda\otimes\mathbb{R}$. Then, for $(a,ga)\in\mathcal{W}$:
    \begin{equation*}
        \widetilde{\Gamma}_{a,ga}(0,z)^2=\lambda^{N-1}.
    \end{equation*}
\end{lemma}
\begin{proof}
    With no assumption on the pair $(a,b)$, we have the completely general equality:
    \begin{equation*}
        \Gamma_{a,b}(w,z)\cdot \Gamma_{a,b}(-w,z)=\frac{\theta_0\left(\frac{w}{z(\gamma)},\frac{z(-\alpha)}{z(\gamma)}\right)}{\theta_0\left(\frac{w}{z(\gamma)},\frac{z(\beta)}{z(\gamma)}\right)}
    \end{equation*}
    for any choice of $(\alpha,\beta)$ that is minimal in the sense discussed in Remark \ref{remark: alpha and beta}. Let $b=ga$ and assume $(a,b)\in\mathcal{W}$. Thus, the smoothed elliptic Gamma is a product of $N$-smoothed $\Gamma$ functions and so we deduce:
    \begin{equation*}
        \widetilde{\Gamma}_{a,b}(w,z)\cdot \widetilde{\Gamma}_{a,b}(-w,z)=\frac{\theta_0^N\left(\frac{w}{z(\gamma)},\frac{z(-\alpha)}{z(\gamma)}\right)}{\theta_0^N\left(\frac{w}{z(\gamma)},\frac{z(\beta)}{z(\gamma)}\right)}.
    \end{equation*}
    Combining $g:\ker(a)\xrightarrow{\sim}\ker(b)$ with $z$ an eigenvector of $g$, yields $\frac{z(\beta)}{z(\gamma)}=\frac{z(g\beta)}{z(g\gamma)} = \delta\cdot \frac{z(-\alpha)}{z(\gamma)}$,
    for some $\delta=\big(\begin{smallmatrix}
        a&b\\c&d
    \end{smallmatrix}\big)\in \Gamma_0(N)$. In particular, $g\gamma = c(-\alpha)+d\gamma$ so that $cz(-\alpha)+dz(\gamma)=z(g\gamma)=g^{-1}z(\gamma)=\lambda^{-1}z(\gamma)$. By $\Gamma_0(N)$-invariance:
    \begin{equation*}
        \theta_0^N\left(\frac{w}{z(\gamma)},\frac{z(-\alpha)}{z(\gamma)}\right) = \theta_0^N\left(\frac{w}{cz(-\alpha)+dz(\gamma)},\frac{z(\beta)}{z(\gamma)}\right) =\theta_0^N\left(\lambda\frac{w}{z(\gamma)},\frac{z(\beta)}{z(\gamma)}\right).
    \end{equation*}
    And so:
    \begin{equation*}
        \widetilde{\Gamma}_{a,b}(w,z)\cdot \widetilde{\Gamma}_{a,b}(-w,z)=\frac{\theta_0^N\left(\lambda\frac{w}{z(\gamma)},\frac{z(\beta)}{z(\gamma)}\right) }{\theta_0^N\left(\frac{w}{z(\gamma)},\frac{z(\beta)}{z(\gamma)}\right)}.
    \end{equation*}
    Since $\theta_0^N(-,\tau)$ is a meromorphic function on $\mathbb{C}$ with a zero of order $N-1$ at zero, we know $\lim_{x\rightarrow 0}\frac{\theta_0^N(\lambda x,\tau)}{\theta_0^N(x,\tau)} = \lambda^{N-1}$ for all $\lambda\in \mathbb{C}^*$. In particular, $\widetilde{\Gamma}_{a,ga}(0,z)^2=\lambda^{N-1}$.
\end{proof}
\begin{proof}[Proof of Proposition \ref{prop: h in L}]
    By the Lemma above, it suffices to show $\widetilde{\Phi}\triple{h} = \widetilde{\Gamma}_{a,ga}(z(h),z)\overset{?}{=}\widetilde{\Gamma}_{a,ga}(0,z)$. From the description of $\Gamma_{a,b}(w,z)$ as a product of standard elliptic Gamma, it is clear that $\Gamma_{a,b}(w,z)$ is invariant under translating $w$ by $\mathbb{Z}\cdot z(\gamma)$. By construction, $\gamma$ is a (positive) primitive vector in the line $\langle h \rangle_{\mathbb{Q}}\cap L$. Thus, when $h\in L$, it follows that $h$ is an \textit{integer} multiple of $\gamma$, which completes the proof of the Proposition.
\end{proof}

\newpage 
\bibliographystyle{alpha}
\bibliography{refs}
\end{document}